\documentclass[11pt,a4paper]{article}
\newif\ifjournaledition
\journaleditionfalse
\usepackage[T1]{fontenc}
\usepackage[utf8]{inputenc}
\usepackage[margin=24mm,headheight=15pt]{geometry}
\usepackage{mathpazo,amsmath,amssymb,amsthm}
\usepackage{tikz}
\usetikzlibrary{arrows.meta,positioning,calc}
\usepackage[scaled=.92]{helvet}
\usepackage[expansion=false]{microtype}
\usepackage{xcolor,booktabs,longtable,tabularx,array,enumitem,fancyhdr,needspace}
\usepackage[unicode,colorlinks=true]{hyperref}
\definecolor{Ink}{HTML}{233D4D}
\definecolor{Geometry}{HTML}{176B65}
\definecolor{Algebra}{HTML}{885B16}
\definecolor{Analysis}{HTML}{4F5793}
\definecolor{Probability}{HTML}{8B4653}
\hypersetup{linkcolor=Geometry,urlcolor=Geometry,citecolor=Analysis}
\ifjournaledition
  \hypersetup{hidelinks,pdfsubject={Journal edition}}
  \colorlet{Ink}{black}
\else
  \hypersetup{pdfsubject={Reader edition}}
\fi
\newcommand{\headingcolor}[2]{\ifjournaledition#2\else\textcolor{#1}{#2}\fi}
\ifjournaledition
\else
\fi
\setlist[itemize]{leftmargin=16pt,itemsep=3pt,topsep=3pt}
\ifjournaledition
\fi
\newcommand{\readerlocator}[1]{\ifjournaledition\else#1\fi}
\newcommand{\ks}[2]{\href{https://arxiv.org/pdf/1502.06284v10\#page=#2}{Kalinin--Shkolnikov, #1\readerlocator{ (PDF p.~#2)}}~\cite{KS26}}
\newcommand{\mooney}[2]{\href{https://arxiv.org/pdf/1806.09945v1\#page=#2}{Mooney, #1\readerlocator{ (PDF p.~#2)}}~\cite{Moo18}}
\newcommand{\MA}{\operatorname{MA}}
\newcommand{\R}{\mathbb R}\newcommand{\Z}{\mathbb Z}\newcommand{\Q}{\mathbb Q}\newcommand{\C}{\mathbb C}
\newtheoremstyle{compact}{4pt}{3pt}{\itshape}{}{\bfseries}{.}{.5em}{}
\theoremstyle{compact}\newtheorem{theorem}{Theorem}[section]
\newtheorem{proposition}[theorem]{Proposition}
\newtheorem{lemma}[theorem]{Lemma}
\newtheorem{corollary}[theorem]{Corollary}
\theoremstyle{definition}\newtheorem{definition}[theorem]{Definition}
\newcolumntype{P}[1]{>{\raggedright\arraybackslash}p{#1}}
\newcolumntype{Y}{>{\raggedright\arraybackslash}X}
\hypersetup{pdftitle={Many-point tropical relaxation and the Monge--Ampere equation},pdfauthor={Kalinin, Lupercio, Serrano, Shkolnikov}}
\tikzset{box/.style={draw,rounded corners=3pt,fill=white,align=center,text width=3.55cm,minimum height=1.03cm,inner sep=5pt,font=\small,execute at begin node={\hyphenpenalty=10000\exhyphenpenalty=10000}},
 a/.style={box,draw=Algebra,fill=Algebra!5},g/.style={box,draw=Geometry,fill=Geometry!5},
 b/.style={box,draw=Analysis,fill=Analysis!5},p/.style={box,draw=Probability,fill=Probability!5},
 arrow/.style={-{Stealth[length=2mm]},draw=Ink!70,line width=.65pt,rounded corners=3pt}}

\begin{document}
\begin{center}
{\Large\bfseries Many-point tropical relaxation\\[3pt]and the Monge--Amp\`ere equation}\par
\vspace{9pt}
Nikita Kalinin\quad Ernesto Lupercio\\Higinio Serrano\quad Mikhail Shkolnikov\par
\vspace{8pt}1 October 2026
\end{center}

\begin{abstract}
In this paper we prove that a process we call tropical relaxation gives a convergent method for numerically approximating the Monge--Amp\`ere Dirichlet problem on bounded convex planar domains. We approximate the prescribed probability measure by point clouds and construct the least concave tropical roof with zero boundary values and a corner at each point, using repeated updates of affine planes with integral gradients. For universally generic clouds of $N$ points in a fixed compact subset of the interior, convergence of the empirical measures implies uniform convergence of the roofs, divided by $\sqrt N$, to the Aleksandrov solution. The limiting measure may be singular. Euler's and Pick's formulas relate curvature to point counts, with an $O(N^{-1/2})$ discrepancy against compactly supported $C^1$ tests. A removal estimate gives almost-sure uniform convergence for independent samples from any absolutely continuous probability law on the whole domain. For generic clouds in a fixed interior compact set of a rational-slope polygon, suitably rescaled sandpile odometers converge to the same solution along meshes chosen sufficiently fine for each cloud. Their deficit measures converge to its negative Laplacian.

\end{abstract}

\medskip\noindent\textbf{Contents}\par
\small
\noindent\begin{tabularx}{\textwidth}{@{}YY@{}}
1. The many-point limit (p.~\pageref{sec:1}) & A. Hahn series (p.~\pageref{sec:A})\\
2. Relaxation, competitors, and removal (p.~\pageref{sec:2}) & B. Supporting slopes (p.~\pageref{sec:B})\\
3. The marked tree (p.~\pageref{sec:3}) & C. Finite parameters (p.~\pageref{sec:C})\\
4. Counting curvature (p.~\pageref{sec:4}) & D. Sampling and sandpiles (p.~\pageref{sec:D})\\
5. The limiting roof (p.~\pageref{sec:5}) & \\
6. Random clouds (p.~\pageref{sec:6}) & List of symbols (p.~\pageref{symbols:list})\\
\ifjournaledition\else Guide to the proof (p.~\pageref{guide:proof}) & \\
\fi
\end{tabularx}\normalsize

\ifjournaledition\else
\medskip\textbf{Color key.}\quad
\textcolor{Algebra}{Algebra and competitors}\quad
\textcolor{Geometry}{Tropical geometry}\par
\textcolor{Analysis}{Convex analysis and limits}\quad
\textcolor{Probability}{Sampling and sandpiles}
\fi

\needspace{12\baselineskip}\section{\texorpdfstring{\headingcolor{Ink}{The many-point limit}}{The many-point limit}}\label{sec:1}

In this paper we will prove that a process we call tropical relaxation gives a numerical approximation of the Monge--Amp\`ere Dirichlet problem in the plane. Let $\rho$ be a probability density on a bounded open convex domain $\Omega$, and consider
\[
 \det(-D^2F)=\rho\quad\text{in }\Omega,
 \qquad F|_{\partial\Omega}=0,
\]
with $F$ concave and the equation understood in the sense of Aleksandrov. Choose an infinite sequence of independent points with density $\rho$. From its first $N$ points, form the least tropical roof with zero boundary values and divide its height by $\sqrt N$. These roofs converge uniformly on $\overline\Omega$, almost surely, to the unique solution $F$ (Theorem~\ref{thm:iid-boundary}).

A tropical roof is a continuous concave function on $\overline\Omega$ that vanishes on $\partial\Omega$. In the interior it is locally the minimum of finitely many affine functions with integral gradients. For a finite set $P\subset\Omega$, there is a least such roof having a corner at every point of $P$; we denote it by $F_{\Omega,P}$. A corner means a point of nondifferentiability and may lie in the interior of an edge.

The relaxation procedure of Kalinin and Shkolnikov~\cite{KS18,Kal18} constructs $F_{\Omega,P}$ from zero. An affine plane is active at a point when its value equals the roof there. At a mark where the current roof is affine, raise its active plane until a plane of a different gradient attains the same value at the mark. Repeatedly visiting every mark gives an increasing sequence converging to $F_{\Omega,P}$, independently of the order. Section~\ref{sec:2} gives the coefficient rule and an example on a square.

For a smooth concave function $F$, the Monge--Amp\`ere measure $\MA(F)$ has density $\det(-D^2F)$. Aleksandrov's definition extends this measure to arbitrary concave functions by taking the area of their supporting slopes. It applies both to the tropical roofs and to their limits, including solutions whose curvature measure has atoms.

Convergence also holds for deterministic generic clouds, whose limiting measure may be singular. Let $P_N$ contain $N$ points, with empirical measures
\[
 \mu_N=\frac1N\sum_{p\in P_N}\delta_p\rightharpoonup\mu.
\]
For generic configurations contained in a fixed compact set $K\Subset\Omega$, we prove
\[
 \frac1{\sqrt N}F_{\Omega,P_N}\longrightarrow F_{\mu,\Omega}
 \qquad\text{uniformly on }\overline\Omega.
\]
The limit is the unique continuous concave solution of
\[
 \MA(F_{\mu,\Omega})=\mu\quad\text{in }\Omega,
 \qquad F_{\mu,\Omega}|_{\partial\Omega}=0.
\]
Theorem~\ref{thm:many-point} gives the precise genericity hypothesis and a quantitative estimate for the curvature measures.

The convergence theorem relates two measures supported at different points. In a generic configuration on a rational-slope polygon, the sources lie inside edges of the tropical curve, while curvature is concentrated at its vertices. The mass at a vertex is the area of the polygon spanned by the adjacent gradients. Figure~\ref{fig:roof-slopes} shows these supporting slopes for a roof on a square.

\ifjournaledition
\begin{figure}[hbp]
\else
\begin{figure}[tbp]
\fi\centering
$F(x,y)=\min\{1-x,1+x,1-y,1+y\},\qquad (x,y)\in[-1,1]^2.$
\medskip

\begin{tikzpicture}[x=1cm,y=1cm,font=\small]
\begin{scope}[shift={(0,0)},x=2.15cm,y=2.15cm]
 \fill[Geometry!6] (0,0)--(1,-1)--(1,1)--cycle;
 \fill[Geometry!11] (0,0)--(1,1)--(-1,1)--cycle;
 \fill[Geometry!4] (0,0)--(-1,1)--(-1,-1)--cycle;
 \fill[Geometry!8] (0,0)--(-1,-1)--(1,-1)--cycle;
 \draw[Ink!65] (-1,-1) rectangle (1,1);
 \draw[Geometry,line width=.9pt] (-1,-1)--(1,1) (-1,1)--(1,-1);
 \draw[Probability,line width=1.7pt] (0,0)--(1,1);
 \node at (-.65,0) {$1+x$};
 \node at (0,.72) {$1-y$};
 \node at (0,-.72) {$1+y$};
 \node at (.63,-.27) {$1-x$};
 \fill[Geometry] (.75,0) circle (2.8pt);
 \node[above=3pt,Geometry] at (.75,0) {$u$};
 \draw[Probability,fill=white,line width=1.1pt] (.5,.5) circle (3pt);
 \node[below right=2pt,Probability] at (.5,.5) {$r$};
 \fill[Analysis] (0,0) circle (3.2pt);
 \node[left=5pt,Analysis] at (0,0) {$v$};
 \node[above,font=\small\bfseries] at (0,1.43) {The affine cells in the domain};
 \node[below] at (0,-1.1) {$u=(3/4,0),\quad r=(1/2,1/2),\quad v=(0,0)$};
\end{scope}
\begin{scope}[shift={(7.8,0)},x=2.15cm,y=2.15cm]
 \draw[->,Ink!60] (-1.18,0)--(1.26,0) node[right] {$m_1$};
 \draw[->,Ink!60] (0,-1.14)--(0,1.23) node[above] {$m_2$};
 \filldraw[fill=Analysis!12,draw=Analysis,line width=.9pt] (1,0)--(0,1)--(-1,0)--(0,-1)--cycle;
 \draw[Probability,line width=1.7pt] (-1,0)--(0,-1);
 \fill[Geometry] (-1,0) circle (2.8pt);
 \foreach \x/\y in {1/0,0/1,0/-1}\fill[Analysis] (\x,\y) circle (2pt);
 \node[above left] at (-1,0) {$-e_1$};
 \node[above right] at (1,0) {$e_1$};
 \node[above left] at (0,1) {$e_2$};
 \node[below left] at (0,-1) {$-e_2$};
 \node[Analysis,align=center] at (.16,.07) {$\partial^+F(v)$\\area $2$};
 \node[Probability,rotate=-45,below] at (-.55,-.55) {$\partial^+F(r)$};
 \node[Geometry,left] at (-1.03,-.2) {$\partial^+F(u)$};
 \node[above,font=\small\bfseries] at (0,1.43) {The supporting slopes};
\end{scope}
\end{tikzpicture}
\medskip

\begin{tabular}{@{}lll@{}}\toprule
Point in the domain & Supporting slopes & Euclidean area\\\midrule
\textcolor{Geometry}{Face point $u$} & $\{-e_1\}$ & $0$\\
\textcolor{Probability}{Edge point $r$} & $\operatorname{conv}\{-e_1,-e_2\}$ & $0$\\
\textcolor{Analysis}{Vertex $v$} & $\operatorname{conv}\{\pm e_1,\pm e_2\}$ & $2$\\\bottomrule
\end{tabular}
\medskip

$\displaystyle \MA(F)=2\delta_v.$
\caption{Supporting slopes of the displayed tropical roof. At the face point $u$ they form a singleton; at the edge point $r$, a segment of area zero; and at $v$, a diamond of area $2$. Hence $\MA(F)=2\delta_v$. The diamond has one interior lattice point, so the vertex surplus of Definition~\ref{def:window} is $1$. This roof illustrates curvature; it is not a strongly generic least roof.}
\label{fig:roof-slopes}
\end{figure}

For uniform sampling in the unit disk $D$, the limiting roof is
\[
 F_{\mathrm{unif},D}(x)=\frac{1-|x|^2}{2\sqrt\pi}.
\]
\hyperref[app:diagnostics]{Appendix D4} compares this quadratic profile with roofs computed on rational polygons enclosing the disk. For three clouds at each size, the mean relative $L^2$ errors on $\{|x|\le0.9\}$, estimated on a grid, are approximately $11.5\%$, $5.5\%$, and $3.4\%$ at $N=25,100,400$, respectively.

\subsection*{The convergence theorem}
The Monge--Amp\`ere measure scales quadratically with height: $\MA(aF)=a^2\MA(F)$ for $a>0$. Division of the roof by $\sqrt N$ therefore divides its curvature by $N$, the normalization used in $\mu_N$. On each affine cell, the normalized roof has gradient in $N^{-1/2}\Z^2$.

For the quantitative estimate we suppose that the sources lie in a fixed $K\Subset\Omega$. On a polygon with rational side slopes, strong genericity requires the marks to lie in distinct edge interiors. Small independent movements of the marks must also preserve the incidences of the curve, its cell gradients, and the edges containing the marks. Definition~\ref{def:strong-genericity} states this condition using the marked embedded type. Strongly generic configurations form an open set of full measure. For a general domain we impose the corresponding conditions on the countable family of rational-coordinate polygons in Definition~\ref{def:universal}. We refer to this as universal genericity.

\needspace{5\baselineskip}\begin{theorem}[The many-point limit]\label{thm:many-point}

Let $\Omega\subset\R^2$ be a bounded open convex domain with nonempty interior, and fix $K\Subset\Omega$. Let $P_N\subset K$ be universally generic sets of $N$ distinct points. Put
\[
 u_N=N^{-1/2}F_{\Omega,P_N},\qquad
 \mu_N=N^{-1}\sum_{p\in P_N}\delta_p.
\]
For every compact $L\Subset\Omega$ there is $C=C(\Omega,K,L)$ such that, for every $N$ and every $\varphi\in C_c^1(\Omega)$ supported in $L$,
\begin{equation}\label{eq:main-discrepancy}
 \left|\int\varphi\,d\bigl(\MA(u_N)-\mu_N\bigr)\right|
 \le \frac{C}{\sqrt N}\bigl(\|\varphi\|_\infty+\|\nabla\varphi\|_\infty\bigr).
\end{equation}
If $\mu_N\rightharpoonup\mu$, then $u_N\to F_{\mu,\Omega}$ uniformly on $\overline\Omega$, where $F_{\mu,\Omega}$ is the unique continuous concave solution of
\[
 \MA(F_{\mu,\Omega})=\mu\quad\hbox{in }\Omega,
 \qquad F_{\mu,\Omega}|_{\partial\Omega}=0.
\]
Moreover, $\MA(u_N)\rightharpoonup\mu$ vaguely in $\Omega$. The limiting probability measure may be singular.
\end{theorem}
\needspace{7\baselineskip}\begin{proof}
Choose one nested exhaustion $(\Delta_j)$ of $\Omega$ by polygons with rational vertex coordinates, with $K\Subset\Delta_1^\circ$; Lemma~\ref{lem:cofinality} supplies such an exhaustion. Every cloud $P_N$ lies inside every polygon of this exhaustion. Universal genericity therefore makes each $P_N$ strongly generic on each of these polygons, by Definition~\ref{def:universal}. The same exhaustion works for all $N$. Theorem~\ref{thm:fixed-exhaustion} gives the discrepancy estimate under this hypothesis and gives uniform convergence of the roofs and vague convergence of their curvature measures when $\mu_N\rightharpoonup\mu$.
\end{proof}

Every compactly supported probability measure admits universally generic empirical approximations in a fixed compact neighborhood of its support. One first chooses empirical approximations and then perturbs their points into the dense generic locus. Appendix~\ref{sec:C}, C5, gives the construction.

\paragraph*{Numerical approximation.}\label{intro:relaxation-approximation}
For a compactly supported finite positive measure $\nu$ of mass $m>0$, choose universally generic clouds in a fixed compact neighborhood of its support, with empirical measures tending to $\nu/m$. Multiply their least roofs by $\sqrt{m/N}$. Theorem~\ref{thm:many-point} and homogeneity of the Monge--Amp\`ere measure then give uniform convergence to the solution with measure $\nu$.

There are two approximations here. Increasing the number of sources approximates the prescribed measure; repeated visits to a fixed cloud approximate its least roof. For each cloud $P_N$, start from zero and let $f_{N,k}$ be the roof after $k$ visits, with every mark revisited infinitely often. Thus $N$ counts sources and $k$ counts updates. As proved in Appendix~\ref{sec:C}, $f_{N,k}\to F_{\Omega,P_N}$ uniformly on $\overline\Omega$ as $k\to\infty$. Given positive $\varepsilon_N\to0$, choose $k_N$, depending on the cloud, so that $\|f_{N,k_N}-F_{\Omega,P_N}\|_\infty\le\varepsilon_N\sqrt N$. Then
\[
 \left\|N^{-1/2}f_{N,k_N}-F_{\mu,\Omega}\right\|_\infty
 \le\varepsilon_N+
 \left\|N^{-1/2}F_{\Omega,P_N}-F_{\mu,\Omega}\right\|_\infty.
\]
The right-hand side tends to zero by Theorem~\ref{thm:many-point}, so the normalized finite iterates converge uniformly to the solution.

For each compactly supported $C^1$ test, the difference between $\MA(u_N)$ and $\mu$ is bounded by two errors. The estimate~\eqref{eq:main-discrepancy} controls the difference between $\MA(u_N)$ and $\mu_N$; empirical convergence controls the difference between $\mu_N$ and $\mu$.

On a fixed polygon with rational side slopes, strong genericity suffices, and the curvature measures converge weakly on the closed polygon (Theorem~\ref{thm:polygon-limit}). The total-mass identity of Theorem~\ref{thm:total-mass} gives the boundary control needed for this conclusion.

\noindent\emph{Example (the genericity hypothesis).}
On $\Omega=(0,1)^2$, the tropical roof
\[
 R(x,y)=\min\{x,1-x,y,1-y\}
\]
has a corner at every point of the segment $S=\{(t,t):1/4\le t\le1/3\}$. Let $P_N$ consist of $N$ equally spaced distinct points on $S$. Their empirical measures converge to normalized arclength on $S$. Since $R$ is admissible for every $P_N$, minimality gives
\[
 0\le N^{-1/2}F_{\Omega,P_N}\le N^{-1/2}R\longrightarrow0
 \qquad\hbox{uniformly on }\overline\Omega.
\]
The zero roof has zero Monge--Amp\`ere measure, whereas the limiting source measure has mass one. Thus the many-point theorem cannot hold for all distinct configurations. Here the additional marks impose corner conditions already satisfied by one fixed roof.

\subsection*{Why the roofs converge}\label{intro:mechanism}
The proof has two tasks. We first obtain compactness for the roofs divided by $\sqrt N$. We then show that the curvature measure of every limit is the prescribed measure $\mu$. Uniqueness of the Dirichlet solution will identify all these limits.

To bound the height of $F_{\Omega,P}$, we construct one admissible roof of controlled height; the least roof lies below it. Write $P=\{p_1,\ldots,p_N\}$ and choose $M>N$ distinct integral gradients $m_1,\ldots,m_M$. The interpolation theorem in Section~\ref{sec:2} provides a nonempty set $J\subset\{1,\ldots,M\}$ and real numbers $c_j$ such that
\[
 q(x)=\min_{j\in J}\bigl(c_j+\langle m_j,x\rangle\bigr)
\]
has a corner at every mark. Here $c_j$ sets the vertical position of the plane with gradient $m_j$. The index set $J$ records the planes retained by the construction.

\ifjournaledition
\else
To find these constants, the interpolation argument solves $N$ homogeneous linear equations over a field of formal series, with unknowns $z_1,\ldots,z_M$. There is a nonzero solution because $M>N$. Retain the indices for which $z_j\ne0$, and let $c_j$ be the least exponent of $z_j$. With $t$ a formal symbol, the equation at $p_i$ is $\sum_j t^{\langle m_j,p_i\rangle}z_j=0$. Its nonzero summands have least exponents $c_j+\langle m_j,p_i\rangle$. Their sum vanishes, so the smallest exponent must occur at least twice. Thus two distinct affine terms attain the minimum defining $q$ at $p_i$, giving the required corner. Appendix~\ref{sec:A} proves the interpolation statement in this way.
\fi

Choose the gradients in the square $[-r,r]^2$, with $r$ a positive integer. It contains $(2r+1)^2$ integral points, so a radius $r$ of order $\sqrt N$ supplies more than $N$ gradients. Each has norm at most $\sqrt2\,r$, and hence $q$ is Lipschitz with this constant. Subtracting its minimum on $\overline\Omega$ preserves the corners and gives
\[
 0\le\widehat q:=q-\min_{\overline\Omega}q
 \le\sqrt2\,r\operatorname{diam}(\Omega)=O_\Omega(\sqrt N).
\]
To impose zero boundary values, use the lattice support distance $D_\Omega$ defined in Section~\ref{sec:2}. This tropical function is positive in $\Omega$ and vanishes continuously on $\partial\Omega$. Since there are finitely many marks, an integer $\ell_{\rm cap}$ can be chosen so that $\ell_{\rm cap}D_\Omega(p)>\widehat q(p)$ at each one. The strict inequalities persist in neighborhoods of the marks. Consequently $E=\min\{\widehat q,\ell_{\rm cap}D_\Omega\}$ agrees with $\widehat q$ there and is an admissible roof with zero boundary values. Minimality gives
\[
 0\le F_{\Omega,P}\le E\le\widehat q\le C_\Omega\sqrt N.
\]
On a fixed $K\Subset\Omega$, the function $D_\Omega$ has a positive minimum. Thus $\ell_{\rm cap}$ can be chosen of order $\sqrt N$, giving a common boundary barrier for the normalized roofs. Concavity controls their slopes on interior compact sets. Together, these estimates give compactness on $\overline\Omega$.

We next compare the source count with curvature, first for a strongly generic cloud on a rational-slope polygon. Consider the regions where the roof is affine whose closures lie in the interior. Their supporting planes can be lowered slightly while preserving nonnegativity and the boundary values. At a generic mark, lowering both adjacent planes equally preserves the corner; leaving both fixed also preserves it. Minimality rules out any nonempty choice of bounded cells that permits such a lowering. Section~\ref{sec:3} expresses this restriction through incidence vectors over $\mathbf F_2$ and proves that there are exactly $N$ bounded cells. The dual edges crossing the $N$ marked carriers (the edges containing the marks) form a spanning tree. It follows that cutting the tropical curve at all its marks gives a tree.

The tree makes Euler's formula applicable to source counts. Pick's formula expresses the area of each gradient polygon in terms of its lattice points. Combining the two formulas gives the curvature in an interior region as its source count plus a boundary-crossing correction and a nonnegative lattice correction. The latter comes from interior lattice points in gradient polygons and from edge weights greater than one. Section~\ref{sec:4} defines these terms and proves the exact identity. Lowering suitable supporting planes, together with adjacent planes when a mark must be preserved, excludes the interior contributions to the lattice correction. The remaining contributions are bounded by the total weight of edges leading to the polygon boundary.

Apply this identity to the superlevel sets of a test function and integrate over the level. Coarea~\cite{FR60,Fed69} turns the crossing count into an integral of the absolute tangential derivative along the curve. The length and lattice-correction bounds are $O(\sqrt N)$ on a fixed polygon, so division by $N$ gives the discrepancy estimate of Theorem~\ref{thm:many-point}.

For a general convex domain, we choose an increasing sequence of polygons with rational vertex coordinates. Bounds on slopes and crossings in fixed interior rectangles make the local discrepancy estimates uniform along this exhaustion. Section~\ref{sec:5} passes these estimates to the domain and identifies the curvature of each subsequential limit by Aleksandrov stability. Comparison then gives uniqueness and convergence of the entire sequence. Appendix~\ref{sec:B} recalls the supporting-plane arguments used in this last step.

\subsection*{Sampling up to the boundary}\label{intro:sampling}
Subadditivity allows us to pass from compactly supported sampling to sampling throughout the domain. The sum of two concave tropical roofs retains every corner of either summand. Minimality gives
\[
 F_{\Omega,P}\le F_{\Omega,P\cup Q}
 \le F_{\Omega,P}+F_{\Omega,Q},
\]
and the unrestricted height bound then shows that deleting $r$ points changes an $N^{-1/2}$-normalized roof by at most $C_\Omega\sqrt{r/N}$. The constant is independent of the distance of the deleted points from the boundary.

Let the samples be independent with an absolutely continuous probability law $\mu$. Choose compact sets $K_j$ whose interiors exhaust $\Omega$. The samples falling in $K_j$, taken in their original order, have the conditional law on $K_j$; their empirical measures converge by Varadarajan\textquotesingle s theorem~\cite{Var58}, and the compact-support theorem applies. As $j$ increases, the proportion discarded tends to zero. The removal estimate gives almost-sure uniform convergence on $\overline\Omega$ and vague curvature convergence in $\Omega$ (Theorem~\ref{thm:iid-boundary}). Appendix~\ref{sec:D} gives the conditional sampling argument; the proof uses a single probability-one event for the entire exhaustion.

\subsection*{Affine covariance}\label{intro:consequences}
Affine covariance follows directly from the geometric definition of the measure. Under a change of variables $T(x)=Ax+b$, supporting slopes transform by $A^{-T}$; computing the change of their area and using uniqueness gives
\[
 F_{T_\#\mu,T\Omega}
 =|\det A|^{1/2}F_{\mu,\Omega}\circ T^{-1}.
\]
The uniform ellipse solutions follow from the disk formula. Integral unimodular maps preserve the allowable slopes of the finite roofs, while the limiting equation is covariant under every invertible affine map, as proved in Theorem~\ref{thm:covariance} and Proposition~\ref{prop:lattice-covariance}.

\subsection*{Sandpiles}
The finite roofs also arise as limits of sandpile odometers. On the square lattice of spacing $h>0$ in a rational-slope polygon, place three grains at every site and add one at each rounded source position. An unstable site sends a grain to each of its four neighbors, and the odometer $H_{h,P}$ counts its topplings. For a fixed cloud, the sandpile approximation of Kalinin and Shkolnikov~\cite{KS26}, with the rounding and boundary conventions established in Appendix~\ref{app:diagonal}, gives $hH_{h,P}\to F_{\Omega,P}$ as $h\to0$.

Let $P_N$ be strongly generic clouds in a fixed $K\Subset\Omega$, where $\Omega$ is the interior of a rational-slope polygon, and suppose their empirical measures converge to $\mu$. Choose the meshes as in Theorem~\ref{thm:sandpile}. The many-point theorem then gives
\[
 \frac{h_N}{\sqrt N}H_{h_N,P_N}\longrightarrow F_{\mu,\Omega}
 \quad\hbox{uniformly on the lattice sites},
 \qquad Nh_N^2\longrightarrow0.
\]
By grain conservation, the corresponding deficit measures converge vaguely in $\Omega$ to $-\Delta F_{\mu,\Omega}$ (Theorem~\ref{thm:deficit}). Thus one limiting roof carries two measures: its Monge--Amp\`ere measure records the distribution of prescribed corners, and its negative Laplacian records the limiting sandpile deficit. When the roof is smooth, their densities are the determinant and the trace of $-D^2F$, respectively.

\subsection*{Earlier work}\label{intro:context}
The sandpile originates in the work of Bak, Tang, and Wiesenfeld and Dhar~\cite{BTW87,Dha90}; its tropical geometry and discrete patterns are developed in~\cite{CPS10}. For a different sandpile scaling limit, see Pegden--Smart~\cite{PS13}. General accounts of tropical curves and their Newton duality are~\cite{RST05,Mik05,MS15}; see also~\cite{PR04} for the relation with Monge--Amp\`ere measures.
Kalinin and Shkolnikov introduced the relaxation operators and developed the \href{https://arxiv.org/abs/1706.03062v1}{theory of tropical series}~\cite{KS18}. Their one-point operations give the dynamics of the tropical sandpile described in~\cite{Kal18}. We study the normalized roofs when the number of prescribed points tends to infinity. Kalinin and Prieto proved the generic bounded-face count and the tree statement after removing marks in \href{https://arxiv.org/pdf/1906.02802v3\#page=6}{Theorem 3.3 and its proof\readerlocator{ (PDF p.~6)}}~\cite{KP23}, and studied square-root degree growth experimentally. We prove the finite geometric statements under the strong genericity hypothesis of Theorem~\ref{thm:generic-locus}. The marked dual tree then determines the simultaneous coefficient lowerings used to bound the curvature discrepancy. Related Monge--Amp\`ere problems arise in \href{https://arxiv.org/abs/1803.00785v2}{Berman's discretization of optimal transport}~\cite{Ber21} and in the \href{https://arxiv.org/abs/2208.13697}{Calabi--Yau degenerations studied by Hultgren, Jonsson, Mazzon, and McCleerey}~\cite{Hul24}.

The tropical-series facts used from~\cite{KS18} are recalled where they enter: the lattice support distance in Section~\ref{sec:2}, and the one-point coefficient rule and convergence of repeated updates in Sections~\ref{sec:2} and~\ref{sec:C}. We give proofs of relaxation minimality, the boundary flux identity, and the bound on essential slopes. The sandpile argument uses the lower-bound construction of~\cite[Proposition~5.4]{KS26}, including its smoothing and wave lemmas. Proposition~\ref{prop:ks-wave-input} states the results used; the rest of Appendix~\ref{app:diagonal} proves the rounding adaptation, the upper bound, and the comparison of boundary conventions.

\subsection*{Notation and conventions}\label{sec:conventions}
Throughout, $\Omega$ is bounded, open, and convex, with nonempty interior. A rational polygon $\Delta$ has rational side slopes; rational vertex coordinates will be required explicitly. Sources are distinct interior points. We write $N=|P|$, $\nu_P=\sum_{p\in P}\delta_p$, and $\mu_P=N^{-1}\nu_P$ for $N\ge1$. The empty roof is zero. When the domain is fixed, we abbreviate $F_{\Omega,P}$ to $F_P$ and write $\Gamma_P=\Gamma(F_P)$ for its corner locus.

The concave superdifferential is
\[
 \partial^+F(x)=\{p\in\R^2:F(y)\le F(x)+p\cdot(y-x)\text{ for all }y\in\Omega\}.
\]
An element of $\partial^+F(x)$ is a supporting slope at $x$. The union of these slope sets over a Borel set $E\subset\Omega$ has ordinary Euclidean area $\MA(F)(E)$; a unimodular triangle has mass $1/2$. Vague convergence is tested against $C_c(\Omega)$, and weak convergence on $\overline\Omega$ against $C(\overline\Omega)$, extending interior measures by zero to the boundary. Roof convergence is uniform on $\overline\Omega$.

\ifjournaledition\else
The guide to the proof begins on p.~\pageref{guide:proof}; the list of symbols begins on p.~\pageref{symbols:list}.
\fi

\ifjournaledition\else
\clearpage
\section*{Guide to the proof}\label{guide:proof}
\pdfbookmark[1]{Guide to the proof}{guide-proof}
The diagrams locate the principal results used in the proofs. Section~\ref{sec:2} begins with the one-point update and its square example; Appendix~\ref{sec:C}, p.~\pageref{app:wave-convergence}, proves convergence of repeated updates. For the combinatorial argument, begin with the bounded-cell count in Section~\ref{sec:3}; for the supporting-slope definition of curvature, see Appendix~\ref{sec:B}.

\subsection*{The deterministic argument}
The diagrams show the principal dependencies: an arrow means that the later proof uses the preceding result. The sources lie in a fixed compact subset of the domain. The limit statements also assume convergence of their empirical measures.

\begin{center}
\begin{tikzpicture}[x=1cm,y=1cm,box/.append style={font=\small}]
\node[a] (start) at (0,0) {Interpolation and minimality\\Proposition~\ref{prop:relaxation};\\Theorem~\ref{thm:interpolation}};
\node[g] (tree) at (5.2,0) {Generic cell count\\and marked cut tree\\Theorems~\ref{thm:genus}, \ref{thm:marked-tree}};
\node[g] (balance) at (10.4,0) {Euler--Pick window identity\\Theorem~\ref{thm:window-balance}};
\node[a] (bounds) at (0,-2.1) {Height, length, and\\side-quasidegree bounds\\\ref{prop:barriers}; \ref{thm:complexity}};
\node[g] (lowering) at (5.2,-2.1) {Lowering excludes\\interior lattice surplus\\\ref{lem:lowering}--\ref{thm:mildness}};
\node[b] (coarea) at (10.4,-2.1) {Integrated window inequality\\via coarea\\Theorem~\ref{thm:coarea}};
\node[g] (total) at (5.2,-4.2) {Terminal surplus; total mass\\$N-1+\tfrac12D_{\rm term}$\\\ref{prop:terminal}; \ref{thm:total-mass}};
\draw[arrow] (start)--(bounds);
\draw[arrow] (start)--(tree);
\draw[arrow] (bounds.north east)--(tree.south west);
\draw[arrow] (tree)--(balance);
\draw[arrow] (tree)--(lowering);
\draw[arrow] (balance)--(coarea);
\draw[arrow] (balance.south west)--(7.9,-1.15)--(7.9,-4.2)--(total.east);
\draw[arrow] (lowering)--(total);
\draw[arrow] (bounds.south east)--(total.north west);
\end{tikzpicture}
\end{center}

\noindent\textit{From estimates to limits.} Inputs from the first diagram are repeated below.

\begin{center}
\begin{tikzpicture}[x=1cm,y=1cm]
\node[a] (bounds) at (0,0) {Height and boundary bounds\\Proposition~\ref{prop:barriers}};
\node[g] (lowering) at (5.2,0) {Interior lattice simplicity\\Theorem~\ref{thm:mildness}};
\node[b] (coarea) at (10.4,0) {Coarea inequality\\Theorem~\ref{thm:coarea}};
\node[b] (analytic) at (0,-2.1) {Compactness, stability,\\and comparison\\\ref{prop:compactness}; \ref{thm:stability}; \ref{thm:comparison}};
\node[b] (local) at (5.2,-2.1) {Local length, crossings,\\and surplus bounds\\\ref{lem:crofton}; \ref{prop:local-excess}};
\node[b] (general) at (10.4,-2.1) {Discrepancy along a fixed\\polygonal exhaustion\\\ref{prop:exhaustion}; \ref{thm:general-discrepancy}};
\node[b] (polygon) at (0,-4.4) {Uniform roof convergence;\\weak curvature convergence\\on the closed polygon\\Theorem~\ref{thm:polygon-limit}};
\node[b] (main) at (5.2,-4.4) {Uniform roof convergence\\on the closed convex domain\\Theorems~\ref{thm:fixed-exhaustion}, \ref{thm:many-point}};
\node[b] (fixed) at (10.4,-4.4) {Fixed-polygon discrepancy\\and total mass\\Theorems~\ref{thm:polygon-discrepancy}, \ref{thm:total-mass}};
\draw[arrow] (bounds)--(analytic);
\draw[arrow] (bounds.south east)--(local.north west);
\draw[arrow] (lowering)--(local);
\draw[arrow] (local)--(general);
\draw[arrow] (coarea)--(general);
\draw[arrow] (general.south west)--(main.north east);
\draw[arrow] (analytic.south east)--(main.north west);
\draw[arrow] (analytic)--(polygon);
\draw[arrow] (coarea.east)--++(.5,0)|-(fixed.east);
\draw[arrow] (fixed.south)--++(0,-.95)-|(polygon.south);
\end{tikzpicture}
\end{center}

The fixed-polygon estimates also use the global length and terminal-surplus bounds in the first diagram. Passing through the exhaustion uses stability. Strong genericity is Definition~\ref{def:strong-genericity}; genericity on a fixed exhaustion is Definition~\ref{def:universal}. Corollary~\ref{prop:holder} gives a polygonal modulus from the total-mass bound.

\clearpage
\section*{Guide to the proof: sampling and sandpiles}\label{guide:consequences}
\pdfbookmark[1]{Guide to the proof: sampling and sandpiles}{guide-consequences}
Theorem~\ref{thm:iid-compact} treats compactly supported sampling, and Theorem~\ref{thm:iid-boundary} treats an absolutely continuous law on the whole domain. Their probability argument is in \hyperref[app:accepted-stream]{Appendix D1}. For sandpiles, \hyperref[app:diagonal]{Appendix D2} gives the fixed-cloud approximation; Theorem~\ref{thm:sandpile} then chooses a mesh for each cloud.

\begin{center}
\begin{tikzpicture}[x=1cm,y=1cm]
\node[a] (remove) at (0,0) {Removing $M$ sources changes\\heights by $O(\sqrt M)$\\Lemma \ref{lem:subadditivity};\\Corollary \ref{cor:removal}};
\node[b] (many) at (5.2,0) {Normalized roofs converge\\on convex domains\\Theorem \ref{thm:many-point}};
\node[b] (poly) at (10.4,0) {Normalized roofs converge\\on closed polygons\\Theorem \ref{thm:polygon-limit}};
\node[p] (accepted) at (0,-2.2) {Retained samples have\\the conditional law\\Lemma \ref{lem:accepted-stream}};
\node[p] (iid) at (5.2,-2.2) {Compactly supported samples\\give uniform convergence\\Theorem \ref{thm:iid-compact}};
\node[p,dashed] (fixed) at (10.4,-2.2) {Scaled odometers $hH$\\approach the fixed-cloud roof\\Appendix~\ref{sec:D}, D2};
\node[p] (boundary) at (5.2,-4.7) {Sampling on all of $\Omega$\\gives uniform convergence\\Theorem \ref{thm:iid-boundary}};
\node[p] (sandpile) at (10.4,-4.7) {Normalized odometers converge\\on chosen meshes\\Theorem \ref{thm:sandpile}};
\node[p] (disk) at (5.2,-7) {Uniform disk and ellipse\\limits are quadratic\\Corollary \ref{cor:disk-ellipse}};
\node[p] (deficit) at (10.4,-7) {Normalized deficit\\tends to $-\Delta F$\\Theorem \ref{thm:deficit};\\Appendix~\ref{sec:D}, D3};
\draw[arrow] (many)--(iid);
\draw[arrow] (iid)--(boundary);
\draw[arrow] (accepted.east)--++(.45,0)|-(boundary.west);
\draw[arrow] (remove.west)--++(-.3,0)|-($(boundary.south west)+(0,-.25)$)--(boundary.south west);
\draw[arrow] (boundary)--(disk);
\draw[arrow] (poly.east)--++(.35,0)|-(sandpile.east);
\draw[arrow] (fixed)--(sandpile);
\draw[arrow] (sandpile)--(deficit);
\end{tikzpicture}
\end{center}

The dashed box uses the fixed-cloud lower-bound construction of Kalinin--Shkolnikov~\cite{KS26}. Appendix~\ref{sec:D}, D2, proves the rounding adaptation and the upper bound, and compares the two sink conventions.

The proof of Theorem~\ref{thm:iid-boundary} also uses stability and comparison. For the deficit identity, see \hyperref[app:deficit]{Appendix D3}; its limiting form uses the meshes chosen in Theorem~\ref{thm:sandpile}.

The affine transformation laws are proved in Theorem~\ref{thm:covariance} and Proposition~\ref{prop:lattice-covariance}. Corollary~\ref{cor:disk-ellipse} includes the direct Hessian calculation for the ellipse.
\clearpage

\fi

\needspace{12\baselineskip}\section{\texorpdfstring{\headingcolor{Algebra}{Relaxation, competitors, and removal}}{Relaxation, competitors, and removal}}\label{sec:2}

We construct the least roof by repeated one-point relaxation and bound its height in terms of the number of sources. The bound is
\[
 0\le F_{\Omega,P}\le C_\Omega\sqrt{|P|}
\]
for every finite $P\subset\Omega$, even when the sources approach the boundary. Since the least roof lies below every admissible competitor, it suffices to construct one roof with the prescribed corners and controlled height.

For $N=|P|$, interpolation supplies the corners using slopes in a square of radius $O(\sqrt N)$. A homogeneous linear system over a Hahn field gives the coefficients; Appendix~\ref{sec:A} proves this step. Subtracting the interpolant's minimum gives a nonnegative function of height $O(\sqrt N)$. Taking its minimum with a sufficiently large integer multiple of the tropical distance imposes zero boundary values. The multiple is chosen so that the interpolant remains strictly lower near every mark, preserving the corners.

For sources in a fixed interior compact set of a rational-slope polygon, the same construction bounds the side quasidegrees and curve length. Adding two roofs gives an admissible roof for the union of their source sets. This subadditivity and the height estimate yield the removal bound used in Section~\ref{sec:6}.

\needspace{5\baselineskip}\begin{definition}[Roofs and relaxation]\label{def:tropical}

An $\Omega$-tropical series is a continuous nonnegative concave function on $\overline\Omega$, zero on $\partial\Omega$, locally expressible in $\Omega$ as a minimum of finitely many affine functions with gradients in $\Z^2$. Its corner locus $\Gamma(F)$ is the set where it is not differentiable. A finite global presentation is called a tropical polynomial. For a finite source set $P\subset\Omega$, write $F_{\Omega,P}=G_P0_\Omega$ for the least such roof whose corner locus contains $P$, and set $F_{\Omega,\varnothing}=0$. More generally, $G_Pf$ denotes the pointwise least $\Omega$-tropical roof above $f$ whose corner locus contains $P$. For a single point we write $G_pf=G_{\{p\}}f$; if $p\in\Gamma(f)$, then $G_pf=f$.
\end{definition}
A complementary cell is a connected component of $\Omega\setminus\Gamma(F)$, where the roof is affine. We call it bounded when its closure lies in $\Omega$. A marked carrier is the edge containing a prescribed mark, before any cut is made. Cutting at a mark creates two distinct ends. Compact internal edges include their endpoints; terminal edges lead to the boundary of the stated ambient polygon.

For a bounded open convex domain, the lattice support distance
\begin{equation}\label{eq:lattice-distance}
 D_\Omega(x)=\inf_{m\in\Z^2\setminus\{0\}}
       \left(m\cdot x-\inf_{y\in\Omega}m\cdot y\right)
\end{equation}
is positive in $\Omega$, locally a minimum of finitely many of the displayed planes, and continuous with value zero on $\partial\Omega$. This is \href{https://arxiv.org/pdf/1706.03062v1\#page=5}{Kalinin--Shkolnikov, Lemmas 4.3--4.4\readerlocator{ (PDF pp.~5--6)}}~\cite{KS18}.\label{ext:tropical-distance}
To see why the minimum is locally finite, fix a ball $B_{2r}(x)\subset\Omega$. On $B_r(x)$, every displayed plane of slope $m$ has value at least $r\|m\|$, since it is nonnegative on $\Omega$. Any one of these planes gives a uniform upper bound for the infimum on $B_r(x)$. Only finitely many integral slopes can fall below that bound, so only those planes are needed on this smaller ball.

We must check that taking the infimum of admissible roofs preserves the prescribed corners. Once this is proved, comparisons follow from minimality.

\needspace{5\baselineskip}\begin{proposition}[Existence and minimality]\label{prop:relaxation}

Let $\Omega\subset\R^2$ be a bounded open convex domain with nonempty interior, let $P\subset\Omega$ be finite, and let $F$ be an $\Omega$-tropical series. Then $G_PF$ exists as an $\Omega$-tropical series. It is the pointwise least roof above $F$ having corners at every point of $P$.

If, in addition, $\Omega=\Delta^\circ$ for a compact convex polygon $\Delta$ with rational side slopes and $F$ is a tropical polynomial, then $G_PF$ has a finite global presentation. In particular, $F_{\Delta^\circ,P}$ is a tropical polynomial.
\end{proposition}
\begin{proof}
We construct an admissible upper bound, then show that the infimum of all admissible roofs is still admissible. The roof
\[
 U=F+\sum_{p\in P}\min\{D_\Omega,D_\Omega(p)\}
\]
is admissible. At $p$, its corresponding capped summand has the constant plane and a nonzero lattice slope active. Adding a supporting slope of every other summand preserves these two distinct slopes. Finite minima of admissible roofs are also admissible: at a mark, any summand realizing the smallest value contributes its two supporting slopes to the minimum.

Write $I$ for the pointwise infimum of all admissible roofs. Replacing each roof $g$ by $\min(g,U)$ leaves this infimum unchanged, so assume $F\le g\le U$. Fix $W\Subset\Omega$ and $r>0$ such that a ball of radius $r$ about every point of $W$ lies in $\Omega$. Every locally active plane is a global support by concavity. If it has slope $m$, nonnegativity on that ball gives
\[
 r\|m\|\le g(x)\le\|U\|_\infty \qquad(x\in W).
\]
Thus a finite set $A\subset\Z^2$ suffices on $W$. Complete the coefficients by
$\kappa_m(g)=\max_{\overline\Omega}(g-m\cdot x)$. All these coefficients are bounded, and every $g$ equals $\min_{m\in A}(\kappa_m(g)+m\cdot x)$ on $W$. Consequently
\[
 I(x)=\min_{m\in A}(c_m+m\cdot x),\qquad
 c_m=\inf_g\kappa_m(g),\qquad x\in W.
\]
For each $j$, choose one admissible roof approximating each of the finitely many $c_m$ within $1/j$, and take their minimum $g_j$. Then
$c_m\le\kappa_m(g_j)\le c_m+1/j$ for every $m\in A$. At a prescribed mark in the interior of $W$, two distinct slopes are active for every $g_j$. A fixed pair recurs along a subsequence and remains active in the limit. Hence the mark is a corner of $I$. Finally $F\le I\le U$ gives the zero boundary values and their continuity. Thus $I$ is admissible and is the least element by its definition.\label{app:general-relaxation}

For the polygonal finiteness assertion, return to a polynomial initial roof $F$ and let
$B_\Delta=\min_S\lambda_S$, where $\lambda_S$ has primitive inward integral normal and vanishes on side $S$. Replace $D_\Omega$ by $B_\Delta$ in the first display. This gives a finite admissible polynomial $U_\Delta$, with some Lipschitz constant $L$, and
$I(x)\le U_\Delta(x)\le L\operatorname{dist}(x,\partial\Delta)$.
For an active nonzero slope $m$ of $I$ at $x$, follow the ray in direction $-m$ to $y\in\partial\Delta$. Its length $t$ is at least $\operatorname{dist}(x,\partial\Delta)$, and the supporting inequality gives
\[
 0=I(y)\le I(x)-t\|m\|,\qquad \|m\|\le L.
\]
Only finitely many lattice slopes occur anywhere. For each slope its supporting plane has a unique intercept; their minimum represents $I$ globally. This proves the polygonal finiteness assertion.
\end{proof}
See also \href{https://arxiv.org/pdf/1706.03062v1\#page=7}{Kalinin--Shkolnikov, Lemma 5.7 and Proposition 6.1\readerlocator{ (PDF pp.~7--8)}}~\cite{KS18}, and \href{https://arxiv.org/pdf/1706.03062v1\#page=12}{Remark 9.5\readerlocator{ (PDF p.~12)}}~\cite{KS18}.

\needspace{9\baselineskip}\subsection*{One-point relaxation}\label{app:one-point-relaxation}
\pdfbookmark[2]{One-point relaxation}{one-point-relaxation}
To compute $G_pf$, give every integral slope its lowest supporting plane. Thus
\[
 a_m=\kappa_m(f)=\max_{x\in\overline\Omega}\bigl(f(x)-m\cdot x\bigr),
 \qquad f(x)=\min_{m\in\Z^2}(a_m+m\cdot x),\quad x\in\Omega.
\]
The minimum is locally finite in the interior. Some of these planes touch only on an edge, at a vertex, or on the boundary. They must be retained because a subsequent increase can make them active on an open set. This is the canonical presentation of \href{https://arxiv.org/pdf/1706.03062v1\#page=4}{Kalinin--Shkolnikov, Definition 3.7\readerlocator{ (PDF p.~4)}}~\cite{KS18}.

Suppose that $p\notin\Gamma(f)$ and that $m_*$ is the unique active slope at $p$. Put
\[
 H(x)=\min_{m\ne m_*}(a_m+m\cdot x),\qquad
 \delta=H(p)-f(p)>0.
\]
Raise the $m_*$-plane by $\delta$ and leave the other coefficients fixed. Then
\[
 G_pf(x)=\min\{a_{m_*}+\delta+m_*\cdot x,\ H(x)\}.
\]
At $p$ the two terms agree. This is the first height at which the active plane meets a competing plane over the mark. The new coefficient is $H(p)-m_*\cdot p$; its increment is $H(p)-f(p)$. The rule is \href{https://arxiv.org/pdf/1706.03062v1\#page=7}{Kalinin--Shkolnikov, Lemma 5.9\readerlocator{ (PDF pp.~7--8)}}~\cite{KS18}. Appendix~\ref{sec:C}, p.~\pageref{app:wave-convergence}, verifies the boundary values and minimality of the updated roof.

For a nonempty cloud $P=\{p_1,\ldots,p_N\}$, choose an ordering and set
\[
 f_0=0,\qquad f_{k+1}=G_{p_{1+(k\bmod N)}}f_k.
\]
These roofs increase and converge uniformly on $\overline\Omega$ to $F_{\Omega,P}$. Any sequence of visits in which every mark occurs infinitely often gives the same limit. Only finitely many coefficients can change for a fixed cloud, and each of them increases to a finite limit. Appendix~\ref{sec:C} gives the argument, following \href{https://arxiv.org/pdf/1706.03062v1\#page=8}{Kalinin--Shkolnikov, Proposition 6.1\readerlocator{ (PDF p.~8)}}~\cite{KS18}.

The numerical approximation on p.~\pageref{intro:relaxation-approximation} uses this uniform convergence for each cloud.

\needspace{7\baselineskip}\subsubsection*{Three updates on a square}\label{app:wave-square}
An update at one mark can remove a corner already created at another, so the marks must be revisited. The following example reaches the least roof after three visits. On $\Omega=(0,1)^2$, write
\[
 D(x,y)=\min\{x,1-x,y,1-y\},\qquad
 p=(1/4,1/2),\quad q=(1/2,1/2).
\]
Starting from zero gives
\begin{align*}
 f_1=G_p0&=\min\{D,1/4\},\\
 f_2=G_qf_1&=D,\\
 f_3=G_pf_2&=\min\{2x,x+1/4,1-x,y,1-y\}\\
           &=\min\{2x,1-x,y,1-y\}.
\end{align*}
The first update creates a corner at $p$. The second raises the constant plane from $1/4$ to $1/2$. The resulting roof $D$ equals $x$ near $p$. On returning to $p$, the plane $x$ rises to $x+1/4$, and $2x$ becomes active. Both marks are now corners, and subsequent visits leave the roof unchanged. Thus $f_3=F_{\Omega,\{p,q\}}$. The upper row of Figure~\ref{fig:one-point-relaxation} shows the three curves.

The last equality for $f_3$ follows because $x+1/4\ge2x$ for $x\le1/4$, while $x+1/4\ge1/2\ge\min\{y,1-y\}$ for $x\ge1/4$. The raised plane still belongs to the canonical presentation, although it no longer defines an open cell.

At $p$ the three active planes are $2x,y,1-y$. They all take the value $1/2$, and their gradients $(2,0),(0,1),(0,-1)$ are not collinear. At $q$, the active planes are $1-x,y,1-y$, again with value $1/2$. The symmetry $y=1/2$ has made three cells meet at each mark. This configuration is nongeneric: strong genericity requires the marks to lie in edge interiors. The horizontal edge joining them has weight $2$, since its adjacent gradients differ by $(0,2)$. The weight-one conclusion of Theorem~\ref{thm:mildness} assumes strong genericity. The relaxation rule and its convergence do not.

\subsubsection*{A nearby generic configuration}
The following perturbation places both marks in distinct edge interiors and produces two bounded cells. Move the marks to
\[
 p'=(1/4,12/25),\qquad q'=(13/25,27/50).
\]
The same three visits give
\[
 G_{p'}0=\min\{D,1/4\},\qquad
 G_{q'}G_{p'}0=\min\{D,23/50\},
\]
and then
\[
 f'=G_{p'}G_{q'}G_{p'}0
   =\min\{23/50,\ x+21/100,\ 2x,\ 1-x,\ y,\ 1-y\}.
\]
At $p'$ only the constant plane and $x+21/100$ attain the minimum. At $q'$ only the constant plane and $1-y$ do. Both marks are now in edge interiors, so the iteration stops and $f'=F_{\Omega,\{p',q'\}}$. Two bounded cells have appeared: the rectangle where the constant plane is active, and the narrow trapezoid where $x+21/100$ is active. The lower row of Figure~\ref{fig:one-point-relaxation} shows them. All eleven edges have weight one.

The marked type persists under small independent movements of both points. Write $p'=(a,b)$, $q'=(c,d)$, and $t=1-d$. The strict inequalities
\[
 0<a<t<\min\{2a,1/2\},\qquad t<b<1-t,\qquad t<c<1-t
\]
hold at the chosen points. On this open set the three visits give
\[
 \min\{D,a\},\qquad \min\{D,t\},\qquad
 \min\{t,x+t-a,2x,1-x,y,1-y\}.
\]
Indeed, the first two visits raise the constant plane. At the third, $x$ is the unique active plane at $(a,b)$, and the next height there is $t$; raising it by $t-a$ exposes $2x$. Every other unchanged plane with integral gradient is dominated by one of $2x,1-x,y,1-y$. Put $s=t-a$. The two bounded cells have vertices
\[
 (s,2s),\ (a,t),\ (a,1-t),\ (s,1-2s)
 \quad\hbox{and}\quad
 (a,t),\ (1-t,t),\ (1-t,1-t),\ (a,1-t).
\]
Since $0<s<a$ and $s<1/4$, all their sides remain nondegenerate. The six vertices are trivalent, the marked carriers remain distinct, and the four terminal edges still end at the corners of the square. The marked type is therefore constant on this open set. Together with the two distinct marked edge interiors, this proves strong genericity.

\begin{figure}[!htbp]\centering
\begin{tikzpicture}[x=1cm,y=1cm,font=\small,>=stealth]
\begin{scope}[xshift=0.0cm,x=3.7cm,y=3.7cm]
\draw[Ink!60] (0,0) rectangle (1,1);
\draw[Geometry,line width=.9pt] (0.00000000,0.00000000)--(0.25000000,0.25000000);
\draw[Geometry,line width=.9pt] (0.00000000,1.00000000)--(0.25000000,0.75000000);
\draw[Geometry,line width=.9pt] (0.25000000,0.25000000)--(0.25000000,0.75000000);
\draw[Geometry,line width=.9pt] (0.25000000,0.25000000)--(0.75000000,0.25000000);
\draw[Geometry,line width=.9pt] (0.25000000,0.75000000)--(0.75000000,0.75000000);
\draw[Geometry,line width=.9pt] (0.75000000,0.25000000)--(0.75000000,0.75000000);
\draw[Geometry,line width=.9pt] (0.75000000,0.25000000)--(1.00000000,0.00000000);
\draw[Geometry,line width=.9pt] (0.75000000,0.75000000)--(1.00000000,1.00000000);
\draw[draw=Ink,fill=Ink,line width=.8pt] (0.25000000,0.50000000) circle[radius=2.3pt];
\node[above left=4pt] at (0.25000000,0.50000000) {$p$};
\draw[draw=Ink,fill=white,line width=.8pt] (0.50000000,0.50000000) circle[radius=2.3pt];
\node[above right=4pt] at (0.50000000,0.50000000) {$q$};
\node[below=8pt] at (0.50000000,0.00000000) {$f_1=G_p0=\min\{D,1/4\}$};
\end{scope}
\begin{scope}[xshift=5.15cm,x=3.7cm,y=3.7cm]
\draw[Ink!60] (0,0) rectangle (1,1);
\draw[Geometry,line width=.9pt] (0.00000000,0.00000000)--(0.50000000,0.50000000);
\draw[Geometry,line width=.9pt] (0.00000000,1.00000000)--(0.50000000,0.50000000);
\draw[Geometry,line width=.9pt] (0.50000000,0.50000000)--(1.00000000,0.00000000);
\draw[Geometry,line width=.9pt] (0.50000000,0.50000000)--(1.00000000,1.00000000);
\draw[draw=Ink,fill=white,line width=.8pt] (0.25000000,0.50000000) circle[radius=2.3pt];
\node[above left=4pt] at (0.25000000,0.50000000) {$p$};
\draw[draw=Ink,fill=Ink,line width=.8pt] (0.50000000,0.50000000) circle[radius=2.3pt];
\node[above right=4pt] at (0.50000000,0.50000000) {$q$};
\node[below=8pt] at (0.50000000,0.00000000) {$f_2=G_qf_1=D$};
\end{scope}
\begin{scope}[xshift=10.3cm,x=3.7cm,y=3.7cm]
\fill[Geometry!12] (0,0)--(.25,.5)--(0,1)--cycle;
\draw[Ink!60] (0,0) rectangle (1,1);
\draw[Geometry,line width=.9pt] (0.00000000,0.00000000)--(0.25000000,0.50000000);
\draw[Geometry,line width=.9pt] (0.00000000,1.00000000)--(0.25000000,0.50000000);
\draw[Geometry,line width=.9pt] (0.25000000,0.50000000)--(0.50000000,0.50000000);
\draw[Geometry,line width=.9pt] (0.50000000,0.50000000)--(1.00000000,0.00000000);
\draw[Geometry,line width=.9pt] (0.50000000,0.50000000)--(1.00000000,1.00000000);
\draw[draw=Ink,fill=Ink,line width=.8pt] (0.25000000,0.50000000) circle[radius=2.3pt];
\node[above left=4pt] at (0.25000000,0.50000000) {$p$};
\draw[draw=Ink,fill=Ink,line width=.8pt] (0.50000000,0.50000000) circle[radius=2.3pt];
\node[above right=4pt] at (0.50000000,0.50000000) {$q$};
\node[below=5pt] at (0.37500000,0.50000000) {$w=2$};
\node[Geometry,rotate=90] at (0.07000000,0.50000000) {$2x$};
\node[below=8pt] at (0.50000000,0.00000000) {$f_3=G_pf_2$};
\end{scope}
\draw[->,Ink!70] (3.88,1.85)--(4.94,1.85) node[midway,above] {$G_q$};
\draw[->,Ink!70] (9.03,1.85)--(10.09,1.85) node[midway,above] {$G_p$};
\begin{scope}[xshift=.8cm,yshift=-5.75cm,x=4.4cm,y=4.4cm]
\fill[Algebra!20] (0.21000000,0.42000000)--(0.25000000,0.46000000)--(0.25000000,0.54000000)--(0.21000000,0.58000000)--cycle;
\fill[Analysis!15] (0.25000000,0.46000000)--(0.54000000,0.46000000)--(0.54000000,0.54000000)--(0.25000000,0.54000000)--cycle;
\draw[Ink!60] (0,0) rectangle (1,1);
\draw[Geometry,line width=.9pt] (0.00000000,0.00000000)--(0.21000000,0.42000000);
\draw[Geometry,line width=.9pt] (0.00000000,1.00000000)--(0.21000000,0.58000000);
\draw[Geometry,line width=.9pt] (1.00000000,0.00000000)--(0.54000000,0.46000000);
\draw[Geometry,line width=.9pt] (1.00000000,1.00000000)--(0.54000000,0.54000000);
\draw[Geometry,line width=.9pt] (0.21000000,0.42000000)--(0.21000000,0.58000000);
\draw[Geometry,line width=.9pt] (0.21000000,0.42000000)--(0.25000000,0.46000000);
\draw[Geometry,line width=.9pt] (0.21000000,0.58000000)--(0.25000000,0.54000000);
\draw[Geometry,line width=.9pt] (0.25000000,0.46000000)--(0.25000000,0.54000000);
\draw[Geometry,line width=.9pt] (0.25000000,0.46000000)--(0.54000000,0.46000000);
\draw[Geometry,line width=.9pt] (0.25000000,0.54000000)--(0.54000000,0.54000000);
\draw[Geometry,line width=.9pt] (0.54000000,0.46000000)--(0.54000000,0.54000000);
\draw[draw=Probability,fill=Probability,line width=.8pt] (0.25000000,0.48000000) circle[radius=2.3pt];
\draw[draw=Probability,fill=Probability,line width=.8pt] (0.52000000,0.54000000) circle[radius=2.3pt];
\draw[Ink!60,dashed] (.17,.37) rectangle (.61,.66);
\node[below=8pt] at (0.50000000,0.00000000) {$f'=F_{\Omega,\{p',q'\}}$};
\end{scope}
\draw[->,Ink!60] (5.5,-3.4)--(7.35,-3.4) node[midway,above,font=\footnotesize] {enlargement};
\begin{scope}[xshift=5.7cm,yshift=-9.15cm,x=12cm,y=12cm]
\begin{scope}\clip (.17,.37) rectangle (.61,.66);
\fill[Algebra!20] (0.21000000,0.42000000)--(0.25000000,0.46000000)--(0.25000000,0.54000000)--(0.21000000,0.58000000)--cycle;
\fill[Analysis!15] (0.25000000,0.46000000)--(0.54000000,0.46000000)--(0.54000000,0.54000000)--(0.25000000,0.54000000)--cycle;
\draw[Geometry,line width=.9pt] (0.00000000,0.00000000)--(0.21000000,0.42000000);
\draw[Geometry,line width=.9pt] (0.00000000,1.00000000)--(0.21000000,0.58000000);
\draw[Geometry,line width=.9pt] (1.00000000,0.00000000)--(0.54000000,0.46000000);
\draw[Geometry,line width=.9pt] (1.00000000,1.00000000)--(0.54000000,0.54000000);
\draw[Geometry,line width=.9pt] (0.21000000,0.42000000)--(0.21000000,0.58000000);
\draw[Geometry,line width=.9pt] (0.21000000,0.42000000)--(0.25000000,0.46000000);
\draw[Geometry,line width=.9pt] (0.21000000,0.58000000)--(0.25000000,0.54000000);
\draw[Geometry,line width=.9pt] (0.25000000,0.46000000)--(0.25000000,0.54000000);
\draw[Geometry,line width=.9pt] (0.25000000,0.46000000)--(0.54000000,0.46000000);
\draw[Geometry,line width=.9pt] (0.25000000,0.54000000)--(0.54000000,0.54000000);
\draw[Geometry,line width=.9pt] (0.54000000,0.46000000)--(0.54000000,0.54000000);
\end{scope}
\draw[Ink!60,dashed] (.17,.37) rectangle (.61,.66);
\draw[draw=Probability,fill=Probability,line width=.8pt] (0.25000000,0.48000000) circle[radius=2.3pt];
\draw[draw=Probability,fill=Probability,line width=.8pt] (0.52000000,0.54000000) circle[radius=2.3pt];
\node[right=5pt,Probability] at (0.25000000,0.48000000) {$p'$};
\node[above left=4pt,Probability] at (0.52000000,0.54000000) {$q'$};
\node[rotate=90,font=\scriptsize,Algebra] at (0.22400000,0.51000000) {$x+21/100$};
\node[Analysis] at (0.39800000,0.50000000) {$23/50$};
\node[below,font=\footnotesize] at (0.39000000,0.34000000) {Two bounded cells; both marks in edge interiors};
\end{scope}
\end{tikzpicture}
\caption{Three updates on the square, and a nearby generic configuration. Above, the visit to $q$ removes the corner at $p$; the return to $p$ exposes $2x$ and restores it. Filled marks are corners of the current roof; open marks are not. This symmetric example is nongeneric: both final marks are vertices, and their joining edge has weight $2$. Below, the recomputed least roof for $p'=(1/4,12/25)$ and $q'=(13/25,27/50)$ has two bounded cells and puts both marks in edge interiors. The enlargement shows their active planes. All eleven edges of this second curve have weight one.}
\label{fig:one-point-relaxation}
\end{figure}

\needspace{8\baselineskip}\subsection*{Supporting planes and boundary slopes}
The contact set of a supporting plane determines where moving it can change the roof. We use the essential planes, the affine terms active on open cells, to describe these contact sets. Concavity makes each a support throughout the polygon. Their behavior near the boundary will also give a bound for the curve's weighted length.

\needspace{5\baselineskip}\begin{proposition}[Essential planes]\label{prop:essential}

A finite tropical polynomial has an essential presentation with one affine function for each open cell. Each such function lies above the roof everywhere, and its contact set is the closure of that cell. At an interior point of an edge, exactly the two adjacent essential planes are active.
\end{proposition}
\begin{proof}
Let $F$ be the polynomial, and let $L_R$ be its affine function on a full-dimensional cell $R$. Concavity makes $L_R$ a global support. Cell closures cover the polygon, giving $F=\min_R L_R$. The closed contact set $\{F=L_R\}$ is convex by concavity. If it reached outside $\overline R$, its convex hull with a small ball in $R$ would enlarge the cell, contradicting maximality. Thus it is exactly $\overline R$.

At an interior point of an edge, precisely two full-dimensional cells are adjacent. The contact-set identity excludes every other essential plane there.
\end{proof}
The canonical presentation retains planes that touch only an edge, a vertex, or the boundary. The essential presentation consists of the planes active on open cells (Proposition~\ref{prop:essential}). If $L_1,\ldots,L_k$ are active at a point and $\theta_j\ge0$, $\sum_j\theta_j=1$, then $L=\sum_j\theta_jL_j$ still lies above $F$. Moreover,
\[
 \{L=F\}=\bigcap_{\theta_j>0}\{L_j=F\},
\]
because each summand $L_j-F$ is nonnegative. In Lemma~\ref{lem:lowering}, this identity lets us choose a plane whose contact set is exactly the face to be lowered.

\needspace{5\baselineskip}\begin{lemma}[The side normal form]\label{lem:side}

For each side $S$ of $\Delta$, choose its primitive integral inward normal $n_S$ and an affine defining function $\lambda_S$ with gradient $n_S$, positive in $\Delta^\circ$ and zero on $S$. Near the relative interior of $S$, a nonnegative zero-boundary tropical polynomial has the form $m_F(S)\lambda_S$, with $m_F(S)\in\Z_{\ge0}$.
\end{lemma}
\begin{proof}
Fix $x\in\operatorname{relint}S$. An essential plane active at $x$ is nonnegative on $\Delta$ and zero at $x$. Its restriction to $S$ is affine and nonnegative, so it vanishes along the whole side. It is therefore $k\lambda_S$; primitivity of the normal makes $k$ an integer, and nonnegativity makes $k\ge0$.

Take the smallest such $k$, denoted $m_F(S)$. Every other essential plane is positive at $x$ and stays above $m_F(S)\lambda_S$ nearby. This proves the normal form. The same minimum works along the relative interior of $S$.
\end{proof}
We call $m_F(S)$ the side quasidegree and $D_\partial(F)=\sum_Sm_F(S)$ the total side quasidegree. Comparing the side quasidegrees will give a bound for the weighted length of the curve in Theorem~\ref{thm:complexity}.

\needspace{8\baselineskip}\subsection*{A competitor of controlled height}
To bound the least roof, we construct a competitor with all the prescribed corners. The first step is interpolation with a small set of gradients; the boundary values will be imposed afterward.

\needspace{5\baselineskip}\begin{theorem}[Interpolation by a small set of slopes]\label{thm:interpolation}

Let $p_1,\ldots,p_N\in\R^2$, with $N\ge1$, and let $m_1,\ldots,m_M\in\Z^2$ be distinct, with $M>N$. There are a nonempty index set $J\subset\{1,\ldots,M\}$ and real coefficients $c_j$, $j\in J$, such that the finite tropical polynomial
\[
 q(x)=\min_{j\in J}\{c_j+\langle m_j,x\rangle\}
\]
has a corner at every $p_i$. In particular, taking all slopes in $[-r_N,r_N]^2\cap\Z^2$, where
\[
 r_N=\left\lceil\frac{\sqrt{N+1}-1}{2}\right\rceil,
\]
gives $\operatorname{Lip}(q)\le\sqrt2\,r_N=O(\sqrt N)$. On any bounded domain $\Omega$, the normalization $\widehat q=q-\min_{\overline\Omega}q$ has the same corners and satisfies
\[
 0\le\widehat q\le\sqrt2\,r_N\operatorname{diam}(\Omega).
\]
\end{theorem}
\begin{proof}
\ifjournaledition
\else
Work over the Hahn field $\mathbb K=\C((t^\R))$, constructed in Appendix~\ref{sec:A}. The matrix
\[
 \widetilde A_{ij}=t^{\langle m_j,p_i\rangle}
\]
has a nonzero kernel vector $z$, since $M>N$. Retain its nonzero coordinates and put $c_j=\operatorname{val}(z_j)$. In each row, $\sum_j\widetilde A_{ij}z_j=0$, so the least value of $c_j+\langle m_j,p_i\rangle$ occurs at least twice by Lemma~\ref{lem:hahn-cancellation}. The corresponding distinct gradients are supporting slopes of $q$ at $p_i$; hence $q$ is not differentiable there.

For the specified $r_N$, the square contains $(2r_N+1)^2\ge N+1$ slopes, all of norm at most $\sqrt2\,r_N$. If the $j$th affine function is active at $y$, then
\[
 q(x)-q(y)\le\langle m_j,x-y\rangle\le\sqrt2\,r_N|x-y|.
\]
Interchanging $x,y$ proves the Lipschitz estimate. Subtracting the minimum on $\overline\Omega$ proves the height estimate and preserves the corners.
\fi
\end{proof}
The remaining condition is the boundary value. We impose it by a cap which vanishes on $\partial\Omega$ and lies strictly above $\widehat q$ near every mark.

\needspace{5\baselineskip}\begin{proposition}[Height and boundary control]\label{prop:barriers}

Let $D_\Omega$ denote the lattice support-distance function
\[
 D_\Omega(x)=\inf_{v\in\Z^2\setminus\{0\}}
 \left(v\cdot x-\inf_{y\in\Omega}v\cdot y\right).
\]
It is a tropical series, positive in $\Omega$ and continuous with value zero on $\partial\Omega$. For every finite $P\subset\Omega$,
\begin{equation}\label{eq:depth-free}
 \|F_{\Omega,P}\|_\infty\le C_\Omega\sqrt{|P|}.
\end{equation}
If $P\subset K\Subset\Omega$, then
\begin{equation}\label{eq:fixed-depth}
 F_{\Omega,P}\le C_{\Omega,K}\sqrt{|P|}\,D_\Omega.
\end{equation}
There is an admissible competitor $E=\min\{\widehat q,\ell_{\rm cap}D_\Omega\}$, where $\widehat q=q-\min_{\overline\Omega}q\ge0$ and the integer $\ell_{\rm cap}$ is chosen so that $\ell_{\rm cap}D_\Omega(p)>\widehat q(p)$ at every source. For fixed $K$, one can choose $\ell_{\rm cap}=O_{\Omega,K}(\sqrt{|P|})$.
\end{proposition}
\begin{proof}
For $P=\varnothing$, the roof is zero and both bounds hold. Suppose now that $N=|P|\ge1$. Use the normalized interpolant $\widehat q$ of Theorem~\ref{thm:interpolation}, and put $H=\max\widehat q\le C_\Omega\sqrt N$. The properties of $D_\Omega$ stated on page~\pageref{ext:tropical-distance} give
\[
 \rho=\min_{p\in P}D_\Omega(p)>0,\qquad
 \ell_{\rm cap}=\left\lceil H/\rho\right\rceil+1.
\]
The function $E=\min\{\widehat q,\ell_{\rm cap}D_\Omega\}$ is nonnegative, continuous, zero on the boundary, and locally tropical with integer slopes. At every mark,
\[
 \widehat q(p)\le H<\ell_{\rm cap}\rho\le \ell_{\rm cap}D_\Omega(p).
\]
The strict inequality holds near $p$, so $E=\widehat q$ there and the corner survives. Minimality gives
\begin{equation}\label{eq:two-competitor-bounds}
 0\le F_{\Omega,P}\le E\le\widehat q,
 \qquad F_{\Omega,P}\le E\le \ell_{\rm cap}D_\Omega.
\end{equation}
The bound $F_{\Omega,P}\le\widehat q$ proves the height estimate independently of the distance of the marks from the boundary. To obtain a common boundary barrier for $P\subset K\Subset\Omega$, replace $\rho$ by $\delta_K=\min_KD_\Omega>0$. Then
$\ell_{\rm cap}\le H/\delta_K+2\le C_{\Omega,K}\sqrt N$, and the second inequality in \eqref{eq:two-competitor-bounds} proves the boundary estimate.

For the finite polygonal competitor, use the primitive side barrier
\begin{equation}\label{eq:polygon-competitor}
 B_\Delta=\min_S\lambda_S,\quad
 \ell_{{\rm cap},\Delta}=\left\lceil\frac{H}{\min_K B_\Delta}\right\rceil+1,
 \quad E_\Delta=\min\{\widehat q,\ell_{{\rm cap},\Delta} B_\Delta\}.
\end{equation}
The same strict comparison at each mark proves admissibility. Its slopes belong to the interpolating square or to the finite list $\ell_{{\rm cap},\Delta} n_S$, so their norms are $O_{\Delta,K}(\sqrt N)$. Near the relative interior of $S$, $B_\Delta=\lambda_S$. Since $E_\Delta\le \ell_{{\rm cap},\Delta} B_\Delta$, Lemma~\ref{lem:side} gives
$m_{E_\Delta}(S)\le \ell_{{\rm cap},\Delta}$.
\end{proof}
The barriers satisfy $D_{\Delta^\circ}\le B_\Delta$ and need not agree. On a polygon, $B_\Delta$ controls the side quasidegrees; on a general convex domain, $D_\Omega$ supplies a continuous boundary barrier. More explicitly, $\ell_{\rm cap}\le C_\Omega\sqrt N/\rho_\Omega(P)+2$, where $\rho_\Omega(P)=\min_PD_\Omega$.

On a fixed polygon, the boundary gradients also control a weighted sum of edge lengths, called the symplectic area of the curve. The weight of an edge is the lattice length of the difference of its adjacent cell gradients. Summing the gradient flux over the cells expresses this weighted length as a sum over the polygon sides. The competitor just constructed bounds this sum by $O(\sqrt N)$ when the sources stay in a fixed interior compact set.

\needspace{5\baselineskip}\begin{theorem}[Polygonal complexity]\label{thm:complexity}

Let $\Delta$ be a compact convex polygon with rational side slopes, and put $\Omega=\Delta^\circ$. For a finite nonnegative tropical polynomial $F$ on $\Delta$, zero on $\partial\Delta$, the symplectic area satisfies
\[
 A_{\rm symp}(\Gamma(F))
 =\sum_S m_F(S)\,A_{\rm trop}(S),\qquad
 A_{\rm symp}(\Gamma)=\sum_e w(e)\,\ell(e)\,\|v_e\|,
\]
where $\ell$ denotes Euclidean length, $w(e)$ is the edge weight, $v_e$ and $v_S$ are primitive integral tangents to $e$ and $S$, and $A_{\rm trop}(S)=\ell(S)\|v_S\|$. The least roof has symplectic area at most that of a finite admissible competitor. Consequently, for every finite $P\subset K\Subset\Omega$,
\[
 A_{\rm symp}(\Gamma(F_{\Omega,P})),\quad
 \mathcal H^1(\Gamma(F_{\Omega,P})),\quad
 D_\partial(F_{\Omega,P})
 \ \le C_{\Omega,K}\sqrt{|P|}.
\]
\end{theorem}
\begin{proof}
For a finite tropical polynomial $F$, the boundary identity is \href{https://arxiv.org/pdf/1706.03062v1\#page=18}{Kalinin--Shkolnikov, Lemma 14.6\readerlocator{ (PDF p.~18)}}~\cite{KS18}. A flux calculation explains its normalization. On every complementary cell $C$, the gradient $m_C$ is constant, so its total outward flux is zero. Sum these identities over all cells. Across an interior edge $e$, choose the unit normal $n$ from $C$ to its neighbor $C'$. Concavity gives
\[
 (m_C-m_{C'})\cdot n=\|m_C-m_{C'}\|
 =w(e)\|v_e\|.
\]
On side $S$, Lemma~\ref{lem:side} gives outward flux density
$-m_F(S)\|n_S\|=-m_F(S)\|v_S\|$. The edge contributions and boundary contributions sum to zero. This proves the stated identity directly. Vertices have zero length and contribute nothing to this calculation.

If $E$ is any finite admissible competitor, minimality gives $F_P\le E$. Compare their side normal forms at an inward point where $\lambda_S>0$ to obtain
$m_{F_P}(S)\le m_E(S)$. Multiply by $A_{\rm trop}(S)>0$ and sum. This proves the area comparison.

Apply this comparison to the competitor $E_\Delta$ in \eqref{eq:polygon-competitor}:
\[
 A_{\rm symp}(\Gamma_P)
 \le A_{\rm symp}(\Gamma(E_\Delta))
 \le \ell_{{\rm cap},\Delta}\sum_S A_{\rm trop}(S)
 \le C_{\Delta,K}\sqrt N.
\]
Since $w(e)\ge1$ and $\|v_e\|\ge1$, the same bound controls ordinary length. Finally, with $a_*=\min_S A_{\rm trop}(S)>0$, the boundary identity gives
$a_*D_\partial(F_P)\le A_{\rm symp}(\Gamma_P)$. This proves all three bounds.
\end{proof}

\needspace{8\baselineskip}\subsection*{Adding and removing sources}
The height bound also controls the change in a roof when sources are added or removed. Adding two concave tropical roofs preserves every corner of either summand. Their sum is therefore admissible for the union of the source sets.
\needspace{5\baselineskip}\begin{lemma}[Subadditivity of sources]\label{lem:subadditivity}

For arbitrary finite $P,Q\subset\Omega$,
\begin{equation}\label{eq:subadditivity}
 F_{\Omega,P}\le F_{\Omega,P\cup Q}
 \le F_{\Omega,P}+F_{\Omega,Q}.
\end{equation}
\end{lemma}

\begin{proof}
A roof admissible for $P\cup Q$ is admissible for $P$, giving the first inequality. For the second, add the two roofs.

This sum is again a tropical series. Indeed, on a neighborhood where $F_P=\min_i L_i$ and $F_Q=\min_j M_j$, we have
\[
 F_P+F_Q=\min_{i,j}(L_i+M_j).
\]
The gradients remain integral. The sum is nonnegative, concave, continuous on $\overline\Omega$, and zero on the boundary.

It remains to check the corners. At an interior point $p$, adding two supporting-plane inequalities gives
\begin{equation}\label{eq:sum-supergradients}
 \partial^+(F_P+F_Q)(p)\supset
 \partial^+F_P(p)+\partial^+F_Q(p).
\end{equation}
Each superdifferential is nonempty: an active local affine function gives a supporting slope, and concavity makes its support global. If $p\in P$, choose two distinct slopes $a,b\in\partial^+F_P(p)$ and any $c\in\partial^+F_Q(p)$. The distinct slopes $a+c,b+c$ belong to the left-hand side of \eqref{eq:sum-supergradients}. Thus $p$ remains a corner. The same argument applies at every point of $Q$.

Thus $F_P+F_Q$ is admissible for $P\cup Q$, giving the second inequality by minimality. The convention $F_\varnothing=0$ covers empty sets.
\end{proof}

\needspace{5\baselineskip}\begin{corollary}[Removing a small set of sources]\label{cor:removal}

For finite $A\subset P\subset\Omega$,
\begin{equation}\label{eq:removal}
 0\le F_{\Omega,P}-F_{\Omega,A}
 \le C_\Omega\sqrt{|P\setminus A|}.
\end{equation}
Thus deleting $r$ points from an $N$-point cloud changes its $N^{-1/2}$-normalized roof by at most $C_\Omega\sqrt{r/N}$, including when the deleted points approach the boundary.
\end{corollary}

\begin{proof}
Apply Lemma~\ref{lem:subadditivity} to $A$ and $P\setminus A$:
\[
 0\le F_P-F_A\le F_{P\setminus A}.
\]
The height estimate in Proposition~\ref{prop:barriers} bounds the last roof by $C_\Omega\sqrt{|P\setminus A|}$. For an empty discarded set the difference is zero. Dividing by $\sqrt N$ gives the normalized assertion.
\end{proof}
After adding sources, we may restart the iteration from the roof already obtained. If $P\subset Q$, then
\[
 G_QF_{\Omega,P}=F_{\Omega,Q}.
\]
Indeed, $F_{\Omega,Q}\ge F_{\Omega,P}$ and is the least roof with all the corners in $Q$. To obtain it from $F_{\Omega,P}$, visit every mark of the enlarged set repeatedly.

\needspace{12\baselineskip}\section{\texorpdfstring{\headingcolor{Geometry}{The tree determined by the marked points}}{The tree determined by the marked points}}\label{sec:3}

For a strongly generic cloud of $N$ sources on a rational-slope polygon, the tropical curve has exactly $N$ bounded cells, and cutting it at all the marks gives a tree. Here a bounded cell means a complementary cell whose closure lies inside the domain. These conclusions relate the prescribed corners to the topology of the curve and will allow us to count curvature in Section~\ref{sec:4}.

Minimality constrains which bounded cells can occur. Their supporting planes are strictly positive on the closed domain, so they may be lowered slightly while the boundary planes stay fixed. Such a lowering is admissible only if at least two planes remain active at every mark. We use this condition first to count the bounded cells and then to prove the tree statement.

The incidence vectors of marks and bounded cells span $\mathbf F_2^g$, giving $g\le N$ for the number $g$ of bounded cells. Genericity supplies the reverse inequality: roofs of a fixed type have at most $g$ parameters, while $N$ edge-interior marks add at most $N$ more. Independent motion in $2N$ coordinates forces $g=N$. The incidence vectors therefore form a basis. In the planar dual graph, whose vertices represent faces, the edges crossing the marked carriers form a spanning tree. Deleting one of these dual edges will determine which planes can be lowered together in Section~\ref{sec:4}.

Here the embedded type records the vertices and edges, their incidences, the cell gradients, edge weights, cyclic orders, and ordered boundary contacts. The marked embedded type also records the carrier of each labelled mark. We write $\operatorname{Conf}_N(\Omega)$ for the space of ordered $N$-tuples of distinct points of $\Omega$.

\needspace{5\baselineskip}\begin{definition}[Strong genericity]\label{def:strong-genericity}
Let $\Delta$ be a compact convex polygon with rational side slopes. A configuration $P\in\operatorname{Conf}_N(\Delta^\circ)$ is \emph{strongly generic} if every mark lies in an edge interior of $\Gamma(F_{\Delta^\circ,P})$, different marks have different carrier edges, and the marked embedded type is constant on some open neighborhood of $P$ in $\operatorname{Conf}_N(\Delta^\circ)$. Thus the type must persist under all sufficiently small independent perturbations of the $2N$ source coordinates. We denote this locus by $G_N(\Delta)$.
\end{definition}
The definition does not exclude transverse four-valent nodes.

\needspace{5\baselineskip}\begin{theorem}[Strong genericity on a polygon]\label{thm:generic-locus}

For each rational-slope polygon $\Delta$ and each $N$, the strongly generic locus $G_N(\Delta)\subset\operatorname{Conf}_N(\Delta^\circ)$ is open, dense, invariant under permutations, and of full measure.
\end{theorem}
\begin{proof}
On each rational-coordinate open polygon $O\Subset\Delta^\circ$, Lemma~\ref{lem:slope-support} gives a common finite slope set. With these slopes fixed, the canonical coefficients have a semilinear graph and depend affinely on the marks on each cell of a finite partition (Proposition~\ref{prop:coefficient-model}). The exact active faces and their incidences refine this partition into strata of constant marked embedded type (Propositions~\ref{prop:active-faces}--\ref{prop:type-strata}). Proposition~\ref{prop:discriminants} places a mark at a vertex, or two marks on one edge, in exceptional sets of dimension at most $2N-1$. The remaining open strata therefore form a dense set of full measure. Their stable type records gradients, weights, endpoints, cyclic orders, boundary contacts and their order, and the exact marked faces. All these data persist under small independent perturbations of the $2N$ source coordinates. The details are given in Appendix~\ref{sec:C}.

The properties in Definition~\ref{def:strong-genericity} persist throughout a smaller neighborhood, so $G_N(\Delta)$ is open. A countable exhaustion by $O$ makes its complement null, hence it is dense. Relabelling preserves it.
\end{proof}

\needspace{5\baselineskip}\begin{lemma}[Bounded cells and the binary constraint]\label{lem:binary}

Write $R_1,\ldots,R_g$ for the bounded cells, whose closures lie in $\Omega$. Their essential planes are strictly positive on $\overline\Omega$. Set
\[
 A_i=\{j:p_i\in\overline{R_j}\}.
\]
Let $a_i=\mathbf1_{A_i}\in\mathbf F_2^g$: its $j$th coordinate is one precisely when $p_i\in\overline{R_j}$. The vectors $a_1,\ldots,a_N$ span $\mathbf F_2^g$. In particular $g\le N$.
\end{lemma}
\begin{proof}
An essential plane $L_j$ of a bounded cell is nonnegative on $\Delta$. A boundary zero would be a contact point, hence belong to $\overline{R_j}$ by Proposition~\ref{prop:essential}, a contradiction. An interior zero would make the nonnegative affine function identically zero and force $F_P=0$. Thus $\min_\Delta L_j>0$.

Suppose the incidence vectors fail to span. There is a nonempty set $J\subset\{1,\ldots,g\}$ whose indicator is orthogonal to all of them. Thus $|A_i\cap J|$ is even for every $i$: it is zero or at least two. Set
\[
 \alpha=\min_{j\in J}\min_\Delta L_j>0,\qquad
 \beta=\min_{\substack{i,\ j\in J\\j\notin A_i}}
           \bigl(L_j(p_i)-F_P(p_i)\bigr)>0,
\]
with $\beta=+\infty$ if there are no such pairs. Choose $0<\varepsilon<\min(\alpha,\beta)$ and lower exactly the planes indexed by $J$ by $\varepsilon$.

Selected planes stay positive and boundary-cell planes remain fixed. If $A_i\cap J$ is empty, $\beta$ keeps the shifted planes above the old minimum. Otherwise at least two active planes shift equally and remain tied. Thus all corners and boundary values survive, but the roof drops by at least $\varepsilon$ on a selected cell. Minimality gives the contradiction, so the vectors span and $g\le N$.
\end{proof}
The inequality $g\le N$ requires no genericity. To prove the reverse inequality, we use independent motion of the marks: the roof has at most $g$ parameters, and each mark in an edge interior contributes one more.

\needspace{5\baselineskip}\begin{proposition}[The dimension count]\label{prop:dimension}

For a fixed embedded type $\tau$ with $g$ bounded cells, its realization space has dimension at most $g$. If $e$ labelled marks lie in edge interiors and all remaining marks lie at vertices, the corresponding marked incidence space has dimension at most $g+e$.
\end{proposition}
\begin{proof}
Let $\mathcal M_\tau$ consist of canonical coefficient vectors for nonnegative zero-boundary roofs of the unmarked type. We add the marks when forming the incidence space. Appendix~\ref{sec:C} describes the faces by affine equations and inequalities. The canonical coefficient condition requires each allowed plane to touch the roof somewhere in $\Delta$, an existential affine condition. Hence $\mathcal M_\tau$ is semilinear.

A cell meeting the boundary at $z$ has coefficient $c_m=-m\cdot z$. At a side-interior contact, nonnegativity makes its plane vanish along the side; at a vertex the same equation fixes the coefficient. Thus the type fixes all boundary-cell coefficients.

The $g$ bounded-cell coefficients and fixed boundary planes determine the essential roof, hence also its completion. Projection to those $g$ coordinates is injective and semilinear. Proposition~\ref{prop:semilinear} gives $\dim\mathcal M_\tau\le g$.

For a fixed roof, a mark in a specified open edge contributes one scalar parameter, and a mark at a vertex contributes none. The affine equalities and strict inequalities of Appendix~\ref{sec:C} specify the face containing each mark. The fiber-dimension inequality gives $\dim\mathcal I_\tau\le g+e$. Requiring the marks to be distinct can only decrease the dimension.

A strip meeting the polygon boundary is a boundary cell. Its affine coefficient is already fixed. The endpoints of its carrier are intersections of the affine lines determined by the same coefficients and the fixed boundary, so the carrier's length is determined as well.
\end{proof}

Apply this dimension bound to the open set of source configurations supplied by strong genericity. That set has dimension $2N$, so $g+N\ge2N$. Together with $g\le N$, this gives the cell count.

\needspace{5\baselineskip}\begin{theorem}[One bounded cell per generic source]\label{thm:genus}

For strongly generic $P$ of cardinality $N$, the number of bounded cells is $g=N$. The $N$ binary incidence vectors in \ref{lem:binary} form a basis of $\mathbf F_2^N$.
\end{theorem}
\begin{proof}
Let $P$ be strongly generic. Its marked embedded type $\tau$ persists on a small open box $V$ in the $2N$ source coordinates. Choose an open polygon $O\Subset\Delta^\circ$ containing these configurations. The graph of the canonical coefficient map over $V$ is semilinear by Appendix~\ref{sec:C} and projects bijectively to $V$, so it has dimension $2N$.

This graph lies in the marked incidence space of type $\tau$. Each of its $N$ marks is in a prescribed edge interior. Proposition~\ref{prop:dimension} and Lemma~\ref{lem:binary} give
\begin{equation}\label{eq:generic-genus-count}
 2N\le\dim\mathcal I_\tau\le g+N,\qquad g\le N.
\end{equation}
Therefore $g=N$, and the $N$ spanning incidence vectors form a basis. The argument uses the stable neighborhood itself, so it also covers points on boundaries of the chosen parameter cells.
\end{proof}
The basis in Theorem~\ref{thm:genus} has a graph interpretation: its vectors are the reduced incidence columns of the marked dual edges. To justify this interpretation, we first identify the exterior face of a finite connected truncation of the curve.

\needspace{5\baselineskip}\begin{lemma}[Connectivity and terminal branches]\label{lem:connectivity}

For a nonzero nonnegative zero-boundary tropical polynomial on $\Delta$, the full curve is connected and has no interior leaves. Its facet-adjacency graph is connected. A terminal branch exists, and each terminal edge has exactly one interior endpoint. An interior edge cannot have the same bounded cell on both sides.
\end{lemma}
\begin{proof}
Fix the polynomial $F$ in the statement. To prove connectivity of its facet-adjacency graph, join interior points of any two cells by a polygonal path avoiding vertices and crossing edges transversely. The successive cells along this path give a path in the adjacency graph.

For a closed cell $C$, its essential plane $L_C$ is nonnegative. Its boundary contact is the zero face of $L_C$ on $\Delta$: empty, a vertex, or a segment in one side. It is proper, since otherwise $L_C=0$ and $F=0$. Thus $\partial C\cap\Delta^\circ$ is connected, being a boundary circle with at most one proper face removed. Adjacent cells share an interior edge. Their connected boundaries therefore have connected union, the corner locus.

At every interior vertex, the successive gradient differences telescope to zero. Rotating them by ninety degrees gives the balancing relation
\[
 \sum_{\eta\in\mathcal E_\Gamma(v)}w(\eta)v_\eta=0,
\]
with outgoing primitive tangents and edge-germ multiplicity. One nonzero germ cannot balance, so there is no interior leaf.

Without a terminal branch the curve would be compactly interior, leaving a connected boundary collar in one cell. Its plane vanishes along the whole polygon boundary, so is zero. A nonnegative concave function vanishing on an interior open set is identically zero, a contradiction.

If an edge had both endpoints on $\partial\Delta$, its open chord would be a connected component of the curve. Connectivity would make it the entire curve. The resulting two affine pieces would have to cover $\partial\Delta$ by their zero sets. Each zero set is a proper exposed face of $\Delta$, and two such faces cannot cover a polygon with nonempty interior. Hence every terminal edge has exactly one interior endpoint.

Finally, an open cell is convex. If the same bounded cell occurred on both sides of an edge, a segment between two nearby points in that cell would cross the edge while staying in the cell. This is impossible.
\end{proof}

\needspace{5\baselineskip}\begin{lemma}[Planar separation]\label{lem:separation}

Let $G$ be a finite plane multigraph, with transverse intersections treated as vertices. If two faces lie on opposite sides of an embedded simple cycle $C\subset G$ (including a two-edge cycle), every path between them in the planar dual crosses an edge of $C$. Further, if $G\subset\Omega$, $R$ is a bounded face, and a finite embedded tree $T\subset\overline\Omega$ meets $G$ only at one attachment point $a$, with every other leaf on $\partial\Omega$, then $T\cap R=\varnothing$.
\end{lemma}
\begin{proof}
An embedded simple cycle is a Jordan curve, including a two-edge cycle with distinct embedded arcs. A path in the planar dual from one side to the other must cross it. Draw each dual edge to meet the primal graph only in the interior of its corresponding edge; the asserted crossing follows.

Suppose now that $T$ enters $R$. Subdivide at $a$ and close a component of $T\cap R$ inside $T$, obtaining a nontrivial tree $K$. Since $\partial R\subset G$, it meets $\partial R$ only at $a$. It has another leaf $b\in R$. Every incident edge of $T$ starts in $R$ and hence belongs to $K$, making $b$ an interior leaf of $T$, a contradiction.
\end{proof}

\needspace{9\baselineskip}\subsection*{The finite graph and its exterior face}\label{app:primal-truncation}
To apply planar graph duality, we shorten the terminal branches and obtain a finite graph whose boundary regions belong to one exterior face. Take $P\subset\Delta^\circ$ strongly generic and $N=|P|\ge1$. Choose a convex polygon $\Delta'\Subset\Delta^\circ$ containing the marks, all interior tropical vertices, and the closures of all bounded cells. Its boundary avoids the vertices and marks and crosses every terminal edge once. The resulting finite plane graph
\[
 G_0=\overline{\Gamma_P\cap(\Delta')^\circ}
\]
has a degree-one vertex at every crossing with $\partial\Delta'$. Keep transverse nodes as vertices and suppress artificial degree-two marks; each $p_i$ remains a distinguished point in its carrier $e_i$. The carriers are pairwise distinct.

The truncation only shortens leaf tails, so $G_0$ remains connected by Lemma~\ref{lem:connectivity}. Bounded cells survive; boundary cells open to the exterior through $\partial\Delta'$, including those meeting $\partial\Delta$ only at a vertex. Thus $G_0$ has $N$ bounded faces and $\beta_1(G_0)=N$. Its dual vertices are $f_\infty,f_1,\ldots,f_N$, with all boundary cells represented by $f_\infty$.

At a marked carrier, $a_i$ is the incidence column of $e_i^\vee$ with the exterior row deleted: $\mathbf e_r+\mathbf e_s$ between bounded faces, $\mathbf e_r$ between a bounded face and the exterior, and zero between exterior faces. Lemma~\ref{lem:connectivity} excludes self-incidence of a bounded face. Figure~\ref{fig:marked-trees} shows the dual tree, deletion of the marked carriers, and cutting at the marks.

\needspace{5\baselineskip}\begin{theorem}[The marked dual tree and the cut curve]\label{thm:marked-tree}

For strongly generic $P$, the dual edges crossing its marked carrier edges form a spanning tree of the face-dual graph, including the exterior face. Terminal edges are unmarked. Deleting the uncut marked carriers leaves a primal spanning tree. Cutting at the marks and completing the resulting ends gives a finite tree $T_P$; its open interior part is connected and acyclic.
\end{theorem}
\begin{proof}
\emph{The marked edges in the dual.} By Theorem~\ref{thm:genus}, the reduced incidence columns of $e_1^\vee,\ldots,e_N^\vee$ are a basis of $\mathbf F_2^N$. A graph on $v$ vertices and $c$ components has incidence rank $v-c$ over $\mathbf F_2$: vectors in the left kernel are exactly those constant on each component. Deleting the exterior row leaves that rank unchanged, since all rows sum to zero. Thus the marked dual subgraph is connected on $N+1$ vertices. With $N$ edges it is a spanning tree, denoted $S^\vee$.

A terminal carrier would be a leaf edge of $G_0$, hence a bridge. Its dual would be a loop with zero incidence column, which the basis excludes. Consequently every marked carrier is compactly contained in $\Delta^\circ$.

\emph{Delete the carriers.} Put $S=\{e_1,\ldots,e_N\}$ and
$T_0=(V(G_0),E(G_0)\setminus S)$. Euler's formula gives
\begin{equation}\label{eq:tree-complement-count}
 |E(T_0)|=|E(G_0)|-N=|V(G_0)|-1.
\end{equation}
If $T_0$ contained a cycle, it would contain an embedded simple cycle $C$. The dual tree $S^\vee$ connects a face inside $C$ to the exterior. Lemma~\ref{lem:separation} forces this path to cross an edge of $C$. But its crossed edges are in $S$, and every edge of $C$ lies outside $S$. This contradiction proves acyclicity. A spanning forest with the edge count \eqref{eq:tree-complement-count} is connected, so $T_0$ is a tree.

\emph{Cut at the marks.} Restore each deleted carrier as two half-edges, attached at its original endpoints and ending at distinct formal leaves $p_i^+$ and $p_i^-$. Attaching leaves to $T_0$ produces a tree $T_P$. Retain the terminal ends as formal leaves, even when distinct branches approach the same boundary point of $\Delta$. Extending their tails from $\partial\Delta'$ toward $\partial\Delta$ changes no incidence. Removing the formal leaf vertices gives a connected acyclic open graph, precisely the curve cut at its marks and boundary.

\end{proof}
The planar tree--cotree argument is standard graph topology; see~\cite{Die17}. The cut tree will give the Euler count in Section~\ref{sec:4}. Deleting a marked dual edge will determine the block of faces lowered in Lemma~\ref{lem:block}.

\needspace{5\baselineskip}\begin{corollary}[The compact core]\label{cor:compact-core}
For a nonempty strongly generic $P$, let $H_P\subset T_P$ be the minimal subtree joining the $2N$ marked ends. Its image lies compactly inside $\Delta^\circ$. Identifying each marked pair gives a connected graph $\Gamma_P^{\rm core}$ with
\begin{equation}\label{eq:core-betti}
 \chi(\Gamma_P^{\rm core})=1-N,\qquad
 \beta_1(\Gamma_P^{\rm core})=N.
\end{equation}
This core contains every cycle of the original curve. The remaining unmarked trees attach to it once and have their other leaves on $\partial\Delta$.
\end{corollary}
\begin{proof}
A terminal leaf cannot lie on a path between marked ends, so the image of $H_P$ lies compactly inside $\Delta^\circ$. Each identification $p_i^+\sim p_i^-$ merges two distinct vertices, preserves connectedness, and reduces the Euler characteristic by one. This gives \eqref{eq:core-betti}.

Cutting any cycle of the original curve at its marks produces paths between marked ends. Every such path lies in $H_P$, by uniqueness of paths in $T_P$; a cycle with no mark would contradict that $T_P$ is a tree. Thus the core contains every cycle. Each component outside $H_P$ attaches once, since two attachments would create a cycle. All marked leaves lie in $H_P$, and there are no interior leaves, so its other leaves are terminal. Reidentifying marked pairs preserves these attachments.
\end{proof}

\begin{figure}[tbp]
\centering
\begingroup
\newcommand{\trcoordinates}{
 \coordinate (a) at (-2,1);\coordinate (b) at (0,1);\coordinate (c) at (2,1);
 \coordinate (d) at (-2,-1);\coordinate (e) at (0,-1);\coordinate (f) at (2,-1);
 \coordinate (aa) at (-3,2);\coordinate (bb) at (0,2);\coordinate (cc) at (3,2);
 \coordinate (dd) at (-3,-2);\coordinate (ee) at (0,-2);\coordinate (ff) at (3,-2);
 \path[use as bounding box] (-3.4,-3.65) rectangle (3.4,2.95);
}
\newcommand{\trinternal}[1]{\draw[#1] (a)--(b)--(c)--(f);\draw[#1] (a)--(d)--(e);}
\newcommand{\trterminal}[1]{\draw[#1] (a)--(aa) (b)--(bb) (c)--(cc) (d)--(dd) (e)--(ee) (f)--(ff);}
\newcommand{\trvertices}[1]{\foreach \v in {a,b,c,d,e,f}\fill[#1] (\v) circle (1.4pt);}
\newcommand{\trends}[1]{\foreach \v in {aa,bb,cc,dd,ee,ff}\draw[#1,fill=white] (\v) circle (1.8pt);}
\begin{minipage}{.48\textwidth}\centering
\begin{tikzpicture}[x=.9cm,y=.8cm,font=\footnotesize]
\trcoordinates
\fill[Geometry!8] (-2,-1) rectangle (0,1);
\fill[Algebra!9] (0,-1) rectangle (2,1);
\trinternal{Geometry,line width=1pt}\trterminal{Geometry,line width=1pt}
\draw[Probability,line width=1.5pt] (b)--(e)--(f);
\trvertices{Ink}\trends{Geometry}
\fill[Probability] (0,0) circle (2.3pt);\fill[Probability] (1,-1) circle (2.3pt);
\node[Probability,right] at (.12,.1) {$p_1$};\node[Probability,below] at (1,-1.1) {$p_2$};
\node at (-1,0) {$f_1$};\node at (1,.35) {$f_2$};
\node[font=\small\bfseries] at (0,2.65) {(a) The uncut graph $G_0$};
\node at (0,-2.65) {$V=12,\quad E=13,\quad\beta_1=2$};
\node at (0,-3.15) {A plane-graph schematic.};
\end{tikzpicture}
\end{minipage}\hfill
\begin{minipage}{.48\textwidth}\centering
\begin{tikzpicture}[x=.9cm,y=.8cm,font=\footnotesize]
\trcoordinates
\trinternal{Ink!45,line width=.8pt}\trterminal{Ink!45,line width=.8pt}
\draw[Ink!45,line width=.8pt] (b)--(e)--(f);\trvertices{Ink!55}\trends{Ink!45}
\draw[Algebra,line width=1.6pt,dashed] (-1,0)--(1,0)--(1,-2.55);
\foreach \x/\y in {-1/0,1/0,1/-2.55}\fill[Algebra] (\x,\y) circle (2.5pt);
\node[above,fill=white,inner sep=1pt] at (-1,.08) {$f_1$};
\node[above,fill=white,inner sep=1pt] at (1,.08) {$f_2$};
\node[right,fill=white,inner sep=1pt] at (1.08,-2.55) {$f_\infty$};
\node[Algebra,above] at (0,0) {$e_1^\vee$};
\node[Algebra,right,fill=white,inner sep=1pt] at (1,-1.55) {$e_2^\vee$};
\node[font=\small\bfseries] at (0,2.65) {(b) The marked dual tree};
\node at (-1.1,-3.2) {$[a_1\ a_2]=\begin{pmatrix}1&0\\1&1\end{pmatrix}$};
\end{tikzpicture}
\end{minipage}

\vspace{4pt}
\begin{minipage}{.48\textwidth}\centering
\begin{tikzpicture}[x=.9cm,y=.8cm,font=\footnotesize]
\trcoordinates
\trinternal{Geometry,line width=1.2pt}\trterminal{Geometry,line width=1.2pt}
\trvertices{Ink}\trends{Geometry}
\node[font=\small\bfseries] at (0,2.65) {(c) The complement $T_0$};
\node at (0,-2.65) {$V=12,\quad E=11,\quad\beta_1=0$};
\end{tikzpicture}
\end{minipage}\hfill
\begin{minipage}{.48\textwidth}\centering
\begin{tikzpicture}[x=.9cm,y=.8cm,font=\footnotesize]
\trcoordinates
\trterminal{Ink!45,line width=.8pt}\trends{Ink!40}
\trinternal{Geometry,line width=1.3pt}\trvertices{Ink}
\draw[Probability,line width=1.4pt] (b)--(0,.19) (e)--(0,-.19) (e)--(.82,-1) (f)--(1.18,-1);
\foreach \x/\y in {0/.19,0/-.19,.82/-1,1.18/-1}\draw[Probability,fill=white,line width=.8pt] (\x,\y) circle (2pt);
\node[Probability,right] at (.08,.32) {$p_1^+$};\node[Probability,right] at (.08,-.3) {$p_1^-$};
\node[Probability,below] at (.60,-1.08) {$p_2^+$};\node[Probability,below] at (1.53,-1.08) {$p_2^-$};
\node[font=\small\bfseries] at (0,2.65) {(d) The completed cut tree $T_P$};
\node at (0,-2.65) {$V=16,\quad E=15,\quad\beta_1=0$};
\end{tikzpicture}
\end{minipage}
\endgroup
\caption{A schematic plane graph illustrating Theorem~\ref{thm:marked-tree} and Corollary~\ref{cor:compact-core}. The two marked carriers have independent reduced dual incidence columns, and their dual edges form a tree. Deleting the carriers leaves the primal tree $T_0$; restoring them as separate half-edges gives $T_P$. Its colored subtree $H_P$ joins the marked ends, and its gray tails lead to the boundary. The gaps distinguish formal cut ends. Identifying each marked pair restores one cycle.}
\label{fig:marked-trees}
\end{figure}

\needspace{12\baselineskip}\section{\texorpdfstring{\headingcolor{Geometry}{Counting curvature}}{Counting curvature}}\label{sec:4}

The curvature of a tropical roof is the sum of the areas of the polygons spanned by its active gradients at vertices. We compare this area with the number of prescribed sources. The tree from Section~\ref{sec:3} makes Euler's formula a count of sources, and Pick's formula turns the gradient-polygon areas into lattice counts.

Work with a strongly generic cloud on a rational-slope polygon. Choose an interior region $U\Subset\Omega$ whose boundary avoids vertices and marks and crosses edges transversely. Such regions, with the finiteness conditions specified below, are called regular windows. Let $B_P(U)$ count boundary crossings, and let $c_{\rm cut}(U)$ count components of the curve inside $U$ after cutting at the marks in $U$. Pick's formula has a nonnegative extra-area term, contributed by interior lattice points of the Newton polygons and by edge weights greater than one. Denote its sum over vertices in $U$ by $\Xi_P(U)$, the lattice surplus. Euler's formula for the cut graph and Pick's formula give
\[
 \MA(F_P)(U)=|P\cap U|+\tfrac12B_P(U)-c_{\rm cut}(U)+\Xi_P(U).
\]
Each cut raises the Euler characteristic by one, accounting for the coefficient of the source count. Every cut component reaches the window boundary, so the crossing count also bounds the component term.

The essential estimate concerns the surplus. An interior lattice point in a Newton polygon or an internal edge of weight greater than one supplies a supporting plane with integral gradient that can be lowered. At a marked edge, the dual tree determines which adjacent planes must move with it to preserve the other corners. Minimality excludes these contributions. The remaining terminal surplus is controlled by the total side quasidegree, giving both the local discrepancy and the total-mass identity.

\needspace{5\baselineskip}\begin{definition}[Windows, vertices, and surplus]\label{def:window}

A regular window $U$ is a relatively compact open set with finitely many components, each bounded by finitely many piecewise $C^1$ Jordan curves that are disjoint within that component. Its boundary avoids vertices and marks and meets edges transversely at finitely many edge-interior points, all at $C^1$ points of the boundary. The empty window is allowed.

The truncated graph is the disjoint union of the graphs obtained in the components of $U$; each boundary incidence gives a separate endpoint, even when endpoints coincide geometrically. Include every interval when an edge leaves and reenters the window, including intervals with no tropical vertex. Write $B_F(U)$ for the number of crossing edge germs, $c_F(U)$ for the number of components of this truncated graph, and $\beta_F(U)$ for its first Betti number.

For an interior vertex $v$, let $Q_v=\partial^+F(v)$ be its Newton polygon and $d(v)$ its edge-germ valence. Define
\[
 \sigma_F(v)=\operatorname{Area}(Q_v)-\frac{d(v)-2}{2},
 \qquad \Xi_F=\sum_v\sigma_F(v)\delta_v.
\]
All areas are ordinary Euclidean areas.
\end{definition}

If the adjacent cell gradients differ by $w n$, with $n\in\Z^2$ primitive, the edge has weight $w$. A primitive tangent is a ninety-degree rotation of $n$. At a vertex we count incident germs, including both germs of a transverse branch. Open edges carry no area in slope space. Appendix~\ref{sec:B} identifies the polygon of supporting slopes with the convex hull of the active slopes and gives
\[
 \MA(F)=\sum_v\operatorname{Area}(Q_v)\delta_v.
\]

For distinct edge-interior marks, cut each carrier at every point of $P\cap U$ and keep its two new ends separate. Write $c_{\rm cut}(U)$ and $\beta_{\rm cut}(U)$ for the resulting counts. The symbol $B_F(U)$ continues to count only window-boundary incidences. For a least roof, write $B_P=B_{F_P}$ and $\Xi_P=\Xi_{F_P}$; $c_{\rm cut}(U)$ continues to count components after cutting. All counts vanish for the empty window.

\needspace{5\baselineskip}\begin{proposition}[The surplus in Pick\textquotesingle s formula]\label{prop:pick}

If $I(Q_v)$ is the number of interior lattice points of $Q_v$, then
\[
 \sigma_F(v)=I(Q_v)+\frac12\sum_{e\ni v}(w(e)-1)\ge0.
\]
The sum is over incident edge germs. In particular a unimodular triangle contributes Monge--Amp\`ere mass $1/2$, and a primitive lattice parallelogram of area one contributes mass $1$.
\end{proposition}
\begin{proof}
The boundary lattice count $b(Q_v)$ is the sum of the side lengths, each equal to its dual edge weight. Thus
\[
 b(Q_v)=\sum_{e\ni v}w(e),\qquad \#\{\text{sides of }Q_v\}=d(v).
\]
Pick's formula, with Euclidean area, is
\[
 \operatorname{Area}(Q_v)=I(Q_v)+\frac{b(Q_v)}2-1.
\]
Subtracting $(d(v)-2)/2$ gives the asserted surplus formula. For the two examples, a unimodular triangle has $I=0$, $b=d=3$, and mass $1/2$; a unit lattice parallelogram has $I=0$, $b=d=4$, and mass one. Both have zero surplus.
\end{proof}
\emph{Pick\textquotesingle s formula with Euclidean area.}\label{app:pick-normalization}
Refine a lattice triangulation through every lattice point of the polygon. Each resulting triangle is empty and has area $1/2$. Indeed, translate one vertex to zero and call the other two $u,v$. If $|\det(u,v)|>1$, their half-open fundamental parallelogram contains a nonzero lattice representative $z=su+tv$, $0\le s,t<1$. Either $z$ or $u+v-z$ is a further lattice point of the triangle, a contradiction. Now let $T,E,V$ count its triangles, edges, and vertices, and put $b=b(Q)$, $I=I(Q)$. Euler's formula for the triangulated disk and the triangle--edge incidence count
\[
 V=I+b,\qquad 3T=2E-b,\qquad V-E+T=1
\]
give $T/2=I+b/2-1$. This proves Pick\textquotesingle s formula~\cite{Pic99} for the convex lattice polygons used here.

An empty Newton polygon can still have surplus. For example, the triangle
\[
 Q=\operatorname{conv}\{(0,0),(1,0),(-1,2)\}
\]
has area one and boundary lattice lengths $1,2,1$. Thus $I(Q)=0$ but $\sigma=1/2$. The weight-two side accounts for the difference. We shall have to control both the interior lattice points and the edge weights.

We first sum the valence term in Pick\textquotesingle s formula. This count depends only on the graph inside the window and on the number of crossings of its boundary.

\needspace{5\baselineskip}\begin{lemma}[Euler characteristic in a window]\label{lem:euler}

For a regular window,
\[
 \sum_{v\in U}\frac{d(v)-2}{2}
 =\beta_F(U)+\frac{B_F(U)}{2}-c_F(U).
\]
\end{lemma}
\begin{proof}
Form the finite truncated graph separately in every component of $U$. Let $V$ count its tropical vertices, $E$ its edge intervals, and $B=B_F(U)$ its degree-one boundary vertices. Each internal vertex keeps all its incident germs. The incidence count and the Euler characteristic are
\begin{equation}\label{eq:window-euler-counts}
 \sum_{v\in U}d(v)+B=2E,
 \qquad V+B-E=c_F(U)-\beta_F(U).
\end{equation}
Consequently
\[
 \sum_{v\in U}\frac{d(v)-2}{2}
 =E-\frac B2-V
 =\beta_F(U)+\frac B2-c_F(U).
\]
A vertex-free interval contributes $0=0+2/2-1$. The same count applies to windows with holes, with boundary incidences counted separately in each component.
\end{proof}

\needspace{5\baselineskip}\begin{theorem}[The exact source--curvature identity]\label{thm:window-balance}

For a tropical polynomial and a regular window,
\[
 \MA(F)(U)=\beta_F(U)+B_F(U)/2-c_F(U)+\Xi_F(U).
\]
Cutting at $n(U)=|P\cap U|$ marked edge-interior points gives
\[
 \MA(F)(U)-n(U)=\beta_{\rm cut}(U)-c_{\rm cut}(U)
                  +B_F(U)/2+\Xi_F(U).
\]
For a strongly generic least roof the cut graph is a forest, and each component reaches the window boundary. Hence $0\le c_{\rm cut}(U)\le B_P(U)$ and
\begin{equation}\label{eq:window-balance}
 \MA(F_{\Omega,P})(U)=|P\cap U|+B_P(U)/2-c_{\rm cut}(U)+\Xi_P(U),
\end{equation}
so $|\MA(F_{\Omega,P})(U)-|P\cap U|-\Xi_P(U)|\le B_P(U)/2$.
\end{theorem}
\begin{proof}
At every vertex,
$\operatorname{Area}(Q_v)=(d(v)-2)/2+\sigma_F(v)$.
Sum over the vertices in $U$ and apply Lemma~\ref{lem:euler}. This gives the first identity.

A cut replaces one edge interval by two and creates two formal endpoints. It therefore raises $V-E$, hence Euler characteristic, by one. After all $n(U)$ cuts,
\begin{equation}\label{eq:cut-euler-change}
 c_{\rm cut}(U)-\beta_{\rm cut}(U)
 =c_F(U)-\beta_F(U)+n(U).
\end{equation}
Substitution gives the marked identity. It applies to any finite tropical polynomial with distinct edge-interior marks.

Now let $F=F_P$ and assume strong genericity. For an empty cloud every term is zero, so take $P\ne\varnothing$. By Theorem~\ref{thm:marked-tree}, the completed marked cut graph is a tree. Its componentwise restriction to $U$ is a forest, so $\beta_{\rm cut}(U)=0$. To see why every component reaches the window boundary, start at one of its points and follow the unique path in the global tree toward a terminal end. Such an end exists by Lemma~\ref{lem:connectivity}, survives the marked cuts, and lies outside $\overline U$. The path must first leave the relevant window component through a retained boundary incidence. Distinct restricted components use disjoint incidences. Hence
\[
 0\le c_{\rm cut}(U)\le B_P(U).
\]
The identity now reads \eqref{eq:window-balance}. Its remaining signed term is $B_P(U)/2-c_{\rm cut}(U)$, whose absolute value is at most $B_P(U)/2$.
\end{proof}
The term $|P\cap U|$ comes from the increase in Euler characteristic at the cuts; the Monge--Amp\`ere mass remains supported at the tropical vertices. Figure~\ref{fig:window-pick} shows the calculation for one vertex and one edge-interior mark.

\begin{figure}[htbp]
\centering
\begin{minipage}{.32\textwidth}\centering
\begin{tikzpicture}[x=1.03cm,y=1.03cm,font=\footnotesize]
\path[use as bounding box] (-1.55,-1.8) rectangle (1.55,1.8);
\draw[Ink!60,dashed] (0,0) circle (1.2);
\draw[Geometry,line width=1pt] (0,1.2)--(0,0)--(1.2,0) (0,0)--(-.8485,-.8485);
\foreach \x/\y in {0/1.2,1.2/0,-.8485/-.8485}\draw[Geometry,fill=white] (\x,\y) circle (2pt);
\fill[Analysis] (0,0) circle (2.6pt);\node[Analysis,left] at (-.07,.12) {$v$};
\fill[Probability] (.6,0) circle (2.3pt);\node[Probability,above] at (.6,.05) {$p$};
\node[font=\small\bfseries,align=center] at (0,1.70) {(a) One vertex\\and one mark};
\node at (0,-1.5) {$B=3,\quad c=1,\quad\beta=0$};
\end{tikzpicture}
\end{minipage}\hfill
\begin{minipage}{.32\textwidth}\centering
\begin{tikzpicture}[x=1.03cm,y=1.03cm,font=\footnotesize]
\path[use as bounding box] (-1.55,-1.8) rectangle (1.55,1.8);
\filldraw[fill=Analysis!12,draw=Analysis,line width=1pt] (-.5,-.35)--(.7,-.35)--(-.5,.85)--cycle;
\foreach \x/\y in {-.5/-.35,.7/-.35,-.5/.85}\fill[Analysis] (\x,\y) circle (2pt);
\node[below left] at (-.5,-.35) {$0$};\node[below right] at (.7,-.35) {$(1,0)$};
\node[above] at (-.5,.9) {$(0,1)$};
\node at (-.17,.02) {$Q_v$};
\node[font=\small\bfseries] at (0,1.62) {(b) Its Newton polygon};
\node at (0,-1.5) {$\operatorname{Area}=\tfrac12,\quad\sigma=0$};
\end{tikzpicture}
\end{minipage}\hfill
\begin{minipage}{.32\textwidth}\centering
\begin{tikzpicture}[x=1.03cm,y=1.03cm,font=\footnotesize]
\path[use as bounding box] (-1.55,-1.8) rectangle (1.55,1.8);
\draw[Ink!60,dashed] (0,0) circle (1.2);
\draw[Geometry,line width=1pt] (0,1.2)--(0,0)--(.5,0) (0,0)--(-.8485,-.8485);
\draw[Probability,line width=1pt] (.7,0)--(1.2,0);
\foreach \x/\y in {0/1.2,1.2/0,-.8485/-.8485}\draw[Ink,fill=white] (\x,\y) circle (2pt);
\fill[Analysis] (0,0) circle (2.6pt);
\draw[Probability,fill=white] (.5,0) circle (2pt);\draw[Probability,fill=white] (.7,0) circle (2pt);
\node[Probability,above left] at (.48,.05) {$p^-$};\node[Probability,below right] at (.72,-.05) {$p^+$};
\node[font=\small\bfseries] at (0,1.62) {(c) Cut at the mark};
\node at (0,-1.5) {$B=3,\quad c_{\rm cut}=2$};
\end{tikzpicture}
\end{minipage}
\[
 \frac12=\underbrace{1}_{\text{sources}}
 +\underbrace{\frac32}_{\text{half the crossings}}
 -\underbrace{2}_{\text{cut components}}
 +\underbrace{0}_{\text{surplus}}.
\]
\caption{The local model $F=\min\{0,x,y\}$, up to an added constant, has a trivalent vertex of mass $1/2$ and primitive edge germs. Cutting the mark creates two components while keeping the three window incidences. The small gap represents two formal ends at the same marked point; it changes no boundary count. The example concerns the local identity, which applies to any finite tropical polynomial.}
\label{fig:window-pick}
\end{figure}

\subsection*{Lowering a face while keeping the marks}
The window identity leaves two errors: boundary crossings and the surplus $\Xi_P$. An interior lattice point in a Newton face gives another integral supporting plane. If its contact set is compactly contained in the domain and avoids the marks, lowering that plane would produce a smaller admissible roof. Minimality will therefore exclude such lattice points.

\needspace{5\baselineskip}\begin{lemma}[A local lowering operation]\label{lem:lowering}

An integral point in the relative interior of a Newton face determines, by convex interpolation of the incident affine functions, an integral-gradient affine function lying above $F$ and contacting it exactly on the dual face. If that contact face is compactly contained in $\Omega$, lowering this affine function by a sufficiently small positive amount changes the roof only in a chosen neighborhood, preserves nonnegativity and zero boundary values, and strictly lowers it somewhere.
\end{lemma}
\begin{proof}
Let $m_1,\ldots,m_r$ be the vertices of the Newton face $Q$, and $L_1,\ldots,L_r$ their essential supporting planes. The chosen integral point $q\in\operatorname{relint}Q$ has a convex representation with all vertex coefficients positive:
\[
 q=\sum_{j=1}^r\lambda_jm_j,\qquad
 \lambda_j>0,\qquad\sum_j\lambda_j=1.
\]
Set $L_q=\sum_j\lambda_jL_j$. Its gradient is the integral vector $q$. Each difference $L_j-F$ is nonnegative, so
\begin{equation}\label{eq:lowering-contact}
 L_q\ge F,\qquad
 \{L_q=F\}=\bigcap_j\{L_j=F\}=\tau_Q.
\end{equation}
Here $\tau_Q$ is the closed tropical face dual to $Q$; the contact-set identity is Proposition~\ref{prop:essential}.

Suppose $\tau_Q\Subset\Omega$, and choose open sets
$\tau_Q\subset W_0\Subset W\Subset\Omega$. The continuous gap $L_q-F$ vanishes exactly on $\tau_Q$, hence
\[
 \eta=\min_{\Delta\setminus W_0}(L_q-F)>0.
\]
The roof is nonzero in this situation. A nonnegative concave roof which vanishes at an interior point vanishes everywhere, so
$a=\min_{\overline W_0}F>0$. For $0<\varepsilon<\min(\eta,a)$ put
\begin{equation}\label{eq:localized-lowering}
 F_{\varepsilon,q}=\min\{F,L_q-\varepsilon\}.
\end{equation}
Outside $W_0$ the new plane is above $F$. Inside $W_0$ it is at least $F-\varepsilon\ge a-\varepsilon>0$. Thus the new roof is nonnegative and agrees with $F$ near the boundary and outside $W$. On $\tau_Q$ it equals $F-\varepsilon$; continuity makes the inequality strict in a neighborhood. The minimum remains a finite tropical polynomial.
\end{proof}
In the canonical presentation $L_q$ may already touch only on $\tau_Q$. Formula~\eqref{eq:localized-lowering} lowers that coefficient; its contact set is determined by the essential presentation.

At a marked edge, a local lowering can erase the prescribed corner. If the edge has weight at least two, delete its edge from the marked dual tree and lower the planes on the side away from the exterior face. Every other marked pair then moves together or stays fixed. A plane with an intermediate lattice slope keeps a tie at the chosen mark (Figure~\ref{fig:coordinated-lowering}).

The next lemma is for a strongly generic least roof. In its formula, $v$ denotes the gradient of the selected adjacent cell and $u$ the gradient of the unselected cell.
\needspace{5\baselineskip}\begin{lemma}[Simultaneous lowering at a marked edge]\label{lem:block}

Delete the dual edge corresponding to a chosen marked carrier. The component not containing the exterior vertex determines a nonempty collection of bounded cells. The chosen edge is its unique marked crossing. The affine functions of these cells are strictly positive on $\overline\Omega$.

If adjacent gradients satisfy $v-u=mn$ with primitive $n$ and $m\ge2$, insert
\[
 q=v-n,\qquad L_q=\frac1mL_u+\frac{m-1}{m}L_v.
\]
This function has integral gradient, contacts the roof on the chosen closed carrier, and is positive on $\overline\Omega$. For each fixed configuration there are constants $\alpha,\gamma>0$. Lowering the selected functions and $L_q$ by $0<\varepsilon<\min\{\alpha,\gamma\}$ preserves nonnegativity, the boundary data, and every prescribed corner.
\end{lemma}
\begin{proof}
\emph{Choose the bounded block.} Let $e=e_i$ be the chosen marked carrier and delete its dual edge $e_i^\vee$ from the marked dual tree (Theorem~\ref{thm:marked-tree}). The component away from the exterior selects a nonempty set $\mathcal R_e$ of bounded cells. The deleted edge is the unique marked crossing, so every other marked pair is either wholly selected or wholly unselected. In Figure~\ref{fig:marked-trees}(b), deleting $e_2^\vee$ selects $f_1,f_2$, including both sides of $e_1$.

The essential planes $L_R$, $R\in\mathcal R_e$, have positive minima on $\Delta$ by Lemma~\ref{lem:binary}. Denote by $v$ the gradient of the selected cell adjacent to $e$, and by $u$ the unselected gradient. Assume $v-u=mn$, with $n$ primitive and $m\ge2$. Then
\begin{equation}\label{eq:primitive-intermediate}
 \nabla L_q=\frac{u+(m-1)v}{m}
 =v-\frac{v-u}{m}=v-n=q,
 \qquad v-q=n.
\end{equation}
The positive convex combination defining $L_q$ lies above $F_P$. Equality holds precisely where both adjacent planes are active, namely on the closed carrier $e$. By Theorem~\ref{thm:marked-tree}, $e\Subset\Omega$. The roof is positive on $e$, and $L_q>F_P\ge0$ everywhere else in $\Delta$, so $\min_\Delta L_q>0$.

Suppose an essential plane $M_q$ had gradient $q$. Write $M_q=L_q+c$. Evaluation on $e$ gives $c\ge0$. If $c>0$, then $L_q<F_P$ on the cell of $M_q$, contrary to $L_q\ge F_P$. If $c=0$, the contact set of $M_q$ is $e$, so $M_q$ is not essential. Thus no essential plane has gradient $q$.

\emph{Preservation of the prescribed corners.} Set
$\mathcal S_e=\{L_R:R\in\mathcal R_e\}\cup\{L_q\}$.
The terms that must remain tied are as follows:
\begin{center}\small
\begin{tabularx}{\textwidth}{@{}P{.42\textwidth}X@{}}
\toprule Mark & Planes remaining minimal\\\midrule
The chosen $p_i$ & $L_v,L_q$, both lowered\\[3pt]
Another mark with both adjacent cells selected & Its two adjacent essential planes, both lowered\\[3pt]
Another mark with both adjacent cells unselected & Its two adjacent essential planes, both fixed\\\bottomrule
\end{tabularx}
\end{center}
Each has exactly two essential active planes, and $L_q$ touches only the chosen carrier, which has no other mark. Every shifted term outside the designated pair has a positive gap. Set
\[
 \alpha=\min_{L\in\mathcal S_e}\min_\Delta L>0,
 \qquad
 \gamma=\min\{L(p_j)-F_P(p_j):
 L\in\mathcal S_e\text{ is not designated at }p_j\}>0,
\]
using any positive $\gamma$ if the set of gaps is empty. These constants depend on the fixed configuration.

For $0<\varepsilon<\min(\alpha,\gamma)$ define
\begin{equation}\label{eq:block-lowering}
 E_\varepsilon=\min\bigl(
 \{L_R-\varepsilon:R\in\mathcal R_e\}
 \cup\{L_q-\varepsilon\}
 \cup\{L_C:C\notin\mathcal R_e\}\bigr).
\end{equation}
The lowered planes stay positive; fixed planes stay nonnegative. Every boundary point has a zero essential plane, belonging to an unselected boundary cell. Hence $E_\varepsilon=0$ on $\partial\Delta$.

At $p_i$ the distinct-gradient pair $L_v-\varepsilon,L_q-\varepsilon$ is tied at $F_P(p_i)-\varepsilon$. All other shifted terms keep their positive gaps; fixed terms have value at least $F_P(p_i)$. The new local carrier has weight one by \eqref{eq:primitive-intermediate}. At a mark with both adjacent cells selected, the same comparison applies:
\[
 (L(p_j)-\varepsilon)-(F_P(p_j)-\varepsilon)
 =L(p_j)-F_P(p_j)\ge\gamma
\]
for every other shifted term. At an unselected pair, the common value of the two unshifted planes is unchanged, while every shifted term satisfies
\[
 L(p_j)-\varepsilon-F_P(p_j)\ge\gamma-\varepsilon>0.
\]
The designated pair at each mark therefore still attains the minimum. The roof is admissible and every prescribed corner survives.
\end{proof}

\begin{figure}[tbp]
\centering
\begin{minipage}[t]{.49\textwidth}\centering
\begin{tikzpicture}[x=1cm,y=1cm,font=\footnotesize]
 \path[use as bounding box] (-2.5,-1.5) rectangle (2.8,1.6);
 \fill[Geometry!10,rounded corners=8pt] (-.4,-.55) rectangle (2.55,1.15);
 \draw[Ink,thick] (-2,0)--(-1,0);
 \draw[Probability,thick,dashed] (-1,0)--(.35,0);
 \draw[Geometry,very thick] (.35,0)--(1.8,.6) (.35,0)--(1.8,-.05);
 \foreach \x/\y in {-2/0,-1/0,.35/0,1.8/.6,1.8/-.05}{\fill[Ink] (\x,\y) circle (2pt);}
 \node[below] at (-2,-.05) {exterior};
 \node[below] at (-1,-.05) {$L_u$};
 \node[below] at (.35,-.05) {$L_v$};
 \node[above,Probability,align=center] at (-.45,.05) {delete the\\chosen edge};
 \node[above,Geometry] at (1.1,1.12) {lower the selected planes};
 \node[align=center] at (.15,-1.1) {Every other marked pair lies\\inside or outside the block.};
\end{tikzpicture}
\end{minipage}\hfill
\begin{minipage}[t]{.48\textwidth}\centering
\begin{tikzpicture}[x=2cm,y=1cm,font=\footnotesize,>=stealth]
 \path[use as bounding box] (-1.2,-2.95) rectangle (1.8,.9);
 \draw[->,Ink!60] (-1.05,0)--(1.6,0) node[right]{$s$};
 \draw[Ink,thick] (-1,-2)--(0,0)--(1.45,0);
 \draw[Probability,very thick] (-1,-2.65)--(0,-.65)--(.65,0)--(1.45,0);
 \draw[<->,Analysis] (0,-.60)--(0,-.05) node[midway,right] {$\varepsilon$};
 \fill[Ink] (0,0) circle (2.3pt);
 \fill[Probability] (0,-.65) circle (2.7pt);
 \fill[Probability] (.65,0) circle (2.7pt);
 \node[above left] at (0,.05) {$0$};
 \node[above right] at (.65,.05) {$\varepsilon$};
 \node[Probability,above] at (.85,.47) {new corner};
 \node[Ink,left] at (-.8,-1.52) {$F$};
 \node[Probability,below right] at (-.65,-2) {$E_\varepsilon$};
 \node[Probability,left] at (-.6,-1.85) {$2$};
 \node[Probability,below right] at (.28,-.37) {$1$};
 \node[Probability,below] at (1.15,-.04) {$0$};
 \node[align=center] at (.2,-2.85) {Lower-graph slopes: $2,1,0$.};
\end{tikzpicture}
\end{minipage}
\caption{Left: a schematic marked dual tree. Deleting the chosen edge selects the component away from the exterior face, whose planes are lowered together. Right: the exact transverse model $F(s)=\min\{0,2s\}$ and $E_\varepsilon(s)=\min\{0,2s-\varepsilon,s-\varepsilon\}$. At $s=0$, the lowered planes of slopes $2$ and $1$ retain a corner at height $-\varepsilon$; another corner appears at $s=\varepsilon$. Lemma~\ref{lem:block} shows that the global construction preserves the boundary values and all prescribed corners.}
\label{fig:coordinated-lowering}
\end{figure}

\needspace{5\baselineskip}\begin{theorem}[Interior lattice simplicity]\label{thm:mildness}

For a strongly generic least roof on a rational-slope polygon, every interior Newton polygon has no interior lattice points, and every compact internal edge has weight one. The unmarked-edge conclusion holds without genericity whenever the entire closed edge is compactly contained in $\Omega$ and disjoint from $P$.
\end{theorem}
\begin{proof}
Strong genericity puts marks inside edges. A lattice point in the interior of $Q_v$ would, by Lemma~\ref{lem:lowering}, allow a lowering near $v$ away from all marks. Minimality forbids it.

If a closed edge $e\Subset\Omega$ avoids $P$ and has weight at least two, its dual segment has an interior lattice point. Lower near $e$, away from the marks, to get the same contradiction.

Finally suppose a marked carrier had weight $m\ge2$. Lemma~\ref{lem:block} supplies an admissible roof $E_\varepsilon$. The selected block is nonempty. On the interior of any selected cell $R$,
\[
 E_\varepsilon\le L_R-\varepsilon=F_P-\varepsilon<F_P.
\]
Every original essential plane was retained or lowered, so $E_\varepsilon\le F_P$ everywhere. This contradicts the least-roof property.

Every compact internal edge of a strongly generic curve is either a marked carrier or has its whole closure disjoint from $P$, because no mark is a vertex.
\end{proof}
By Pick\textquotesingle s formula, only terminal edges can contribute to the surplus, and each contribution lies at the interior endpoint of such an edge. We bound their total weight by comparing the successive gradients near a polygon corner.

Take henceforth a nonempty strongly generic $P\subset K\Subset\Omega$, with $N=|P|$. The corner calculation applies to any finite zero-boundary tropical polynomial. Its $O(\sqrt N)$ bound for $F_P$ uses the assumption $P\subset K\Subset\Omega$.

\needspace{5\baselineskip}\begin{proposition}[Terminal edges and the surplus measure]\label{prop:terminal}

Let $F$ be a finite nonnegative zero-boundary tropical polynomial on $\Delta$. Its terminal edges end at polygon vertices. The total terminal weight at a corner $S_r\cap S_{r+1}$ is at most
$\|n_r\|m_F(S_r)+\|n_{r+1}\|m_F(S_{r+1})$.

For $F_P$, a sufficiently small regular inward truncation contains all marks and interior vertices, crosses each terminal edge once, and crosses no compact internal edge. Write $B_{\rm term}$ for the terminal-edge count and $D_{\rm term}$ for their total weight. For strongly generic $P\subset K\Subset\Omega$,
\[
 \Xi_P(\Omega)=\frac12\sum_{e\ {\rm terminal}}(w(e)-1)
 =\frac{D_{\rm term}-B_{\rm term}}2
 \le C_\Omega D_\partial(F_P)\le C_{\Omega,K}\sqrt N.
\]
The surplus atoms lie at interior endpoints of terminal edges.
\end{proposition}
\begin{proof}
\emph{The fan at a polygon corner.} The side normal form of Lemma~\ref{lem:side} makes the roof affine near every relative-interior point of a side. Thus no edge closure meets a side interior. Every terminal edge reaches a polygon vertex, and its other endpoint is an interior tropical vertex by Lemma~\ref{lem:connectivity}.

For the zero roof there are no terminal edges and all side quasidegrees are zero, so the terminal-weight bound is immediate. For the corner calculation, assume henceforth that $F$ is nonzero. A nonnegative concave function that vanishes at an interior point vanishes identically: extend the segment from any other point slightly beyond the zero and apply concavity. If $m_F(S)=0$, the side normal form makes $F$ zero on an interior neighborhood. Thus every side quasidegree is positive.

Fix a corner $z=S_r\cap S_{r+1}$, and set
$k_r=m_F(S_r)$, $k_{r+1}=m_F(S_{r+1})$.
Use the affine coordinates $a=\lambda_r$, $b=\lambda_{r+1}$. Near $z$ the polygon is the quadrant $a,b\ge0$. Its normals $n_r,n_{r+1}$ are linearly independent. Every essential plane whose cell closure contains $z$ vanishes there and has a unique expression
\[
 L_j=\alpha_j a+\beta_j b,\qquad
 m_j=\alpha_jn_r+\beta_jn_{r+1}.
\]
Nonnegativity near the corner gives $\alpha_j,\beta_j\ge0$. The side normal forms supply the endpoint planes $k_{r+1}b$ and $k_ra$, with $k_r,k_{r+1}>0$. A plane with $\alpha_j>k_r$ would lie strictly above $k_ra$ in the interior and could not be essential. Likewise $\beta_j\le k_{r+1}$. Equality at either upper bound forces the other coefficient to be zero. Intermediate planes therefore lie strictly between the endpoint values in both coordinates.

All other essential planes are positive at $z$ and disappear from the minimum in a sufficiently small corner neighborhood. The remaining minimum is homogeneous in $a,b$. For $a>0$ write $t=b/a$:
\[
 F(a,b)=a\,h(t),\qquad h(t)=\min_j(\alpha_j+\beta_jt).
\]
As $t$ increases, the consecutive active slopes $\beta_j$ strictly decrease. At a switch $t_j>0$,
\begin{equation}\label{eq:corner-switch}
 \alpha_j-\alpha_{j-1}
 =t_j(\beta_{j-1}-\beta_j)>0.
\end{equation}
Ordering the cells from $S_{r+1}$ to $S_r$ gives
\[
 0=\alpha_0<\cdots<\alpha_q=k_r,
 \qquad k_{r+1}=\beta_0>\cdots>\beta_q=0.
\]
The switch rays $b=t_ja$ are exactly the terminal branches at this corner, in a corner neighborhood containing no interior tropical vertex. If
$m_j-m_{j-1}=w_j\nu_j$ with $\nu_j$ primitive, then $w_j$ is the branch weight and $w_j\le\|m_j-m_{j-1}\|$. Hence
\begin{align}
 D_{\rm term}^{(r)}(F)
 &\le\sum_{j=1}^q\|m_j-m_{j-1}\|\notag\\
 &\le\|n_r\|\sum_j(\alpha_j-\alpha_{j-1})
    +\|n_{r+1}\|\sum_j(\beta_{j-1}-\beta_j)\notag\\
 &=\|n_r\|k_r+\|n_{r+1}\|k_{r+1}.\label{eq:corner-weight}
\end{align}
Each side occurs at two corners. Summing gives the explicit bound
\begin{equation}\label{eq:terminal-degree-bound}
 D_{\rm term}(F)\le2\sum_S\|n_S\|m_F(S)
 \le2\max_S\|n_S\|\,D_\partial(F).
\end{equation}

\emph{Truncate only the terminal tails.} Choose a boundary collar avoiding marks, interior vertices, and compact internal edges, thin enough that only the corner rays remain in it. Generic small inward side offsets give a regular convex rational-slope truncation containing that compact set. Each terminal segment crosses once, by convexity; no internal edge crosses. The crossings retain their corner labels, so their count and total weight are independent of the sufficiently small truncation.

\emph{The total surplus.} For a strongly generic relaxation, Theorem~\ref{thm:mildness} removes the interior lattice-point contributions and makes every compact internal edge primitive. A terminal edge has exactly one interior endpoint. Summing Proposition~\ref{prop:pick} over all interior vertices therefore counts it once:
\begin{equation}\label{eq:terminal-surplus}
 \|\Xi_P\|_{\rm TV}=\Xi_P(\Omega)
 =\frac12\sum_{e\ {\rm terminal}}(w(e)-1)
 =\frac{D_{\rm term}-B_{\rm term}}2.
\end{equation}
Nonnegativity justifies the total-variation equality. Now \eqref{eq:terminal-degree-bound} and Theorem~\ref{thm:complexity} give
\[
 \Xi_P(\Omega)\le\max_S\|n_S\|\,D_\partial(F_P)
 \le C_{\Omega,K}\sqrt N.
\]
The contributing vertex is always the interior endpoint of a terminal edge. Its distance to the boundary need not be small uniformly over configurations.
\end{proof}

To extend curvature convergence to continuous tests on the closed polygon, we need the total mass. Choose a sufficiently small regular inward truncation containing all marks and all interior vertices. In the window identity, the cut graph has one component, and the terminal crossing and surplus terms combine to give $D_{\rm term}/2$.

\needspace{5\baselineskip}\begin{theorem}[Exact total mass on a polygon]\label{thm:total-mass}

For a nonempty strongly generic $N$-point cloud in $K\Subset\Omega=\Delta^\circ$,
\begin{equation}\label{eq:total-mass}
 \MA(F_{\Omega,P})(\Omega)=N-1+\tfrac12D_{\rm term},
 \qquad D_{\rm term}\le C_{\Omega,K}\sqrt N.
\end{equation}
\end{theorem}
\begin{proof}
Choose the inward truncation of Proposition~\ref{prop:terminal} to contain all marks and all interior tropical vertices. All curvature and surplus atoms are retained. The marked cut graph inside it is the global cut tree with only terminal tails shortened. It is connected, so its component count is one. The exact window identity gives
\[
 \MA(F_P)(\Omega)
 =N+\frac{B_{\rm term}}2-1+\Xi_P(\Omega).
\]
Substitute \eqref{eq:terminal-surplus}. The unweighted terminal count cancels, leaving
\[
 \MA(F_P)(\Omega)=N-1+\frac{D_{\rm term}}2.
\]
The bound on $D_{\rm term}$ follows from \eqref{eq:terminal-degree-bound} and Theorem~\ref{thm:complexity}. Quadratic homogeneity of the Monge--Amp\`ere measure then yields
\begin{equation}\label{eq:normalized-total-mass}
 \MA(N^{-1/2}F_P)(\Omega)
 =1-\frac1N+\frac{D_{\rm term}}{2N}
 =1+O_{\Omega,K}(N^{-1/2}).
\end{equation}
\end{proof}

\needspace{12\baselineskip}\section{\texorpdfstring{\headingcolor{Analysis}{The limiting roof}}{The limiting roof}}\label{sec:5}

The finite curvature count now leads to convergence of the normalized roofs. We first compare curvature measures with source measures, then use compactness to obtain limits of the roofs themselves. Aleksandrov stability determines the curvature of each limit, and comparison identifies it with the prescribed Dirichlet solution.

On a fixed polygon, integration of the window identity over the superlevel sets of a test function compares its source integral with its curvature integral. Coarea turns the crossing term into tangential variation along the tropical edges. When the sources stay in a fixed interior compact set, the length and surplus bounds are $O(\sqrt N)$, giving discrepancy of order $N^{-1/2}$ after normalization.

For a general convex domain, the estimates must be uniform along a fixed polygonal exhaustion. On interior rectangles, the height bound controls supporting slopes, and concavity along horizontal and vertical lines controls weighted crossings. These give the required local length and surplus bounds. The exhausting polygons are chosen in advance and used for every cloud.

The height and boundary barriers give compactness on $\overline\Omega$. On a fixed rational-slope polygon, the total-mass identity also yields weak curvature convergence on the closure. Appendix~\ref{sec:B} recalls the convex-analytic results used here.

\needspace{8\baselineskip}\subsection*{Curvature discrepancy on a polygon}

\needspace{5\baselineskip}\begin{lemma}[Regular levels of a test function]\label{lem:levels}

For $\varphi\in C_c^\infty(\Omega)$, there is a null set $E_\varphi\subset\R$ containing zero such that, for $t\notin E_\varphi$, the level $\{\varphi=t\}$ is regular in the plane, avoids every tropical vertex and mark, and meets each edge transversely at finitely many edge-interior points. In particular $\{\varphi>t\}$ is a regular window for almost every $t>0$. Apply the same assertion to $-\varphi$ for the negative levels.
\end{lemma}
\begin{proof}
Extend $\varphi$ by zero to the plane. There are two sorts of tangency to exclude: a critical point of $\varphi$ in the plane, and a critical point of its restriction to an edge. Sard's theorem makes the first set of values null. Parametrize each of the finitely many closed edges by arclength. The one-dimensional version makes the union of their critical values null as well; a constant restriction contributes just one value. Add zero and the finitely many values at vertices and marks. Call the resulting null set $E_\varphi$.

For $t\notin E_\varphi$, the implicit function theorem makes $\{\varphi=t\}$ a smooth curve. On an edge, a solution of $\varphi=t$ has nonzero tangential derivative and hence is isolated. All solutions lie in the compact closed edge, away from its endpoints, so there are finitely many. These are precisely the transverse crossings required for a regular window.

For $t>0$, the level is a compact smooth one-dimensional manifold without boundary, hence a finite union of embedded circles: a finite cover by connected coordinate neighborhoods bounds the number of components. Every superlevel component is bounded and has a boundary circle; each circle has just one positive side. Thus there are finitely many superlevel components, bounded by such circles, and distinct closures cannot meet. They are regular windows, including the empty case. Apply the argument to $-\varphi$ for negative levels.
\end{proof}

The exceptional levels have measure zero, so we may integrate the window inequality over the level parameter. Coarea expresses the resulting crossing term as tangential variation along the edges.

\needspace{5\baselineskip}\begin{theorem}[Tangential variation controls discrepancy]\label{thm:coarea}

For a strongly generic polygonal roof, with $\nu_P=\sum_{p\in P}\delta_p$, and $\varphi\in C_c^1(\Omega)$,
\[
 \left|\int\varphi\,d(\MA(F_P)-\nu_P-\Xi_P)\right|
 \le\frac12\int_{\Gamma(F_P)}|\nabla_\tau\varphi|\,d\mathcal H^1.
\]
\end{theorem}
\begin{proof}
Write $\Theta=\MA(F_P)-\nu_P-\Xi_P$. It is a finite signed measure on the polygon interior. The window identity in Theorem~\ref{thm:window-balance} gives
\begin{equation}\label{eq:window-discrepancy-tr7}
 \Theta(U)=\frac{B_P(U)}2-c_{\rm cut}(U),\qquad
 |\Theta(U)|\le\frac{B_P(U)}2,
\end{equation}
because $0\le c_{\rm cut}(U)\le B_P(U)$.

First take $\varphi\in C_c^\infty(\Omega)$. Layer cake, applied to the positive and negative parts and then to the Jordan decomposition of $\Theta$, gives
\begin{equation}\label{eq:signed-layer-cake}
 \int\varphi\,d\Theta
 =\int_0^\infty\Theta(\{\varphi>t\})\,dt
  -\int_0^\infty\Theta(\{\varphi<-t\})\,dt.
\end{equation}
The preceding lemma permits \eqref{eq:window-discrepancy-tr7} at almost every level. Each crossing lies on a unique edge and bounds a unique superlevel component, so its window multiplicity is one. Therefore
\[
 \left|\int\varphi\,d\Theta\right|
 \le\frac12\int_{\R}\#\bigl(\Gamma(F_P)\cap\{\varphi=t\}\bigr)\,dt
 =\frac12\sum_e\int_e|\nabla\varphi\cdot\tau_e|\,d\mathcal H^1.
\]
The equality is the one-dimensional area formula on each edge. Vertex values and the zero level do not affect the integral. Crossings are counted once, without edge weights.

For $\varphi\in C_c^1(\Omega)$, approximate in $C^1$ by smooth functions supported in one compact set. Local total variation controls the left side, and finite graph length controls the right. The approximation proves the inequality for $C_c^1(\Omega)$ using regular levels only for smooth test functions.
\end{proof}

On a fixed polygon, Section~\ref{sec:2} bounds the length of the entire curve, and Section~\ref{sec:4} bounds the total surplus. Substituting these estimates gives a constant independent of the support of the test function.

\needspace{5\baselineskip}\begin{theorem}[Uniform polygonal discrepancy]\label{thm:polygon-discrepancy}

For strongly generic $P\subset K\Subset\Omega=\Delta^\circ$, $|P|=N$, and every $\varphi\in C_c^1(\Omega)$,
\[
 \left|\int\varphi\,d(\MA(F_P)-\nu_P)\right|
 \le C_{\Omega,K}\sqrt N\bigl(\|\varphi\|_\infty+\|\nabla\varphi\|_\infty\bigr).
\]
Here the constant is independent of the support of the test function. Division by $N$ gives the normalized $N^{-1/2}$ estimate.
\end{theorem}
\begin{proof}
By Theorem~\ref{thm:coarea} and the nonnegativity of $\Xi_P$,
\[
 \left|\int\varphi\,d(\MA(F_P)-\nu_P)\right|
 \le \frac12\|\nabla\varphi\|_\infty\,
       \mathcal H^1(\Gamma(F_P))
       +\|\varphi\|_\infty\,\Xi_P(\Omega).
\]
Theorem~\ref{thm:complexity} bounds the length by $C_{\Omega,K}\sqrt N$, and Proposition~\ref{prop:terminal} gives the same bound for the surplus. Both bounds concern the whole fixed polygon, so the constant is independent of the support of $\varphi$. Finally,
$\MA(N^{-1/2}F_P)=N^{-1}\MA(F_P)$ by quadratic homogeneity, which gives the normalized rate.
\end{proof}

\needspace{8\baselineskip}\subsection*{Passage to a convex domain}
Positive level sets give a polygonal description of a roof on a general convex domain. Choose a level below the values of a least roof at its marks. The shifted roof on the resulting polygon is again least for the same marks. If the level is small enough, this gives an exact description on any prescribed compact interior set. These level polygons are related to the tropical wave fronts of Mikhalkin and Shkolnikov~\cite{MS23}. They depend on the individual roof. To compare different clouds, the convergence estimates below use one exhaustion chosen independently of the sources.

\needspace{5\baselineskip}\begin{theorem}[Polygonal level sets and replacement]\label{thm:cores}

Let $F$ be a nonnegative concave tropical series on $\Omega$, continuous and zero on the boundary. If $0<t<\max F$, then $\Delta_t=\{F\ge t\}$ is a compact convex polygon with rational side slopes and nonempty interior. A finite integral-slope presentation describes $F$ near $\Delta_t$; $H=F-t$ is a finite zero-boundary tropical polynomial on $\Delta_t$.

If $G\le H$ is a nonnegative zero-boundary tropical polynomial on $\Delta_t$, the function equal to $t+G$ on $\Delta_t$ and to $F$ outside is again a nonnegative concave integral-slope tropical series, continuous and zero on $\partial\Omega$.

If $F=F_{\Omega,P}$ with nonempty finite $P$, then $F>0$ in $\Omega$. For any compact $C\subset\Omega$ and $0<t<\min_{C\cup P}F$, one has $C\cup P\Subset\Delta_t^\circ$ and
\[
 F=t+F_{\Delta_t^\circ,P}\quad\hbox{on }\Delta_t.
\]
These conclusions also hold when the level set $\{F=t\}$ passes through tropical vertices or contains tropical edges.
\end{theorem}

\begin{proof}
\emph{The finite core.} Concavity makes $\Delta_t$ convex. Continuity and zero boundary values put it compactly inside $\Omega$, and $t<\max F$ gives it nonempty interior. We claim that
\begin{equation}\label{eq:core-boundary}
 \partial\Delta_t=\{x\in\Omega:F(x)=t\}.
\end{equation}
Continuity gives $F=t$ on $\partial\Delta_t$. To prove that no interior point of $\Delta_t$ has value $t$, suppose $x\in\Delta_t^\circ$ and $F(x)=t$, and choose $y$ with $F(y)>t$. For sufficiently small $\varepsilon>0$, the point $z=x-\varepsilon(y-x)$ still lies in $\Delta_t$. Since $x=(z+\varepsilon y)/(1+\varepsilon)$, concavity gives
\[
 t=F(x)\ge\frac{F(z)+\varepsilon F(y)}{1+\varepsilon}>t,
\]
a contradiction.

Choose a compact convex set $K$ with $\Delta_t\Subset K^\circ\Subset K\Subset\Omega$. Local finiteness gives finitely many affine pieces near each point of $K$. Discard pieces that never attain the minimum in their neighborhood. Every retained piece $L$ is a global support of $F$. Indeed, if $L(x)=F(x)$ and $L\ge F$ near $x$, then for any $y\in\Omega$ and sufficiently small $s>0$,
\[
 (1-s)F(x)+sF(y)\le F((1-s)x+sy)
 \le L((1-s)x+sy).
\]
Cancelling the value at $x$ gives $F(y)\le L(y)$.

A finite subcover gives integral-gradient supports $S_1,\ldots,S_M$ with minimum $F$ near $K$. Put $Q=\bigcap_\ell\{S_\ell\ge t\}$; then $Q\cap K=\Delta_t$. A point of $Q$ outside $K$ would join to $\Delta_t^\circ$ by a segment meeting $\partial K$ in $\Delta_t$, contradicting $\Delta_t\Subset K^\circ$. Thus $Q=\Delta_t$, a bounded polygon with integral normals. The finite polynomial $H=F-t$ vanishes on its boundary by \eqref{eq:core-boundary}.

\emph{The replacement.} Choose the essential presentation
\[
 G=\min_{1\le j\le m}L_j\quad\text{on }\Delta_t
\]
from Proposition~\ref{prop:essential}. Each $L_j$ is active at some $x_j\in\Delta_t^\circ$. We claim that
\begin{equation}\label{eq:core-exterior-plane}
 L_j(y)\ge F(y)-t\qquad(y\in\Omega\setminus\Delta_t).
\end{equation}
Fix $j$ and $y$ outside the core. The segment from $x_j$ to $y$ exits $\Delta_t$ at $z=(1-s)x_j+sy$, with $0<s<1$. Write $H=F-t$ also outside the core. At $z$ we have $H(z)=0$ and $L_j(z)\ge G(z)=0$, while $H(x_j)\ge G(x_j)=L_j(x_j)$. Concavity of $H$ and affinity of $L_j$ give
\begin{equation}\label{eq:core-critical-calculation}
 H(y)\le-\frac{1-s}{s}H(x_j)
 \le-\frac{1-s}{s}L_j(x_j)
 \le\frac{L_j(z)-(1-s)L_j(x_j)}s=L_j(y).
\end{equation}
Thus \eqref{eq:core-exterior-plane} holds for every essential piece. The replacement has the global expression
\begin{equation}\label{eq:core-global-minimum}
 \widetilde F=\min\{F,t+L_1,\ldots,t+L_m\}.
\end{equation}
The minimum is $t+G$ inside, since $G\le F-t$, and $F$ outside by \eqref{eq:core-exterior-plane}. It is concave and locally a finite integral-slope minimum. The pieces are nonnegative and agree on $\partial\Delta_t$; near $\partial\Omega$ it equals $F$, preserving the boundary values and continuity.

\emph{The exact least roof.} Suppose $F=F_{\Omega,P}$ and $P\ne\varnothing$. The roof is nonzero, since the zero function has no corners. A nonnegative concave function on an open convex domain that vanishes at an interior point is identically zero: for any $y$ in the domain, extend the segment from $y$ through that point slightly farther within the domain and apply concavity. Thus $F>0$ in $\Omega$.

For $0<t<\min_{C\cup P}F$, continuity gives $C\cup P\Subset\Delta_t^\circ$. The function $H=F-t$ is admissible for the polygonal relaxation, so $G=F_{\Delta_t^\circ,P}\le H$. This $G$ has a finite presentation by Proposition~\ref{prop:relaxation}. Glue it into the core using \eqref{eq:core-global-minimum}. Each mark lies in the core interior, where adding $t$ preserves its corner, so $\widetilde F$ is admissible for $F_{\Omega,P}$. Minimality gives $F\le\widetilde F$, while the construction gives $\widetilde F\le F$. Equality proves the shifted-relaxation formula.

Neither the polygon construction nor the replacement requires the level set $\{F=t\}$ to avoid tropical vertices or edges.
\end{proof}

For the convergence argument, choose the exhausting polygons in advance and keep the cloud fixed. Their least roofs increase to the least roof on $\Omega$. A common finite set of slopes on each interior compact set will ensure that the prescribed corners survive the limit.

\needspace{5\baselineskip}\begin{proposition}[Monotonicity and polygonal exhaustion]\label{prop:exhaustion}

If $\Omega_1\subset\Omega_2$ and $P\Subset\Omega_1$, then $F_{\Omega_1,P}\le F_{\Omega_2,P}$ on $\Omega_1$. For an increasing polygonal exhaustion $\Delta_j^\circ\uparrow\Omega$ containing a fixed finite $P$, the least roofs increase locally uniformly to $F_{\Omega,P}$, and their Monge--Amp\`ere measures converge vaguely in $\Omega$.

\end{proposition}
\begin{proof}
\emph{Order the domains.} A restriction to a smaller domain generally has positive boundary values. The lattice support distance supplies the needed cap. Put $F_2=F_{\Omega_2,P}$ and choose an integer $\ell_{{\rm cap},\Omega_1}$ so large that
$\ell_{{\rm cap},\Omega_1}D_{\Omega_1}>F_2$ near every mark. This is possible because $D_{\Omega_1}(p)>0$ at each of the finitely many marks. The roof
\begin{equation}\label{eq:domain-cap}
 H=\min\{F_2|_{\Omega_1},\ell_{{\rm cap},\Omega_1}D_{\Omega_1}\}
\end{equation}
is nonnegative, locally tropical, and zero on $\partial\Omega_1$. It agrees with $F_2$ near each mark. Minimality gives
$F_{\Omega_1,P}\le H\le F_2$. The empty cloud has zero roof and needs no cap.

\emph{Passage to the limit.} We first prove a local claim. Let $f_j$ be a locally bounded increasing sequence of tropical roofs on increasing domains, with a common finite set of active integral gradients on every interior compact set. Its limit is locally tropical, and every point that is a corner of all sufficiently large $f_j$ remains a corner of the limit.

To prove the claim, take nested compact convex neighborhoods $W\Subset (W')^\circ\Subset U$ in the union $U$ of the increasing domains. A common finite gradient set gives a common Lipschitz bound for all sufficiently large $j$. The bounded increasing limit $f$ is therefore continuous, and Dini's theorem gives uniform convergence on $W'$. If $A\subset\Z^2$ contains all gradients active near $W'$, complete each local presentation by setting
\begin{equation}\label{eq:finite-slope-limit}
 c_{j,m}=\max_{x\in W'}(f_j(x)-m\cdot x),\qquad
 c_m=\max_{x\in W'}(f(x)-m\cdot x),\qquad m\in A.
\end{equation}
Every such affine function majorizes $f_j$ on $W'$. A plane active at a point of $W$ is a support on $W'$ by concavity, so its coefficient equals the corresponding maximum. Hence
\[
 f_j(x)=\min_{m\in A}(c_{j,m}+m\cdot x),\qquad
 f(x)=\min_{m\in A}(c_m+m\cdot x)\quad(x\in W),
\]
where the second equality follows from uniform convergence of the finitely many coefficients. The limit remains locally tropical. At a fixed marked point, some pair of distinct active gradients recurs along a subsequence. Both limiting planes still support and touch $f$ there, so its superdifferential contains two distinct points. The corner survives.

\emph{Identify the roof.} Keep $P$ fixed, and write $F_j=F_{\Delta_j^\circ,P}$ and $F=F_{\Omega,P}$. Domain monotonicity gives
$0\le F_j\le F_{j+1}\le F$. The roof $F$ is bounded. On a compact interior neighborhood, concavity bounds all supporting slopes of $F_j$ uniformly: testing a support at $x\pm re_i$ gives $|m_i|\le\|F\|_\infty/r$ whenever both points lie in that neighborhood. Integral gradients consequently belong to one finite set. The preceding argument gives a locally tropical limit $G$, local uniform convergence, and all the required corners.

Moreover $0\le G\le F$. Since $F$ tends to zero at the boundary, $G$ extends continuously by zero there. It is therefore admissible for $F_{\Omega,P}$. Minimality gives $F\le G$, proving equality. Applying Theorem~\ref{thm:stability} to $-F_j\to-F$ gives vague convergence of the curvature measures. Reflection preserves area, so the limiting measure is $\MA(F)$.
\end{proof}

Proposition~\ref{prop:exhaustion} applies to one fixed cloud. To let $N$ tend to infinity afterward, we need local estimates whose constants are uniform for $P\subset K$.

A rectangle $R=(a,b)\times(c,d)$ is called \emph{regular} when its boundary avoids vertices and marks, contains no edge segment, and crosses edges transversely in the relative interiors of its sides, with no crossing at a corner. Write
$B_F^w(R)=\sum_{x\in\Gamma(F)\cap\partial R}w(e_x)$, where $e_x$ is the unique edge through $x$.

\needspace{5\baselineskip}\begin{lemma}[Local length and weighted crossings]\label{lem:crofton}

\begin{enumerate}[label=\textup{(\roman*)},leftmargin=2em,itemsep=3pt,topsep=3pt]
\item For $P\subset K\Subset\Omega$ with $N=|P|$, the active gradients of $F_{\Delta_j^\circ,P}$ near a fixed $W\Subset\Omega$ are bounded by $C_{\Omega,K,W}\sqrt N$, uniformly for all sufficiently large $j$.
\item Let $R$ be a regular axis-parallel rectangle. If the gradients of a tropical roof $F$ near $\overline R$ have infinity norm at most $\Lambda$, then
\[
 \sum_e w(e)\,\mathcal H^1(e\cap R)
 \le 2\Lambda\bigl(\operatorname{width}R+\operatorname{height}R\bigr),
 \qquad B_F^w(R)\le8\Lambda.
\]
\end{enumerate}
Together these estimates give weighted curve length $O_{\Omega,K,W}(\sqrt N)$ on any fixed interior compact set, uniformly in the exhausting polygon.
\end{lemma}

\begin{proof}
\emph{Height controls the available slopes.} Fix a compact neighborhood $W'\Subset\Omega$ of $W$, and choose $r>0$ so that $x\pm re_i$ lie in a further fixed interior neighborhood for every $x\in W'$. All these sets lie inside $\Delta_j^\circ$ for sufficiently large $j$. Domain monotonicity and Proposition~\ref{prop:barriers} give
\[
 0\le F_{\Delta_j^\circ,P}\le F_{\Omega,P}\le H,
 \qquad H=C_\Omega\sqrt N.
\]
Write $F=F_{\Delta_j^\circ,P}$. If $m$ is a supporting slope of $F$ at $x$, its support inequality at $x\pm re_i$ gives
\[
 0\le F(x+re_i)\le F(x)+rm_i,\qquad
 0\le F(x-re_i)\le F(x)-rm_i.
\]
Thus $|m_i|\le H/r$, uniformly in both polygon and cloud. This gives the asserted $C_{\Omega,K,W}\sqrt N$ bound.

\emph{Slope variation along a line.} Along almost every vertical line $x=t$, the derivative of the concave restriction $y\mapsto F(t,y)$ decreases through values in $[-\Lambda,\Lambda]$. If the line crosses an edge separating gradients $m^+,m^-$, write
$m^+-m^-=w(e)n_e$ with $n_e$ primitive. The derivative drop has size
\begin{equation}\label{eq:crofton-jump}
 |m^+_2-m^-_2|=w(e)|(n_e)_2|\ge w(e).
\end{equation}
The last inequality holds at a transverse vertical crossing because $(n_e)_2$ is a nonzero integer. The total drop is at most $2\Lambda$. Integrating the weighted crossing count in $t$, by the one-dimensional area formula on each straight edge, gives
\begin{align*}
 \sum_e w(e)\int_{e\cap R}|\tau_x|\,d\mathcal H^1
 &\le 2\Lambda(b-a),\\
 \sum_e w(e)\int_{e\cap R}|\tau_y|\,d\mathcal H^1
 &\le 2\Lambda(d-c).
\end{align*}
The second line uses horizontal slices. Add them and use $|\tau_x|+|\tau_y|\ge1$. Horizontal and vertical edges are each counted by the transverse slices.

On each side of a regular rectangle the same restriction has total derivative variation at most $2\Lambda$. The derivative drops by at least the edge weight at every crossing by \eqref{eq:crofton-jump} or its horizontal counterpart. Four sides therefore give $B_F^w(R)\le8\Lambda$.

Cover the compact set by finitely many rectangles in a fixed larger interior compact set. The slope bound and the rectangle estimates give the claimed length bound; $w\ge1$ also bounds ordinary length.
\end{proof}

\begin{figure}[htbp]
\centering
\begin{tikzpicture}[x=1cm,y=1cm,font=\footnotesize,>=stealth]
\node[font=\footnotesize\bfseries] at (0,2.2) {(a) Weighted crossings};
\node[font=\footnotesize\bfseries] at (7.5,2.2) {(b) One vertical slice};
\begin{scope}[x=.8cm,y=.8cm]
\draw[Ink!60,dashed] (-2,-2) rectangle (2,2);
\draw[Analysis!65,densely dashed] (-1,-2)--(-1,2);
\draw[Geometry,line width=1.5pt] (0,0)--(0,2);
\draw[Geometry,line width=2pt] (0,0)--(2,0);
\draw[Geometry,line width=.8pt] (0,0)--(-2,-1.3333);
\fill[Geometry] (0,0) circle (2pt);
\fill[Analysis] (-1,-.6667) circle (2pt);
\node[right] at (.04,1.3) {$w=2$};
\node[above] at (1.2,.08) {$w=3$};
\node[below right] at (-1.75,-1.2) {$w=1$};
\node[above left,Analysis] at (-1,-.6667) {$z$};
\node[below] at (1,-1.0) {$n_e=(2,-3)$};
\node[above] at (-1,2.03) {$x=-1$};
\end{scope}
\begin{scope}[shift={(7.5,1.6)},x=.85cm,y=.58cm]
\draw[->,Ink!60] (-2.2,0)--(2.25,0) node[right] {$y$};
\draw[->,Ink!60] (-2.2,-6.2)--(-2.2,.3);
\node[rotate=90] at (-2.65,-3) {$F(-1,y)$};
\draw[Analysis,line width=1.2pt] (-2,-6)--(-.6667,-2)--(2,-2);
\draw[Ink!60,dotted] (-.6667,0)--(-.6667,-2);
\node[above] at (-.6667,0) {$-2/3$};
\node[above] at (1,-2) {slope $0$};
\fill[Analysis] (-.6667,-2) circle (2pt);
\node[below right,Analysis,align=left] at (-.6667,-2) {$z$\\$3\to0$};
\node[rotate=54,above] at (-1.45,-4.35) {slope $3$};
\end{scope}
\node at (0,-2.3) {$R=(-2,2)^2,\qquad B_F^w(R)=6$};
\node at (7.5,-2.3) {$F(-1,y)=\min\{-2,3y\}$};
\end{tikzpicture}
\caption{For the local model $F(x,y)=\min\{0,2x,3y\}$, the three edge weights are $2,3,1$. The vertical slice through $x=-1$ crosses the diagonal edge of weight one, while its derivative drops by three: the jump is $w|(n_e)_2|$. At a transverse crossing, $|(n_e)_2|\ge1$ because $(n_e)_2$ is a nonzero integer. This local calculation uses concavity and integral slopes.}
\label{fig:crofton-slice}
\end{figure}

The weighted crossing bound also controls the local surplus. For a strongly generic roof, surplus at an interior vertex comes only from terminal edges. Any such edge issuing inside an interior rectangle must cross its boundary, where its weight is counted.

\needspace{5\baselineskip}\begin{proposition}[Local surplus from weighted crossings]\label{prop:local-excess}

Let $F=F_{\Delta^\circ,P}$ be strongly generic, and let $L\Subset R\Subset\Delta^\circ$, with $R$ a regular axis-parallel rectangle. Then
\[
 \Xi_F(L)\le\tfrac12 B_F^w(R).
\]
For clouds in fixed $K$, this gives $\Xi_F(L)\le C_{\Omega,K,L}\sqrt N$, uniformly along a sufficiently large exhaustion of $\Omega$.
\end{proposition}
\begin{proof}
By Theorem~\ref{thm:mildness}, interior Newton polygons have no interior lattice points, and compact internal edges have weight one. Pick's formula therefore reduces the surplus at a vertex to
\begin{equation}\label{eq:local-terminal-surplus}
 2\sigma_F(v)=\sum_{\substack{\eta\text{ incident at }v\\
                         e_\eta\text{ terminal relative to }\Delta}}
                         (w(e_\eta)-1).
\end{equation}
A term at $v\in L$ comes from a straight terminal segment leading from $v$ to $\partial\Delta$. It leaves the convex rectangle $R$ at exactly one point $x_\eta$. Different contributing germs give different crossings: each terminal edge has one interior endpoint, and regularity places each crossing in the interior of one unique edge. Consequently
\begin{equation}\label{eq:surplus-crossing-injection}
 2\Xi_F(L)
 \le\sum_{\eta\text{ contributing}}w(e_\eta)
 \le\sum_{x\in\Gamma(F)\cap\partial R}w(e_x)
 =B_F^w(R).
\end{equation}

To obtain the uniform bound on an arbitrary compact $L\Subset\Omega$, we apply the rectangle estimate on a finite cover. Choose fixed closed rectangles $Q_\alpha$ covering $L$ and open rectangles $S_\alpha$ with $Q_\alpha\Subset S_\alpha\Subset\Omega$. For each finite graph choose regular intermediate rectangles $Q_\alpha\Subset R_{\alpha,F}\Subset S_\alpha$. Side positions avoid finitely many vertex and mark coordinates and edge-containing lines; corner incidences exclude finitely many further relations. The sides of $R_{\alpha,F}$ can be chosen between the fixed inner and outer rectangles while avoiding these finitely many conditions.

The gradient bound on a fixed neighborhood of the $S_\alpha$ is independent of these choices. By equation~\eqref{eq:surplus-crossing-injection} and the bound $B_F^w\le8\Lambda$ in Lemma~\ref{lem:crofton}, we obtain $\Xi_F(Q_\alpha)\le C_\alpha\sqrt N$. Sum over the cover. The constant depends only on the fixed interior cover and the slope bound.
\end{proof}

Fix the support of the test function. The curve length and surplus there are now bounded by $C_{\Omega,K,L}\sqrt N$, independently of the sufficiently large exhausting polygon. For a fixed cloud we may therefore pass the polygonal discrepancy estimate to the limit; the same bound will then apply as the cloud varies.

\needspace{5\baselineskip}\begin{theorem}[Discrepancy on a convex domain]\label{thm:general-discrepancy}

Fix $K\Subset O\Subset\Delta_1^\circ$ and a nested exhaustion $\mathcal E=(\Delta_j)$ of $\Omega$ by compact convex polygons with rational side slopes. For a cloud $P\subset K$ strongly generic for every $\Delta_j$, and any $\varphi\in C_c^1(\Omega)$ supported in $L\Subset\Omega$,
\[
 \left|\int\varphi\,d(\MA(F_{\Omega,P})-\nu_P)\right|
 \le C_{\Omega,K,L}\sqrt{|P|}
        (\|\varphi\|_\infty+\|\nabla\varphi\|_\infty).
\]
The corresponding estimate on $\Delta_j$ has the same constant for all sufficiently large $j$.
\end{theorem}
\begin{proof}
For sufficiently large $j$, the compact support $L$ and a fixed neighborhood of it lie inside $\Delta_j^\circ$. Set $F_j=F_{\Delta_j^\circ,P}$. Theorem~\ref{thm:coarea} gives
\begin{equation}\label{eq:uniform-local-error}
 \left|\int\varphi\,d(\MA(F_j)-\nu_P)\right|
 \le\frac12\|\nabla\varphi\|_\infty
          \mathcal H^1(\Gamma(F_j)\cap L)
       +\|\varphi\|_\infty\Xi_{F_j}(L).
\end{equation}
Here $\nabla\varphi$ vanishes off its support. Lemma~\ref{lem:crofton} and Proposition~\ref{prop:local-excess} bound the two terms by
$C_{\Omega,K,L}\sqrt{|P|}(\|\nabla\varphi\|_\infty+\|\varphi\|_\infty)$, with the same constant for every sufficiently large $j$. Strong genericity gives the coarea discrepancy and the surplus formula.

Now fix the cloud. Proposition~\ref{prop:exhaustion} gives local uniform convergence of $F_j$ to $F_{\Omega,P}$ and vague convergence of their curvature measures. Pass to the limit in the left side of \eqref{eq:uniform-local-error}. The right side stays fixed, giving the asserted convex-domain estimate. The threshold in $j$ depends on the interior neighborhoods, and the constant is uniform over the admissible clouds.
\end{proof}

\needspace{8\baselineskip}\subsection*{Compactness and identification of the limit}
The curvature discrepancy tends to zero after normalization. To obtain convergence of the roofs, we need compactness on $\overline\Omega$. Concavity bounds the slopes on interior sets, and the common barrier makes the roofs uniformly small near the boundary.

\needspace{5\baselineskip}\begin{proposition}[Uniform compactness up to the boundary]\label{prop:compactness}

For every sequence of finite clouds $P_N\subset K\Subset\Omega$, the normalized roofs are uniformly bounded and equicontinuous on $\overline\Omega$, and hence precompact in $C(\overline\Omega)$.
\end{proposition}
\begin{proof}
The height and fixed-depth bounds in Proposition~\ref{prop:barriers} give $0\le u_N\le C_{\Omega,K}D_\Omega$ and $\|u_N\|_\infty\le H$. Given $\varepsilon>0$, choose $\delta>0$ so that $C_{\Omega,K}D_\Omega<\varepsilon/2$ whenever the distance to the boundary is at most $3\delta$. If $|x-y|<\delta$ and one point is within $2\delta$ of the boundary, both values lie in $[0,\varepsilon/2)$. Otherwise both points lie in the convex inner body $\Omega_{2\delta}$, where the slope bound \eqref{eq:interior-slope} gives $|u_N(x)-u_N(y)|\le H|x-y|/\delta$. Taking $|x-y|<\varepsilon\delta/(H+1)$ proves equicontinuity, uniformly in $N$. Arzel\`a--Ascoli gives uniform precompactness. 
\end{proof}

If $\mu_N\rightharpoonup\mu$, the discrepancy estimate identifies the curvature of any uniform subsequential limit. The barrier gives zero boundary values, and comparison leaves only one possible limit.

\needspace{5\baselineskip}\begin{theorem}[The many-point limit along a fixed exhaustion]\label{thm:fixed-exhaustion}

Let $\Omega\subset\R^2$ be a bounded open convex domain with nonempty interior, and fix $K\Subset\Omega$. Choose a nested exhaustion $(\Delta_j)$ of $\Omega$ by compact convex polygons with rational side slopes, with $K\Subset\Delta_1^\circ$. Suppose that each $P_N\subset K$ consists of $N$ distinct points and is strongly generic on every polygon $\Delta_j$. The same exhaustion is used for all $N$. Then all conclusions of Theorem~\ref{thm:many-point} hold.
\end{theorem}
\begin{proof}
Choose an open polygon $O$ with $K\Subset O\Subset\Delta_1^\circ$. Theorem~\ref{thm:general-discrepancy} applies to each cloud on the fixed exhaustion. Together with quadratic homogeneity, it gives \eqref{eq:main-discrepancy}, before any assumption on empirical convergence.

Suppose now that $\mu_N\rightharpoonup\mu$. Since each $\mu_N$ is a probability supported in the same compact $K$, so is $\mu$. Put $\lambda_N=\MA(u_N)$. To pass from differentiable test functions to continuous ones, fix $L\Subset\Omega$ and choose $\chi\in C_c^1(\Omega)$, nonnegative and equal to one near $L$. The normalized discrepancy estimate gives
\begin{equation}\label{eq:uniform-local-mass}
 \lambda_N(L)\le\int\chi\,d\lambda_N
 \le\|\chi\|_\infty+
 \frac{C_\chi}{\sqrt N}
       (\|\chi\|_\infty+\|\nabla\chi\|_\infty).
\end{equation}
Thus the masses are uniformly bounded on every compact interior set. For $\psi\in C_c(\Omega)$, approximate uniformly by $\psi_k\in C_c^1(\Omega)$ with supports in one larger compact $L$. Then
\begin{align*}
 \left|\int\psi\,d(\lambda_N-\mu_N)\right|
 &\le \left|\int\psi_k\,d(\lambda_N-\mu_N)\right|\\
 &\quad+\|\psi-\psi_k\|_\infty\bigl(\lambda_N(L)+1\bigr).
\end{align*}
First let $N\to\infty$ for fixed $k$, and then $k\to\infty$. It follows that $\lambda_N\rightharpoonup\mu$ vaguely in $\Omega$.

Proposition~\ref{prop:compactness} gives a uniformly convergent subsequence on $\overline\Omega$. Its limit $u$ is continuous, concave, nonnegative, and zero on the boundary. Local Aleksandrov stability (Theorem~\ref{thm:stability}) gives
\[
 \MA(u_{N_k})\rightharpoonup\MA(u)\quad\text{vaguely in }\Omega.
\]
Hence $\MA(u)=\mu$. Theorem~\ref{thm:comparison} makes this zero-boundary solution unique. Thus $u=F_{\mu,\Omega}$; this argument also constructs the solution for the measure under consideration. It applies without a density assumption on $\mu$.

Every subsequence has a further uniformly convergent subsequence, always with this same limit. Hence the full sequence converges uniformly.
\end{proof}

The barrier controls roof values at the boundary. To test curvature there, we also need the polygonal total-mass identity.

\needspace{5\baselineskip}\begin{theorem}[Weak curvature convergence on the closed polygon]\label{thm:polygon-limit}

For strongly generic $P_N\subset K\Subset\Omega=\Delta^\circ$ with $|P_N|=N$ and $\mu_N\rightharpoonup\mu$, the normalized roofs converge uniformly on $\overline\Omega$ to $F_{\mu,\Omega}$, hence in $L^p(\Omega)$ for every $1\le p\le\infty$. Extending the interior Monge--Amp\`ere measures by zero to $\partial\Omega$ gives weak convergence on $\overline\Omega$:
\[
 \MA(N^{-1/2}F_{\Omega,P_N})\rightharpoonup\mu.
\]
Their total masses tend to one.
\end{theorem}
\begin{proof}
Strong genericity on the fixed polygon permits Theorem~\ref{thm:polygon-discrepancy}. After normalization it gives discrepancy of order $N^{-1/2}$. The cutoff and approximation argument in the preceding proof then gives vague convergence to $\mu$ in $\Omega$. Proposition~\ref{prop:compactness}, local stability, and uniqueness give uniform convergence of the roofs on $\overline\Omega$ by the same subsequence argument. Since $\Omega$ has finite area, uniform convergence also implies all the stated $L^p$ conclusions.

It remains to test curvature against functions which reach the boundary. Write again $\lambda_N=\MA(u_N)$, extended by zero to $\partial\Omega$. Theorem~\ref{thm:total-mass} gives
\begin{equation}\label{eq:polygon-mass-return}
 \lambda_N(\overline\Omega)
 =1-\frac1N+\frac{D_{\rm term}(F_{\Omega,P_N})}{2N}
 \longrightarrow1.
\end{equation}
Choose $\chi\in C_c^\infty(\Omega)$ with $0\le\chi\le1$ and $\chi=1$ near $K$. Since $\mu$ is supported in $K$, vague convergence yields $\int\chi\,d\lambda_N\to1$. Positivity and \eqref{eq:polygon-mass-return} now show
\begin{equation}\label{eq:no-polygon-boundary-loss}
 \int(1-\chi)\,d\lambda_N
 =\lambda_N(\overline\Omega)-\int\chi\,d\lambda_N
 \longrightarrow0.
\end{equation}
For any $h\in C(\overline\Omega)$, the function $\chi h$ has compact interior support, and
\[
 \left|\int(1-\chi)h\,d\lambda_N\right|
 \le\|h\|_\infty\int(1-\chi)\,d\lambda_N\longrightarrow0.
\]
Combining this with vague convergence for $\chi h$ proves weak convergence on $\overline\Omega$.
\end{proof}

\subsection*{Finite measures}
If $u_N\to F_{\mu,\Omega}$ for a probability measure $\mu$ and we prescribe a finite measure $\nu=m\mu$ of mass $m>0$, homogeneity and uniqueness give
\[
 F_{\nu,\Omega}=\sqrt m\,F_{\mu,\Omega},
 \qquad
 \sqrt{\frac mN}\,F_{\Omega,P_N}\longrightarrow F_{\nu,\Omega}.
\]
The normalization of the data therefore amounts to a change of vertical scale; for zero mass the solution is the zero roof.

\needspace{12\baselineskip}\section{\texorpdfstring{\headingcolor{Probability}{Random clouds and affine geometry}}{Random clouds and affine geometry}}\label{sec:6}

Independent sampling gives roof approximations for probability densities throughout the domain, including near the boundary. Let $\mu$ be an absolutely continuous probability law on $\Omega$. Its independent samples are almost surely generic, and their empirical measures converge to $\mu$. When the law is supported in a fixed interior compact set, the deterministic theorem applies directly.

To allow support reaching the boundary, retain the samples in a compact set $K\Subset\Omega$ of mass $a>0$. In their original order they have conditional law $\mu|_K/a$, and their proportion tends to $a$. Quadratic homogeneity of the Monge--Amp\`ere measure then gives the limiting roof with measure $\mu|_K$. The removal estimate controls the discarded fraction as $K$ exhausts $\Omega$.

Appendix~\ref{sec:D}, D1, proves the conditional-law assertion used in this passage to the boundary.

\needspace{5\baselineskip}\begin{theorem}[Sampling in a fixed compact set]\label{thm:iid-compact}

Let $X_i$ be i.i.d. with an absolutely continuous probability law $\mu$ supported in $K\Subset\Omega$. Almost surely, every initial cloud $P_N=\{X_1,\ldots,X_N\}$ has $N$ distinct points and the required genericity, and
\[
 N^{-1/2}F_{\Omega,P_N}\longrightarrow F_{\mu,\Omega}
 \quad\hbox{uniformly on }\overline\Omega.
\]
The curvature measures converge vaguely in $\Omega$, and weakly on $\overline\Omega$ when $\Omega$ is a rational-slope polygon.
\end{theorem}
\begin{proof}
Absolute continuity keeps each ordered $N$-point cloud out of the null exceptional set in Definition~\ref{def:universal}. On a fixed rational-slope polygon it also gives strong genericity by Theorem~\ref{thm:generic-locus}. The equalities $X_i=X_j$ have probability zero. Intersect over all $N$ and pairs to obtain one probability-one event.

We also need the empirical measures to converge on one event for all continuous tests. Choose a countable uniformly dense family $(g_r)$ in $C(K)$. For each $r$, the strong law gives
\[
 \frac1N\sum_{i=1}^N g_r(X_i)\longrightarrow\int_Kg_r\,d\mu.
\]
Intersect these events over $r$. If $g\in C(K)$ and $\|g-g_r\|_\infty<\varepsilon$, the two approximation errors in its empirical integral and its $\mu$-integral have sum at most $2\varepsilon$. The convergence for $g_r$ therefore proves convergence for $g$. This proves the compact-space form of Varadarajan\textquotesingle s theorem~\cite{Var58}.

On the probability-one event where both genericity and empirical convergence hold, Theorem~\ref{thm:many-point} gives uniform roof convergence and vague interior curvature convergence. Theorem~\ref{thm:polygon-limit} gives weak convergence on the closure for a rational-slope polygon.
\end{proof}

To remove the compact-support hypothesis, we combine this theorem with Corollary~\ref{cor:removal}.

\needspace{5\baselineskip}\begin{theorem}[Sampling laws that reach the boundary]\label{thm:iid-boundary}

Let $X_1,X_2,\ldots$ be i.i.d. with an absolutely continuous probability law $\mu$ on the bounded open convex domain $\Omega$, and put $P_N=\{X_1,\ldots,X_N\}$. There is a unique continuous concave function $F_{\mu,\Omega}$ on $\overline\Omega$ satisfying
\[
 \MA(F_{\mu,\Omega})=\mu\quad\text{in }\Omega,
 \qquad F_{\mu,\Omega}=0\quad\text{on }\partial\Omega.
\]
Almost surely, all samples are distinct and
\begin{align*}
 N^{-1/2}F_{\Omega,P_N}&\longrightarrow F_{\mu,\Omega}
 &&\text{uniformly on }\overline\Omega,\\
 \MA(N^{-1/2}F_{\Omega,P_N})&\rightharpoonup\mu
 &&\text{vaguely in }\Omega.
\end{align*}
The support of $\mu$, regarded in $\overline\Omega$, may meet the boundary.
\end{theorem}

\begin{proof}
Choose nested compact convex polygons $K_j\Subset\Omega$ whose interiors exhaust $\Omega$. Their boundaries have $\mu$-measure zero. Write
\[
 a_j=\mu(K_j),\qquad \mu_j=\frac{\mu|_{K_j}}{a_j},\qquad
 P_N^{(j)}=P_N\cap K_j,\qquad M_{N,j}=|P_N^{(j)}|.
\]
We may discard the first few polygons so that $a_j>0$ for every $j$. Then $a_j\uparrow1$.

For each $j$, the successive samples accepted by $K_j$ form an i.i.d. sequence with law $\mu_j$, by Lemma~\ref{lem:accepted-stream}. Apply Theorem~\ref{thm:iid-compact} to this sequence. Intersect its probability-one events over $j$, together with the strong-law events and the event that all original samples are distinct. On this single event,
\[
 \frac{M_{N,j}}{N}\longrightarrow a_j,
 \qquad M_{N,j}\longrightarrow\infty,
\]
and the compact-support theorem holds for every accepted sequence, including genericity of all its finite initial configurations.

Keep $j$ fixed and let $N$ tend to infinity. The first $M_{N,j}$ accepted samples form $P_N^{(j)}$, so
\begin{equation}\label{eq:accepted-normalization}
 \frac{F_{\Omega,P_N^{(j)}}}{\sqrt N}
 =\sqrt{\frac{M_{N,j}}{N}}\,
   \frac{F_{\Omega,P_N^{(j)}}}{\sqrt{M_{N,j}}}
 \longrightarrow G_j:=\sqrt{a_j}\,F_{\mu_j,\Omega}
\end{equation}
uniformly on $\overline\Omega$. The identity is used for all sufficiently large $N$, when $M_{N,j}>0$. The factor $\sqrt{a_j}$ gives the required restricted measure, since quadratic homogeneity yields
\begin{equation}\label{eq:retained-measure}
 \MA(G_j)=a_j\mu_j=\mu|_{K_j}.
\end{equation}

We next show that these truncated limits converge uniformly. For $j<k$, Corollary~\ref{cor:removal} gives
\[
 0\le\frac{F_{P_N^{(k)}}-F_{P_N^{(j)}}}{\sqrt N}
 \le C_\Omega\sqrt{\frac{M_{N,k}-M_{N,j}}{N}}.
\]
Letting $N\to\infty$ in this inequality and in \eqref{eq:accepted-normalization}, we obtain
\begin{equation}\label{eq:truncated-cauchy}
 0\le G_k-G_j\le C_\Omega\sqrt{a_k-a_j}
 \quad\text{on }\overline\Omega.
\end{equation}
Thus $G_j$ is uniformly Cauchy. Its limit $G$ is continuous and concave on $\overline\Omega$, with zero boundary values. Aleksandrov stability applied to $-G_j$ and \eqref{eq:retained-measure} gives $\MA(G)=\mu$ in $\Omega$: every compactly supported test eventually lies in $K_j$. Comparison gives uniqueness, so $G=F_{\mu,\Omega}$. For comparison and stability, see \mooney{Propositions 3 and 5}{5}. This argument also constructs the solution for the present sampling law. Sending $k\to\infty$ in \eqref{eq:truncated-cauchy} yields
\begin{equation}\label{eq:continuum-tail}
 \|F_{\mu,\Omega}-G_j\|_\infty
 \le C_\Omega\sqrt{1-a_j}.
\end{equation}

Finally compare the original cloud with its truncation. Corollary~\ref{cor:removal} gives
\[
 \left\|\frac{F_{P_N}}{\sqrt N}-\frac{F_{P_N^{(j)}}}{\sqrt N}\right\|_\infty
 \le C_\Omega\sqrt{1-\frac{M_{N,j}}{N}}.
\]
For fixed $j$, combine this with \eqref{eq:accepted-normalization} and \eqref{eq:continuum-tail}:
\begin{equation}\label{eq:boundary-tail}
 \limsup_{N\to\infty}
 \left\|N^{-1/2}F_{P_N}-F_{\mu,\Omega}\right\|_\infty
 \le 2C_\Omega\sqrt{1-a_j}.
\end{equation}
Now let $j\to\infty$. This proves uniform convergence on $\overline\Omega$. Applying Aleksandrov stability once more gives vague convergence of the curvature measures in $\Omega$.
\end{proof}
The boundary-reaching result gives qualitative convergence. The quantitative discrepancy estimate and the sandpile approximation concern sources in a fixed $K\Subset\Omega$. Curvature convergence on $\overline\Omega$ also requires control of the total mass.

\needspace{8\baselineskip}\subsection*{Affine covariance and explicit profiles}
Uniqueness determines how the limiting solution changes under an affine transformation of the domain and measure. Transforming the supporting slopes gives its curvature, and comparison identifies the transformed roof. This covariance also carries the explicit disk solution to every ellipse.

\needspace{5\baselineskip}\begin{theorem}[Affine covariance of the limit]\label{thm:covariance}

Let $T(x)=Ax+b$, $A\in\operatorname{GL}_2(\R)$, and let $\mu$ be a finite nonnegative Borel measure on $\Omega$. Then
\[
 F_{T_\#\mu,T\Omega}=|\det A|^{1/2}F_{\mu,\Omega}\circ T^{-1}.
\]
More generally,
\[
 \MA(aF\circ T^{-1})=\frac{a^2}{|\det A|}\,T_\#\MA(F)
 \quad(a>0).
\]
The integral $J(\mu,\Omega)=\int_\Omega F_{\mu,\Omega}\,dx$ transforms by the factor $|\det A|^{3/2}$.
\end{theorem}
\begin{proof}
Put $G(y)=aF(T^{-1}y)$. Substitution in the supporting-plane inequality gives
\begin{equation}\label{eq:affine-supergradient}
 q\in\partial^+G(Tx)
 \quad\Longleftrightarrow\quad
 a^{-1}A^Tq\in\partial^+F(x),
 \qquad
 \partial^+G(Tx)=aA^{-T}\partial^+F(x).
\end{equation}
Indeed, substituting $Tz$ for the comparison point in
$G(Tz)\le G(Tx)+q\cdot(Tz-Tx)$ gives exactly the supporting-plane inequality for $F$ with slope $a^{-1}A^Tq$.

The linear map on slope space has area multiplier
$|\det(aA^{-T})|=a^2/|\det A|$. Taking unions of the superdifferentials over $T^{-1}B$, for a Borel set $B\subset T\Omega$, gives
\[
 \MA(G)(B)=\frac{a^2}{|\det A|}\MA(F)(T^{-1}B).
\]
This proves the formula for singular measures as well. The pushforward $T_\#$ transports the resulting measure from $\Omega$ to $T\Omega$.

Take $F=F_{\mu,\Omega}$ and $a=|\det A|^{1/2}$. The resulting $G$ is continuous, concave, and zero on $\partial(T\Omega)$, and its curvature is $T_\#\mu$. Existence for finite nonnegative data and zero boundary values is Theorem~\ref{thm:dirichlet}; uniqueness in Theorem~\ref{thm:comparison} identifies $G$ with $F_{T_\#\mu,T\Omega}$. Finally, ordinary change of variables in the base yields
\[
 J(T_\#\mu,T\Omega)
 =\int_\Omega |\det A|^{1/2}F_{\mu,\Omega}(x)|\det A|\,dx
 =|\det A|^{3/2}J(\mu,\Omega).
\]
The absolute determinant includes orientation-reversing maps.
\end{proof}

For the finite roofs, a change of variables must also preserve the lattice of allowable gradients. This holds when $A$ is unimodular.

\needspace{5\baselineskip}\begin{proposition}[Integral-affine covariance before the limit]\label{prop:lattice-covariance}

If $T(x)=Ax+b$ with $A\in\operatorname{GL}_2(\Z)$, then
\[
 F_{T\Omega,T(P)}=F_{\Omega,P}\circ T^{-1}.
\]
\end{proposition}
\begin{proof}
The piece $c+m\cdot x$ becomes $c+m\cdot A^{-1}(y-b)$, with gradient $A^{-T}m$. Both $A^{-T}$ and its inverse preserve $\Z^2$. Pullback gives a bijection between the admissible classes and preserves order, boundary values, and corners. Minimality in both directions gives equality.
\end{proof}
The translation vector $b$ is arbitrary. On a general convex domain, the tropical roofs have locally finite presentations. For the square-lattice sandpile, an affine change of coordinates acts on the underlying graph as well.

\needspace{5\baselineskip}\begin{corollary}[Uniform disk and ellipse sampling]\label{cor:disk-ellipse}

For uniform i.i.d. sampling in the unit disk $D$, the almost-sure limiting roof is
\[
 F_{\mathrm{unif},D}(x)=\frac{1-|x|^2}{2\sqrt\pi}.
\]
For $E=AD+b$, where $A\in\operatorname{GL}_2(\R)$ and $b\in\R^2$, it is
\[
 F_{\mathrm{unif},E}(y)=\frac{|\det A|^{1/2}}{2\sqrt\pi}
             \bigl(1-|A^{-1}(y-b)|^2\bigr).
\]
Both convergences are uniform on the closed domain, with vague interior curvature convergence.
\end{corollary}

\begin{proof}
Put $d=|\det A|$. The uniform measure on $E$ has density $(\pi d)^{-1}$, since $|E|=\pi d$. The displayed function is concave and zero on $\partial E$, and its Hessian satisfies
\[
 -D^2F_{\mathrm{unif},E}
 =\frac{d^{1/2}}{\sqrt\pi}A^{-T}A^{-1},
 \qquad
 \det(-D^2F_{\mathrm{unif},E})
 =\frac{d}{\pi(\det A)^2}=\frac1{\pi d}.
\]
It therefore solves the required Aleksandrov equation. Comparison identifies it with the limiting roof in Theorem~\ref{thm:iid-boundary}. Taking $A=I$, $b=0$ gives the disk. The same formula also follows from Theorem~\ref{thm:covariance}.
\end{proof}
For a disk $B_R$ of radius $R$, take $A=RI$ to obtain
\[
 F_{\mathrm{unif},B_R}(x)=\frac{R^2-|x|^2}{2\sqrt\pi R}.
\]

\needspace{8\baselineskip}\subsection*{Sandpile odometers and deficits}
For the sandpile approximation we return to a rational-slope polygon, with all sources in a fixed interior compact set. We combine the fixed-cloud approximation with the many-point limit by choosing a sufficiently fine mesh for each cloud.

\needspace{5\baselineskip}\begin{theorem}[Sandpile approximation along chosen meshes]\label{thm:sandpile}

Let $\Omega$ be the interior of a rational-slope polygon and $P_N\subset K\Subset\Omega$ strongly generic, with $|P_N|=N$ and empirical measures converging to $\mu$. There exist configuration-dependent meshes $0<h_N<1/N$ and proper roundings satisfying~\eqref{eq:proper-rounding} for which the sandpile odometers $H_N:=H_{h_N,P_N}$ satisfy
\[
 \sup_{v\in\Omega\cap h_N\Z^2}
 \left|\frac{h_N}{\sqrt N}H_N(v)-F_{\mu,\Omega}(v)\right|\longrightarrow0.
\]
In particular, $Nh_N^2<1/N$.
\end{theorem}
\begin{proof}
Keep $N$ fixed for the moment. Appendix~\ref{app:diagonal} states the proper-rounding condition and establishes, with the exterior-sink convention used here, the fixed-cloud convergence
\[
 \sup_{v\in\Omega\cap h\Z^2}|hH_{h,P_N}(v)-F_{\Omega,P_N}(v)|\longrightarrow0
 \qquad(h\to0).
\]
Choose $0<h_N<1/N$, small enough for this particular cloud that the displayed supremum is at most $1/N$, and choose the corresponding proper roundings. We may also make the rounded sources distinct and give each its full four-neighbor stencil inside $\Omega$. Then $Nh_N^2<1/N$, and
\begin{align*}
 \sup_{v\in\Omega\cap h_N\Z^2}
 \left|\frac{h_NH_{h_N,P_N}(v)}{\sqrt N}-F_{\mu,\Omega}(v)\right|
 &\le N^{-3/2}\\
 &\quad+\|N^{-1/2}F_{\Omega,P_N}-F_{\mu,\Omega}\|_\infty.
\end{align*}
The second term tends to zero by Theorem~\ref{thm:polygon-limit}. This proves convergence along the chosen meshes.
\end{proof}
The term $N^{-3/2}$ bounds the normalized discretization error; the remaining error is the many-point approximation term. The chosen mesh satisfies both $h_N<1/N$ and the fixed-cloud approximation bound, whose threshold can depend on all of $P_N$.

The odometer counts topplings. Grain conservation expresses the final deficit through its discrete Laplacian and the added sources. After normalization, the source term vanishes and summation by parts gives the measure $-\Delta F_{\mu,\Omega}$.

\needspace{5\baselineskip}\begin{theorem}[The deficit measure]\label{thm:deficit}

Along the meshes chosen in Theorem~\ref{thm:sandpile}, let $\varphi_N^\circ(v)$ be the final number of grains at the site $v$. Its deficit relative to height three is $d_N(v)=3-\varphi_N^\circ(v)$. Then
\[
 D_N=\frac{h_N}{\sqrt N}\sum_v d_N(v)\delta_v
 \rightharpoonup-\Delta F_{\mu,\Omega}
 \quad\hbox{vaguely in }\Omega.
\]
For a nonempty relatively compact Lipschitz window $W$ with $(-\Delta F_{\mu,\Omega})(\partial W)=0$, the average deficit in $W$, divided by $h_N\sqrt N$, tends to $(-\Delta F_{\mu,\Omega})(W)/|W|$.
\end{theorem}
\begin{proof}
Set $U_N=h_NH_N/\sqrt N$ on the lattice and extend it by zero outside the domain. The normalized odometer identity and discrete summation by parts give, for $\psi\in C_c^\infty(\Omega)$,
\begin{equation}\label{eq:deficit-test}
 \int\psi\,dD_N
 =-\frac{h_N}{\sqrt N}\sum_{j=1}^N\psi(p^h_{N,j})
  -\sum_v U_N(v)\Delta_{h_N}\psi(v).
\end{equation}
Here $p^h_{N,j}$ denotes the rounding at mesh $h_N$, and $\Delta_h$ is the unscaled four-neighbor Laplacian. Appendix~\ref{app:deficit} gives the full calculation and the passage from smooth tests to continuous compactly supported tests. The source term is bounded by $h_N\sqrt N\|\psi\|_\infty$, hence vanishes. Taylor expansion and a Riemann sum give
\[
 -\sum_v U_N(v)\Delta_{h_N}\psi(v)
 \longrightarrow-\int_\Omega F_{\mu,\Omega}\Delta\psi\,dx.
\]
Thus the limiting measure is the positive measure represented by $-\Delta F_{\mu,\Omega}$. Positivity of the deficits and a compactly supported cutoff give the local mass bounds needed for vague convergence.

For a nonempty continuity window $W$, vague convergence gives
$D_N(W)\to(-\Delta F_{\mu,\Omega})(W)$. Its Lipschitz boundary has planar measure zero, so $h_N^2\#(W\cap h_N\Z^2)\to|W|$. Writing $\overline d_N(W)$ for the average deficit, the exact normalization is
\begin{equation}\label{eq:window-deficit-normalization}
 \frac{\overline d_N(W)}{h_N\sqrt N}
 =\frac{D_N(W)}{h_N^2\#(W\cap h_N\Z^2)}.
\end{equation}
Taking the limit proves the window assertion.
\end{proof}

The roof $F$ determines two measures: its Monge--Amp\`ere curvature and the limiting sandpile deficit $-\Delta F$. Supporting-slope area defines the first, and grain conservation gives the second. For a smooth roof,
\[
 \MA(F)=\det(-D^2F)\,dx,
 \qquad -\Delta F=\operatorname{tr}(-D^2F)\,dx.
\]
For the continuum solution on the ellipse $E=AD+b$, with uniform probability and $d=|\det A|$, the two densities are
\begin{equation}\label{eq:ellipse-two-measures}
 \MA(F_{\mathrm{unif},E})=\frac{dy}{\pi d},
 \qquad
 -\Delta F_{\mathrm{unif},E}
 =\frac{\sqrt d}{\sqrt\pi}\operatorname{tr}(A^{-T}A^{-1})\,dy.
\end{equation}
For $A=\operatorname{diag}(2,1/2)$, the ellipse still has area $\pi$, and its curvature density remains $1/\pi$. Its Laplacian density is $17/(4\sqrt\pi)$, whereas the disk's is $2/\sqrt\pi$.

\appendix
\needspace{12\baselineskip}
\section{\texorpdfstring{\headingcolor{Algebra}{Hahn series and tropical interpolation}}{Hahn series and tropical interpolation}}\label{sec:A}

The height bound in Section~\ref{sec:2} begins with interpolation of the prescribed corners using small gradients. Imposing zero boundary values then turns the interpolant into an admissible competitor. Given $N$ marks $p_1,\ldots,p_N$, choose $M>N$ distinct integral gradients $m_1,\ldots,m_M$. The interpolation statement of Theorem~\ref{thm:interpolation} asks for a nonempty set $J\subset\{1,\ldots,M\}$ and real constants $c_j$ such that $q(x)=\min_{j\in J}\{c_j+\langle m_j,x\rangle\}$ satisfies
\begin{equation}\label{eq:hahn-target}
 \min_{j\in J}\{c_j+\langle m_j,p_i\rangle\}
 \quad\hbox{is attained at least twice, for every }i.
\end{equation}
The constants $c_j$ determine the vertical positions of the retained planes; slopes outside $J$ are unused. To find them, we solve a homogeneous linear system with matrix $(t^{\langle m_j,p_i\rangle})$ over a field of formal series. For a nonzero solution $z=(z_j)$, retain the nonzero coordinates and let $c_j$ be the least exponent of $z_j$. Cancellation in each row gives the repeated minimum in \eqref{eq:hahn-target}. The Hahn field allows the real exponents $\langle m_j,p_i\rangle$ arising from arbitrary marks.

\subsection*{Real exponents and a first term}\label{app:hahn-support}
A \emph{Hahn series} with complex coefficients and real exponents is a formal expression
\[
 f=\sum_{\alpha\in S}a_\alpha t^\alpha,\qquad
 a_\alpha\in\C\setminus\{0\},
\]
whose support $S\subset\R$ is \emph{well ordered}: every nonempty subset of $S$ has a least member, in the usual order of the real line. The empty support represents zero. This condition gives a first term to every nonzero series and to each of its nonzero tails.

\ifjournaledition
\else
For example, each of the following is allowed:
\[
 t^{-\sqrt2}+3t^\pi,\qquad
 1+t+t^2+\cdots,\qquad
 \sum_{n\ge1}t^{\,1-1/n}=1+t^{1/2}+t^{2/3}+\cdots.
\]
In the last example a subset has a least exponent because the corresponding set of positive integers has a least member. Its exponents accumulate at $1$. Well ordering therefore allows infinitely many terms below a fixed bound. The exponents $1/n$ have no least member, so $\sum_{n\ge1}t^{1/n}$ is not a Hahn series.

We treat $t$ as a formal symbol and compare coefficients, without assigning it a real value. The least exponent plays the role of an order of vanishing. Negative exponents are allowed, with $t^{-a}$ inverse to $t^a$.

Write $\mathbb K=\C((t^\R))$. Gaussian elimination requires a field. We justify multiplication and division by showing that each coefficient calculation is finite.
\fi

\pdfbookmark[1]{\ifjournaledition\else Why these series form a field\fi}{hahn-field}
\subsection*{\ifjournaledition\else Why these series form a field\fi}\label{app:hahn-field}
To solve the interpolation equations by linear algebra, we need addition, multiplication, and division within the class of Hahn series. Addition is coefficientwise. The support of a sum lies in the union of two well-ordered sets, which is again well ordered. Multiplication is defined by
\begin{equation}\label{eq:hahn-product}
 [t^\gamma](fg)=\sum_{\alpha+\beta=\gamma}a_\alpha b_\beta.
\end{equation}
\ifjournaledition
\else
Here $[t^\gamma]$ means the coefficient of $t^\gamma$. We must show that the coefficient sums are finite and that the product has well-ordered support.

If $A,B\subset\R$ are well ordered, then $A+B$ is well ordered. Indeed, a strictly decreasing sequence $a_n+b_n$ would have a subsequence on which the $a_n$ are nondecreasing, forcing the corresponding $b_n$ to decrease strictly. This is impossible. The subsequence exists by successively taking a minimum in the remaining tail of the sequence $(a_n)$. A subset of $\R$ that is not well ordered admits a strictly decreasing sequence, so this proves the assertion.

Further, a fixed $\gamma$ has only finitely many representations $\gamma=a+b$ with $(a,b)\in A\times B$. Infinitely many representations would give a strictly increasing sequence of distinct first coordinates; the second coordinates would then decrease strictly. Consequently \eqref{eq:hahn-product} is a finite sum, and its support lies in the well-ordered set $A+B$. Applying the same argument to three supports makes every coefficient in a triple product a finite sum. Associativity and distributivity follow from the ordinary identities in $\C$.
\fi

Division uses a geometric series. For $f\ne0$, factor out its first term:
\[
 f=a t^\lambda(1+h),\qquad a\ne0,\qquad
 \operatorname{supp}(h)\subset(0,\infty).
\]
If $h=0$, then $f^{-1}=a^{-1}t^{-\lambda}$. Otherwise its support $S$ has a least member $\delta>0$. Every exponent of $h^n$ is at least $n\delta$, so we may set
\begin{equation}\label{eq:hahn-inverse}
 f^{-1}=a^{-1}t^{-\lambda}\sum_{n=0}^{\infty}(-h)^n.
\end{equation}
\ifjournaledition
\else
We must check that the resulting support is well ordered and that each coefficient is a finite sum. Each $n$-fold sum $nS$ is well ordered by the multiplication argument; put $0S=\{0\}$. Below any real bound $B$, only finitely many of the sets $nS$ can occur, since $n\delta\le B$. Their union is well ordered: for a nonempty subset, choose one of its members $B$ and take the least element in its intersection with $(-\infty,B]$, a subset of a finite union of well-ordered sets. At any fixed exponent only finitely many powers contribute, and each coefficient within a power is already a finite sum.

Thus \eqref{eq:hahn-inverse} is a Hahn series. Multiplication by $f$ telescopes coefficientwise to $1$. Finitely many powers contribute to each exponent, although each power may have infinitely many exponents below a bound. The argument $n\delta\to\infty$ uses the real exponent group.

\smallskip
{\small\color{Ink}For general ordered exponent groups the corresponding summability theorem needs a further argument. See Poonen~\cite{Poo26}, \href{https://math.mit.edu/~poonen/papers/malcev.pdf\#page=3}{\emph{Units in Hahn--Mal'cev--Neumann rings}, Theorem 4.1\readerlocator{ (PDF p.~3)}} (14 January 2026). }
\fi

\pdfbookmark[1]{Taking the least exponent}{hahn-valuation}
\subsection*{Taking the least exponent}\label{app:hahn-valuation}
The least exponents of the series coefficients will become the heights of the interpolating planes. For a nonzero series, define its \emph{valuation} by
\[
 \operatorname{val}(f)=\min\operatorname{supp}(f),
 \qquad \operatorname{val}(0)=+\infty.
\]
If $f$ begins with $a t^\alpha$ and $g$ with $b t^\beta$, their product begins with $ab t^{\alpha+\beta}$. Both coefficients are nonzero. For a sum, the first coefficients may cancel. We therefore have
\begin{equation}\label{eq:hahn-valuation-rules}
 \operatorname{val}(fg)=\operatorname{val}(f)+\operatorname{val}(g),
 \qquad
 \operatorname{val}(f+g)\ge\min\{\operatorname{val}(f),\operatorname{val}(g)\}.
\end{equation}
\ifjournaledition
\else
The second formula is an equality when the two valuations differ. For example, $1+t$ and $-1+t^2$ both have valuation $0$, while their sum $t+t^2$ has valuation $1$. Cancellation moves the first surviving term to a larger exponent.
\fi

\begin{lemma}[The minimum valuation in a vanishing sum]\label{lem:hahn-cancellation}

In the Hahn field $\mathbb K=\C((t^\R))$, the valuation of a nonzero series is its least exponent. If finitely many nonzero series sum to zero, their least valuation is attained at least twice.
\end{lemma}
\begin{proof}
Let $\lambda$ be the least valuation. If only one summand has valuation $\lambda$, its coefficient of $t^\lambda$ survives in the sum: every other summand has coefficient zero there. The sum cannot then vanish.
\end{proof}

We now use a kernel vector to obtain a vanishing sum at every mark.

\begin{lemma}[A kernel gives repeated minima]\label{lem:tropical-dependence}

For an $N\times M$ real matrix $(a_{ij})$ with $M>N$, there are $c_j\in\R\cup\{+\infty\}$, not all infinite, such that each minimum $\min_j(a_{ij}+c_j)$ is attained at least twice.
\end{lemma}
\begin{proof}
The matrix $\widetilde A=(t^{a_{ij}})$ over $\mathbb K$ has more columns than rows, so its kernel contains a nonzero vector $z=(z_j)$. Each coordinate $z_j$ is a formal series. Put $c_j=\operatorname{val}(z_j)$, a real number when $z_j\ne0$ and $+\infty$ otherwise. In row $i$ we have
\[
 \sum_j t^{a_{ij}}z_j=0,\qquad
 \operatorname{val}(t^{a_{ij}}z_j)=a_{ij}+c_j.
\]
Discard the zero coordinates. Each remaining summand is nonzero in every row, so Lemma~\ref{lem:hahn-cancellation} gives the required repeated minimum. At least one coordinate remains, since $z\ne0$.
\end{proof}

For interpolation set $a_{ij}=\langle m_j,p_i\rangle$ and $J=\{j:z_j\ne0\}$. The minimum $q(x)=\min_{j\in J}(c_j+\langle m_j,x\rangle)$ is attained by at least two planes at each mark. Their gradients are distinct, so $q$ has a corner there. This proves Theorem~\ref{thm:interpolation}.

\pdfbookmark[1]{Two marks, three planes}{hahn-example}
\subsection*{Two marks, three planes}\label{app:hahn-example}
Take the marks $p_1=(0,0)$ and $p_2=(1,2)$, and the slopes $(0,0)$, $(1,0)$, $(0,1)$. The resulting matrix over the Hahn field and one of its kernel vectors are
\begin{equation}\label{eq:hahn-example}
 \widetilde A=\begin{pmatrix}1&1&1\\1&t&t^2\end{pmatrix},
 \qquad z=\begin{pmatrix}t^2-t\\1-t^2\\t-1\end{pmatrix}.
\end{equation}
\ifjournaledition
\else
The first row gives $(t^2-t)+(1-t^2)+(t-1)=0$. The second gives
\[
 (t^2-t)+t(1-t^2)+t^2(t-1)=0.
\]
Reading the first exponents of the three coordinates gives $(1,0,0)$. Hence
\[
 q(x,y)=\min\{1,x,y\}.
\]
At $p_1$ the three affine terms have values $(1,0,0)$; at $p_2$ they have values $(1,1,2)$. In each case the minimum occurs twice. In the second row of the lifted equation, the leading terms $-t$ and $t$ cancel. Taking valuations turns the cancellation of the leading terms into a repeated minimum (Figure~\ref{fig:hahn-kernel}).
\fi

\begin{figure}[htbp]\centering
\begin{tikzpicture}[x=1.45cm,y=1.45cm,>=stealth,font=\small]
 \fill[Algebra!6] (-1.15,-1.15)--(1,1)--(1,2.6)--(-1.15,2.6)--cycle;
 \fill[Geometry!6] (-1.15,-1.15)--(2.75,-1.15)--(2.75,1)--(1,1)--cycle;
 \fill[Analysis!6] (1,1) rectangle (2.75,2.6);
 \draw[->,gray!70] (-1.25,0)--(2.85,0) node[right]{$x$};
 \draw[->,gray!70] (0,-1.25)--(0,2.72) node[above]{$y$};
 \draw[gray!65,dashed] (0,2)--(1,2);
 \draw[gray!65,dashed] (0,1)--(1,1)--(1,0);
 \draw[very thick,Algebra,->] (1,1)--(-1.1,-1.1);
 \draw[very thick,Algebra,->] (1,1)--(1,2.58);
 \draw[very thick,Algebra,->] (1,1)--(2.73,1);
 \fill[Ink] (1,1) circle (1.7pt);
 \node[above right,Ink] at (1.04,1.03) {$(1,1)$};
 \fill[Probability] (0,0) circle (2.7pt);
 \node[below right,Probability] at (.04,-.08) {$p_1=(0,0)$};
 \fill[Probability] (1,2) circle (2.7pt);
 \node[right,Probability] at (1.06,2.06) {$p_2=(1,2)$};
 \node[Algebra] at (-.65,1.4) {$q=x$};
 \node[Geometry] at (1.85,-.7) {$q=y$};
 \node[Analysis] at (2.2,2.45) {$q=1$};
\end{tikzpicture}\par
\caption{The kernel vector in \eqref{eq:hahn-example} has valuations $(1,0,0)$, giving $q=\min\{1,x,y\}$. At $p_1$ the terms $x,y$ are minimal; at $p_2$ the terms $1,x$ are minimal. Both points are corners, although neither is a vertex of the curve.}\label{fig:hahn-kernel}
\end{figure}

\ifjournaledition
\else
The rays meet at $(1,1)$: they are $x=y\le1$, $x=1,\ y\ge1$, and $y=1,\ x\ge1$. The marks lie on the diagonal and vertical rays. A prescribed corner need not be a vertex.

Multiplying the kernel vector by $t^b$ adds $b$ to all three valuations. It translates the graph of $q$ vertically and leaves its corner locus fixed. This freedom will let us normalize the height of any interpolant.
\fi

\pdfbookmark[1]{\ifjournaledition\else Why the scale is a square root\fi}{hahn-count}
\subsection*{\ifjournaledition\else Why the scale is a square root\fi}\label{app:hahn-count}
The interpolation argument leaves the gradients free to choose. Taking enough of them near the origin bounds the oscillation of the resulting polynomial. For an integer $d\ge0$, the square $[-d,d]^2$ contains $(2d+1)^2$ integral gradients. For $N\ge1$, choose
\begin{equation}\label{eq:hahn-radius}
 r_N=\left\lceil\frac{\sqrt{N+1}-1}{2}\right\rceil,
 \qquad A_{r_N}=[-r_N,r_N]^2\cap\Z^2.
\end{equation}
The set $A_{r_N}$ contains at least $N+1$ slopes, all of norm at most $R_N=\sqrt2\,r_N$. Thus the interpolation system has more unknowns than equations, while $r_N=O(\sqrt N)$. Every plane retained in $q$ has one of these slopes.

The slope bound gives a Lipschitz bound for the interpolant. If the $j$th plane is active at $y$, then
\begin{equation}\label{eq:hahn-lipschitz}
 q(x)-q(y)\le\langle m_j,x-y\rangle\le R_N|x-y|.
\end{equation}
Interchanging $x$ and $y$ gives the opposite inequality. Thus $q$ is globally $R_N$-Lipschitz, including across its edges. On the compact set $\overline\Omega$, subtract its minimum and obtain
\begin{equation}\label{eq:hahn-height}
 \widehat q=q-\min_{\overline\Omega}q,
 \qquad
 0\le\widehat q\le\sqrt2\,r_N\operatorname{diam}(\Omega).
\end{equation}
The corner locus is unchanged. Since $r_N=O(\sqrt N)$, this is the height estimate of Theorem~\ref{thm:interpolation}. The constant depends only on $\operatorname{diam}\Omega$.

\subsection*{\ifjournaledition\else Imposing zero boundary values\fi}
The shifted interpolant $\widehat q$ has the prescribed corners and controlled height. To make it an admissible competitor, we impose zero boundary values while leaving it unchanged near the marks.

\ifjournaledition
\else
The tropical distance $D_\Omega$ of Proposition~\ref{prop:barriers} is positive in $\Omega$ and vanishes on its boundary. Choose a positive integer $\ell_{\rm cap}$ with
\[
 \ell_{\rm cap} D_\Omega(p)>\widehat q(p)\qquad(p\in P),
 \qquad E=\min\{\widehat q,\ell_{\rm cap}D_\Omega\}.
\]
Such an $\ell_{\rm cap}$ exists because $D_\Omega$ is positive at each of the finitely many marks. The strict inequalities persist near the marks, so $E$ agrees there with $\widehat q$ and preserves its corners. Both functions are nonnegative, and $\ell_{\rm cap}D_\Omega$ vanishes on the boundary. Thus $E$ is admissible. Minimality gives
\[
 0\le F_{\Omega,P}\le E\le\widehat q.
\]
The inequality $E\le\widehat q$ gives a height bound independent of $\ell_{\rm cap}$ and of the distance of the marks from the boundary. The removal estimate uses this bound. For marks in a fixed $K\Subset\Omega$, a positive lower bound for $D_\Omega$ on $K$ also controls $\ell_{\rm cap}$, as in Proposition~\ref{prop:barriers}.
\fi

{\small\color{Ink} Hahn's original construction is \cite{Hah07}; Poonen's theorem cited above gives the general inverse formula. For the surrounding theory of tropical matrix ranks, see M.~Develin, F.~Santos, and B.~Sturmfels, \href{https://arxiv.org/abs/math/0312114v2}{\emph{On the rank of a tropical matrix}}~\cite{DSS05}. }

\needspace{10\baselineskip}
\section{\texorpdfstring{\headingcolor{Analysis}{Supporting slopes and the Monge--Amp\`ere measure}}{Supporting slopes and the Monge--Ampere measure}}\label{sec:B}

Supporting-slope area defines the curvature of a concave roof $F$, including at its corners. For a smooth roof it gives $\det(-D^2F)\,dx$; for a polyhedral roof it gives the areas of the Newton polygons. This appendix constructs the measure and proves the comparison and stability results used to identify the limiting roof. The planar boundary estimate supplies continuity up to the boundary and the Dirichlet solution for a prescribed finite measure. Standard references are Aleksandrov~\cite{Ale58}, Rauch--Taylor~\cite{RT77}, Guti\'errez~\cite{Gut16}, Figalli~\cite{Fig17,Fig19}, and Mooney~\cite{Moo18}; for convex duality see~\cite{Roc70,Sch14}.

\subsection*{B1. From a roof to a convex function}\label{app:sign}
For a concave roof $F$, the function $v=-F$ is convex. We use this sign change to apply the convex theory of supporting planes. Let $\Omega\subset\R^2$ be bounded, open, and convex, with nonempty interior. For a finite convex function $v$ on $\Omega$, put
\[
 \partial v(x)=\{p\in\R^2:v(y)\ge v(x)+p\cdot(y-x)\text{ for all }y\in\Omega\}.
\]
\ifjournaledition
\else
Such a $p$ is the slope of a plane supporting the graph from below. The supporting-plane theorem supplies one at every interior point. At a differentiability point it is uniquely $\nabla v(x)$. At a nondifferentiability point, the subdifferential contains more than one slope.
\fi

Multiplying the supporting inequality by $-1$ gives
\begin{equation}\label{eq:sign-slopes}
 \partial v(x)=-\partial^+F(x).
\end{equation}
Reflection preserves area. We may therefore use the convex theory throughout and define the curvature of the roof by
\begin{equation}\label{eq:measure-definition}
 \MA(F)(E)=M_v(E):=|\partial v(E)|,
 \qquad \partial v(E)=\bigcup_{x\in E}\partial v(x),
\end{equation}
for Borel $E\subset\Omega$. Here and below $|\cdot|$ in slope space is ordinary Euclidean area. In particular a triangle with vertices $(0,0),(1,0),(0,1)$ has mass $1/2$.

For disjoint sets of points, their sets of supporting slopes overlap only in a set of area zero. This property, proved in B2, gives countable additivity.

\subsection*{The smooth case and the normalization}
For $v\in C^2(\Omega)$, the slope map is $x\mapsto\nabla v(x)$. The area formula, together with the almost-everywhere uniqueness of the contact point proved below, yields
\[
 M_v(E)=\int_E\det D^2v(x)\,dx,
 \qquad \MA(F)=\det(-D^2F)\,dx
 \quad(F\in C^2).
\]
Thus the definition agrees with the usual Monge--Amp\`ere equation wherever the Hessian exists classically. It also retains curvature concentrated on sets of zero planar area.

Multiplication of $F$ by $a>0$ multiplies every supporting slope by $a$; adding an affine function translates all slopes. Hence
\begin{equation}\label{eq:measure-homogeneity}
 \MA(aF)=a^2\MA(F),\qquad \MA(F+\ell)=\MA(F)
 \quad(\ell\text{ affine}).
\end{equation}
The first identity holds also for $a=0$. Hence dividing a roof by $\sqrt N$ divides curvature by $N$, as in Section~\ref{sec:5} and Section~\ref{sec:6}.

\pdfbookmark[1]{\ifjournaledition\else Why slope area is a measure\fi}{b-measure}
\subsection*{\ifjournaledition\else B2. Why slope area is a measure\fi}\label{ext:measure}
For slope area to define a measure, the slopes supporting the graph over disjoint sets must overlap only on a set of area zero. We prove this by showing that almost every supporting slope has a unique contact point.

The convex dual recovers this contact point. Define
\[
 v^*(p)=\sup_{x\in\Omega}\{p\cdot x-v(x)\}.
\]
\ifjournaledition
\else
A supporting affine function bounds $v$ below on the bounded domain, so this supremum is finite. Moreover, $v^*$ is globally Lipschitz, since all the vectors $x$ in the supremum lie in a bounded set. Rademacher's theorem therefore makes $v^*$ differentiable outside a set of area zero.

If $p\in\partial v(x)$, then $x$ attains the supremum, and
\[
 v^*(q)\ge v^*(p)+x\cdot(q-p)\qquad(q\in\R^2).
\]
Consequently, at a differentiability point of $v^*$ every contact point must equal $T(p):=\nabla v^*(p)$. A slope touching at two distinct points belongs to the exceptional null set.

At a differentiability point $p$, any convergent subsequence of a maximizing sequence has limit $T(p)$, by the same supporting inequality. If $T(p)\in\Omega$, continuity makes it a maximizer. Thus, outside the exceptional null set,
\fi

\begin{equation}\label{eq:dual-measure}
 p\in\partial v(E)\quad\Longleftrightarrow\quad T(p)\in E.
\end{equation}
The map $T$ is Borel on the differentiability set of $v^*$, since its coordinates are limits of difference quotients. Equation~\eqref{eq:dual-measure} identifies $M_v$ with the pushforward under $T$ of completed Lebesgue measure restricted to $\{p:T(p)\in\Omega\}$. This proves measurability and countable additivity.

To prove local finiteness, fix $K\Subset\Omega$ and choose $r>0$ so that $K+r\overline B\Subset\Omega$. Testing the supporting inequality at $x\pm re$ for every unit vector $e$ gives
\[
 |p|\le r^{-1}\operatorname{osc}_{K+r\overline B}v
 \qquad(x\in K, p\in\partial v(x)).
\]
Thus $\partial v(K)$ is bounded. The resulting locally finite Borel measure is Radon. The null-overlap assertion is also \mooney{Proposition 2}{4}.

\subsection*{A vertex and its Newton polygon}\label{app:newton-area}
\ifjournaledition
\else
For a finite maximum $v(x)=\max_j(a_j+p_j\cdot x)$, the subdifferential at $x$ is the convex hull of the active slopes. One inclusion follows by averaging support inequalities. For the reverse inclusion, separate a putative slope from that convex hull. Near $x$, the inactive planes remain below the maximum of the active planes. Testing in the separating direction therefore contradicts the proposed supporting inequality.
\fi

Take $F(x,y)=\min\{0,x,y\}$. At the origin its superdifferential is the triangle drawn below. Along each edge it is a segment, and on each face it is a point. Only the vertex carries area:
\[
 \MA(F)=\tfrac12\delta_0.
\]
\begin{center}
\begin{tikzpicture}[x=1.45cm,y=1.45cm,>=stealth,font=\small]
 \draw[->,Ink!60] (-.18,0)--(1.6,0) node[right]{$m_1$};
 \draw[->,Ink!60] (0,-.15)--(0,1.4) node[above]{$m_2$};
 \node[above] at (.7,1.75) {Slope space};
 \filldraw[fill=Analysis!12,draw=Analysis,thick] (0,0)--(1,0)--(0,1)--cycle;
 \node[below left] at (0,0) {$0$};\node[below] at (1,0) {$(1,0)$};\node[left] at (0,1) {$(0,1)$};
 \node[Analysis,right] at (1.55,.75) {$Q_0=\partial^+F(0),\quad |Q_0|=\tfrac12$};
\end{tikzpicture}
\end{center}
This Newton-area formula is used in the vertex and window counts of Section~\ref{sec:4}.

\pdfbookmark[1]{An atom and a smooth limit}{b-examples}
\subsection*{B3. An atom, and a sequence of atoms}\label{app:atoms}
Curvature may concentrate at a single corner. It may also be distributed among increasingly many atoms whose measures approach area measure. The following examples illustrate both possibilities.

\ifjournaledition
\else
Figure~\ref{fig:atoms-limit} shows the two examples. Consider the convex cone $v(x)=|x|-1$ on the unit disk $D$. At its vertex,
\[
 p\in\partial v(0)
 \quad\Longleftrightarrow\quad p\cdot y\le|y|\text{ for every }y\in D
 \quad\Longleftrightarrow\quad |p|\le1.
\]
Thus $\partial v(0)=\overline D$. Away from the origin the only slope is $x/|x|$, so all those slopes lie on the unit circle, which has area zero. Therefore
\fi

\begin{equation}\label{eq:cone-atom}
 M_v=\pi\delta_0,
 \qquad \MA(1-|x|)=\pi\delta_0.
\end{equation}
\ifjournaledition
\else
The Hessian away from the origin has radial eigenvalue $0$ and tangential eigenvalue $1/|x|$. Its determinant is zero. The disk $\partial v(0)$ has area $\pi$, giving the atom missed by the classical Hessian density.
\fi

Now let $a_h$ be the linear interpolation of $s^2/2$ at the points of $h\Z$, and set $v_h(x,y)=a_h(x)+a_h(y)$. On $[kh,(k+1)h]$ the slope of $a_h$ is $(k+\tfrac12)h$. At a grid vertex,
\begin{equation}\label{eq:grid-subgradient}
 \partial v_h(kh,\ell h)
 =[(k-\tfrac12)h,(k+\tfrac12)h]
  \times[(\ell-\tfrac12)h,(\ell+\tfrac12)h].
\end{equation}
Its area is $h^2$. Elsewhere the subdifferential is a segment or a point, so
\begin{equation}\label{eq:grid-measure}
 M_{v_h}=h^2\sum_{(k,\ell)\in\Z^2}\delta_{(kh,\ell h)}.
\end{equation}
\ifjournaledition
\else
For $s=kh+\theta h$, $0\le\theta\le1$, direct interpolation gives
\[
 a_h(s)-\tfrac12s^2=\tfrac12h^2\theta(1-\theta),
 \qquad 0\le v_h(x,y)-\tfrac12(x^2+y^2)\le\tfrac14h^2.
\]
Hence the functions converge uniformly, while their atomic curvature measures converge vaguely to area measure by Riemann sums. The limit has Hessian $I$ and determinant $1$.
\fi

\begin{figure}[tbp]\centering
\begin{tikzpicture}[x=1cm,y=1cm,font=\small,>=stealth]
 \begin{scope}[shift={(0,0)}]
  \node[font=\small\bfseries] at (0,2.6) {The cone};
  \node at (0,2.1) {$v(x)=|x|-1$};
  \filldraw[fill=Analysis!12,draw=Analysis,thick] (0,0) circle (1.05);
  \draw[->,Ink!60] (-1.3,0)--(1.35,0) node[right] {$m_1$};
  \draw[->,Ink!60] (0,-1.25)--(0,1.35) node[above] {$m_2$};
  \node[Analysis] at (.35,.42) {$\partial v(0)$};
  \node at (0,-1.55) {Slope space: area $\pi$};
  \node at (0,-2.15) {$M_v=\pi\delta_0$};
 \end{scope}
 \draw[Ink!35] (2.45,2.85)--(2.45,-5.2);
 \begin{scope}[shift={(5,0)}]
  \node[font=\small\bfseries] at (0,2.6) {Mesh $h$};
  \node at (0,2.1) {Each atom has mass $h^2$};
  \draw[Ink!60] (-1.15,-1.15) rectangle (1.15,1.15);
  \foreach \a in {-.6,0,.6}{
   \draw[Ink!25] (\a,-1.15)--(\a,1.15) (-1.15,\a)--(1.15,\a);
   \foreach \b in {-.6,0,.6}{\fill[Analysis] (\a,\b) circle (2.6pt);}}
  \fill[Probability] (0,0) circle (2.6pt);
  \node at (0,-1.55) {Atoms in the domain};
  \draw[->,Probability] (0,-1.85)--(0,-2.65) node[midway,right] {$\partial v_h(0)$};
  \draw[->,Ink!60] (-1,-3.3)--(1,-3.3) node[right] {$m_1$};
  \draw[->,Ink!60] (0,-4.1)--(0,-2.75) node[right=2pt] {$m_2$};
  \filldraw[fill=Probability!12,draw=Probability,thick] (-.3,-3.6) rectangle (.3,-3);
  \draw[<->,Probability] (-.3,-3.85)--(.3,-3.85) node[midway,below] {$h$};
  \node at (0,-4.55) {Slope square: area $h^2$};
 \end{scope}
 \begin{scope}[shift={(10.25,0)}]
  \node[font=\small\bfseries] at (0,2.6) {Mesh $h/2$};
  \node at (0,2.1) {Each atom has mass $h^2/4$};
  \draw[Ink!60] (-1.15,-1.15) rectangle (1.15,1.15);
  \foreach \a in {-.9,-.6,-.3,0,.3,.6,.9}{
   \draw[Ink!25] (\a,-1.15)--(\a,1.15) (-1.15,\a)--(1.15,\a);
   \foreach \b in {-.9,-.6,-.3,0,.3,.6,.9}{\fill[Analysis] (\a,\b) circle (1.3pt);}}
  \fill[Probability] (0,0) circle (1.3pt);
  \node at (0,-1.55) {The same domain window};
  \draw[->,Probability] (0,-1.85)--(0,-2.65) node[midway,right] {$\partial v_{h/2}(0)$};
  \draw[->,Ink!60] (-1,-3.3)--(1,-3.3) node[right] {$m_1$};
  \draw[->,Ink!60] (0,-4.1)--(0,-2.75) node[right=2pt] {$m_2$};
  \filldraw[fill=Probability!12,draw=Probability,thick] (-.15,-3.45) rectangle (.15,-3.15);
  \draw[<->,Probability] (-.15,-3.85)--(.15,-3.85) node[midway,below] {$h/2$};
  \node at (0,-4.55) {Slope square: area $h^2/4$};
 \end{scope}
 \node at (7.625,-5.35) {$\displaystyle h^2\sum_{k,\ell}\delta_{(kh,\ell h)}\ \rightharpoonup\ dx\qquad(h\to0)$};
\end{tikzpicture}
\caption{Two examples of Aleksandrov measure. For the cone, the supporting disk at the origin has area $\pi$, giving one atom. For the grid interpolants, each supporting square has area $h^2$, giving the measure $h^2\sum_{k,\ell}\delta_{(kh,\ell h)}$. Halving the mesh quadruples the atom density and divides the mass of each atom by four. The two domain panels show the same window; their supporting squares are drawn below at the same scale. As $h\to0$, these Riemann sums converge vaguely to area measure.}
\label{fig:atoms-limit}
\end{figure}

Theorem~\ref{thm:stability} includes this example.

\pdfbookmark[1]{Comparison and uniqueness}{b-comparison}
\subsection*{B4. Comparison and uniqueness}\label{ext:uniqueness}
The following comparison principle gives uniqueness for prescribed curvature and boundary values.

\begin{theorem}[Comparison]\label{thm:comparison}
Let $\Omega\subset\R^2$ be bounded, open, and convex, with nonempty interior. Let $v,w\in C(\overline\Omega)$ be convex, with $v\le w$ on $\partial\Omega$. If $M_v\ge M_w$ as measures in $\Omega$, then $v\le w$ throughout $\Omega$. Equivalently, for continuous concave functions on $\overline\Omega$,
\begin{equation}\label{eq:comparison-concave}
 F\ge G\text{ on }\partial\Omega,\quad \MA(F)\ge\MA(G)
 \quad\Longrightarrow\quad F\ge G\text{ in }\Omega.
\end{equation}
\end{theorem}
Applying comparison twice gives uniqueness for equal curvature measures and equal boundary values. The equal-boundary formulation is \mooney{Proposition 3}{5}.

\begin{proof}
We use Theorem~\ref{thm:stability}, whose proof in B6 follows directly from convergence of supporting planes. We first establish the geometric inclusion that gives comparison.

\textit{The geometric step.} Suppose $E=\{w<v\}\Subset\Omega$. If a plane of slope $p$ supports $v$ at a point of $E$, then $w$ lies strictly below that plane somewhere in $E$, while on $\partial E$ it lies above or on the plane. Lower the plane until it first touches $w$. The contact is interior to $E$. Convexity turns this local support into a support on all of $\Omega$. Thus
\begin{equation}\label{eq:comparison-inclusion}
 \partial v(E)\subset\partial w(E).
\end{equation}
Every slope supporting $v$ at a point of $E$ therefore also supports $w$ at a point of $E$.

\begin{center}
\begin{tikzpicture}[x=2cm,y=1.2cm,>=stealth,font=\small]
 \draw[->,Ink!60] (-1.1,0)--(1.2,0) node[right] {$s$};
 \draw[Analysis,thick,domain=-1:1,samples=45] plot (\x,{\x*\x-1});
 \draw[Geometry,thick,domain=-1:1,samples=45] plot (\x,{.5*(\x*\x-1)});
 \draw[dashed,domain=-.7:.85] plot (\x,{.5*\x-.625});
 \draw[dashed,domain=-.5:.85] plot (\x,{.5*\x-1.0625});
 \draw[->] (.72,-.265)--(.72,-.7025);
 \fill[Geometry] (.5,-.375) circle (1.5pt);\fill[Analysis] (.25,-.9375) circle (1.5pt);
 \node[Geometry] at (1.42,-.35) {$v$};\node[Analysis] at (1.42,-.9) {$w$};
 \node[left] at (-.7,-1.22) {slope $1/2$};
\end{tikzpicture}\par
{\small A one-dimensional illustration of the support inclusion: the parallel lines of slope $1/2$ support $v(s)=(s^2-1)/2$ and $w(s)=s^2-1$.}
\end{center}

\ifjournaledition
\else
To obtain a contradiction from \eqref{eq:comparison-inclusion}, we need a strict measure inequality. Choose $R>\sup_{\overline\Omega}|x|$ and put
\[
 v_\varepsilon=v+\varepsilon(|x|^2-R^2).
\]
For a positive semidefinite $2\times2$ matrix $A$,
\[
 \det(A+2\varepsilon I)=\det A+2\varepsilon\operatorname{tr}A+4\varepsilon^2
 \ge\det A+4\varepsilon^2.
\]
Convolve $v$ on interior subdomains and pass to the limit using the stability theorem proved below. The smooth density formula gives
\[
 M_{v_\varepsilon}\ge M_v+4\varepsilon^2\,dx.
\]
\fi
If $v>w$ somewhere, then for small $\varepsilon>0$ the open set $E_\varepsilon=\{w<v_\varepsilon\}$ is nonempty and relatively compact, because $v_\varepsilon<w$ on the boundary. Applying \eqref{eq:comparison-inclusion} there gives
\[
 M_v(E_\varepsilon)+4\varepsilon^2|E_\varepsilon|
 \le M_{v_\varepsilon}(E_\varepsilon)
 \le M_w(E_\varepsilon)\le M_v(E_\varepsilon),
\]
a contradiction. These masses are finite because $E_\varepsilon$ is relatively compact.
\end{proof}

\pdfbookmark[1]{The boundary estimate and the one-third modulus}{b-boundary}
\subsection*{B5. The Aleksandrov boundary estimate}\label{ext:maximum}
For a nonnegative concave roof that vanishes on the boundary, a bound on total curvature controls its height near the boundary. A value $h$ at distance $\delta$ from the boundary forces a set of supporting slopes of area at least $h^2/(D\delta)$, where $D$ is the diameter of the domain. 

Let $v\in C(\overline\Omega)$ be convex and zero on the boundary. Fix $x\in\Omega$, and write $h=-v(x)\ge0$, $\delta=\operatorname{dist}(x,\partial\Omega)$, and $D=\operatorname{diam}\Omega$. Assume $h>0$. The slopes of planes through $(x,-h)$ lying below the boundary form the convex body
\[
 K_x=\{p:p\cdot(y-x)\le h\text{ for all }y\in\overline\Omega\}.
\]
For $p\in\operatorname{int}K_x$, the function $v(y)-p\cdot(y-x)$ is strictly larger on the boundary than its value $-h$ at $x$. Its minimum occurs inside. Thus $\operatorname{int}K_x\subset\partial v(\Omega)$.

Choose a nearest boundary point $z$ and put $e=(z-x)/\delta$. A supporting line at $z$ also supports the ball $B_\delta(x)\subset\Omega$, so $(y-x)\cdot e\le\delta$ throughout the domain. Hence $K_x$ contains $(h/\delta)e$. It also contains the disk of radius $h/D$, because $|y-x|\le D$. The triangle with vertices $(h/\delta)e$ and $\pm(h/D)e^\perp$ therefore lies in $K_x$ and has area $h^2/(D\delta)$. We obtain the planar Aleksandrov estimate
\begin{equation}\label{eq:aleksandrov-boundary}
 |v(x)|^2\le D\,\operatorname{dist}(x,\partial\Omega)\,M_v(\Omega).
\end{equation}
For $h=0$ it is immediate. This estimate gives the boundary bound used in Corollary~\ref{prop:holder}. The tropical barrier $D_\Omega$ gives the boundary bound in Proposition~\ref{prop:compactness}, as explained in B8.

\label{app:interior-slopes}A common mass bound now gives $0\le u_N(x)\le A\sqrt{\operatorname{dist}(x,\partial\Omega)}$ and a common height bound $H$ for the concave normalized roofs $u_N$. On the inner parallel body $\Omega_\delta=\{x:\operatorname{dist}(x,\partial\Omega)\ge\delta\}$, test a supporting slope at $x\pm\delta e$ to obtain
\begin{equation}\label{eq:interior-slope}
 |p|\le H/\delta,\qquad
 \operatorname{Lip}(u_N|_{\Omega_\delta})\le H/\delta.
\end{equation}
The set $\Omega_\delta$ is convex: it is the set of centers of closed radius-$\delta$ disks contained in $\overline\Omega$.

\begin{corollary}[One-third modulus on a polygon]\label{prop:holder}
Let $\Omega$ be the interior of a fixed rational-slope polygon, and fix $K\Subset\Omega$. For strongly generic sets $P_N\subset K$ with $|P_N|=N$, the normalized roofs $u_N=N^{-1/2}F_{\Omega,P_N}$ satisfy
\[
 |u_N(x)-u_N(y)|\le C_{\Omega,K}|x-y|^{1/3}
 \qquad(x,y\in\overline\Omega).
\]
\end{corollary}
\begin{proof}
Theorem~\ref{thm:total-mass} and quadratic homogeneity give
\[
 M_{-u_N}(\Omega)=1-N^{-1}+\frac{D_{\rm term}}{2N}\le C_{\Omega,K}.
\]
Thus \eqref{eq:aleksandrov-boundary} and \eqref{eq:interior-slope} provide uniform $A,H$. Put $r=|x-y|$. For $0<r\le1$, take $\delta=r^{2/3}$, so $r\le\delta$. If one point is within $2\delta$ of the boundary, the other is within $3\delta$ and the boundary bound gives $|u_N(x)-u_N(y)|\le A\sqrt{3\delta}$. Otherwise both points lie in the convex set $\Omega_{2\delta}$, and the slope bound gives $|u_N(x)-u_N(y)|\le Hr/\delta$. Both bounds are constant multiples of $r^{1/3}$. The case $r=0$ is immediate; for $r>1$, use the height bound.
\end{proof}
The rational-slope hypothesis gives the uniform total-mass bound in Theorem~\ref{thm:total-mass}.
\ifjournaledition\else
The exponent arises by balancing $\delta^{1/2}$ with $r/\delta$. The boundary-barrier argument for general domains is recalled on page~\pageref{app:compactness-return}.
\fi

\pdfbookmark[1]{Stability under local uniform convergence}{b-stability}
\subsection*{B6. Stability under local uniform convergence}\label{ext:stability}
Local uniform convergence of the functions determines the limit of their curvature measures inside the domain. This is the step that identifies the curvature of a subsequential limit of normalized roofs.

\begin{theorem}[Stability]\label{thm:stability}
If finite convex functions $v_j$ converge to $v$ locally uniformly on a common open convex domain $\Omega$, then
\begin{equation}\label{eq:stability}
 \int_\Omega\phi\,dM_{v_j}\longrightarrow\int_\Omega\phi\,dM_v
 \qquad(\phi\in C_c(\Omega)).
\end{equation}
Equivalently, locally uniform convergence of concave roofs gives vague convergence of their curvature measures in $\Omega$.
\end{theorem}
For this local form of stability, see \href{https://www.bourbaki.fr/TEXTES/Exp1148-Figalli.pdf\#page=5}{Figalli, \emph{On the Monge--Amp\`ere equation}, Proposition 3.2\readerlocator{ (PDF p.~5)}}~\cite{Fig19}; compare \mooney{Proposition 5}{5}.

\begin{proof}
\ifjournaledition
\else
We compare slope sets over compact sets and relatively compact open sets. Limits of supporting planes support the limit function, and a slope with a unique interior contact point also supports each sufficiently close approximating function at a nearby point.

\textit{The upper bound on compact sets.} Fix $K\Subset\Omega$. The interior slope bound gives a common bound on $\partial v_j(K)$. If $x_j\in K$, $x_j\to x$, and $p_j\in\partial v_j(x_j)$ with $p_j\to p$, then passing to the limit in
\[
 v_j(y)\ge v_j(x_j)+p_j\cdot(y-x_j)
\]
for each $y\in\Omega$ gives $p\in\partial v(x)$. The set $\partial v(K)$ is compact. Consequently, for every $\eta>0$, eventually
\[
 \partial v_j(K)\subset\partial v(K)+\eta\overline B.
\]
Taking areas and then $\eta\downarrow0$ gives
\fi

\begin{equation}\label{eq:stability-upper}
 \limsup_j M_{v_j}(K)\le M_v(K).
\end{equation}
\ifjournaledition\else
The area limit holds because the neighborhoods decrease to the compact set.

\textit{The lower bound on relatively compact open sets.}
\fi
Let $G\Subset\Omega$ be open. By B2, almost every $p\in\partial v(G)$ has a unique contact point $x\in G$. Choose a closed ball about $x$ contained in $G$. On its boundary, $v(y)-p\cdot y$ is strictly larger than its value at $x$. Uniform convergence on the ball preserves this strict inequality, so $v_j(y)-p\cdot y$ has an interior minimum there for all large $j$. Convexity makes this minimum global, which means that $p\in\partial v_j(G)$. Fatou's lemma applied to the indicator functions of these slope sets gives

\begin{equation}\label{eq:stability-lower}
 M_v(G)\le\liminf_j M_{v_j}(G).
\end{equation}
The uniform slope bound gives a common mass bound on each compact subset of $\Omega$. Together with \eqref{eq:stability-upper} and \eqref{eq:stability-lower}, this is the local Portmanteau criterion for \eqref{eq:stability}.
\end{proof}

\pdfbookmark[1]{The Dirichlet problem with affine boundary values}{b-existence}
\subsection*{B7. Prescribing the curvature}\label{ext:existence}
We need a Dirichlet theorem for finite measures, including atoms.

\begin{theorem}[Dirichlet problem with affine boundary values]\label{thm:dirichlet}
Let $\Omega\subset\R^2$ be bounded, open, and convex with nonempty interior, let $\mu$ be a finite nonnegative Borel measure on $\Omega$, and let $\ell$ be affine. There is a unique convex $v\in C(\overline\Omega)$ such that
\[
 M_v=\mu\text{ in }\Omega,\qquad v=\ell\text{ on }\partial\Omega.
\]
In particular, with $\ell=0$ there is a unique continuous concave roof $F_{\mu,\Omega}=-v$, nonnegative in $\Omega$, satisfying
\begin{equation}\label{eq:dirichlet-roof}
 \MA(F_{\mu,\Omega})=\mu,\qquad
 F_{\mu,\Omega}|_{\partial\Omega}=0.
\end{equation}
\end{theorem}
This is \mooney{Theorem 1 and Remark 1, Section 3.3}{6}, for affine boundary data. His ``bounded Borel measure'' means finite total mass; atoms are allowed.

\begin{proof}[Proof by atomic approximation]
Subtracting $\ell$ reduces the problem to zero boundary data. We first solve it for finitely many atoms on a polygon, then approximate the domain and the measure, using \eqref{eq:aleksandrov-boundary} and Theorem~\ref{thm:stability}.

\ifjournaledition
\else
For a finite atomic measure on a polygon, we use the polyhedral Dirichlet construction in Mooney\textquotesingle s Theorem 1 cited above.

Approximate $\mu$ by finite atomic measures
\[
 \mu_j=\sum_i a_{ij}\delta_{x_{ij}},\qquad a_{ij}>0,
 \qquad \mu_j(\Omega)\le\mu(\Omega),
\]
converging vaguely to $\mu$. Choose an increasing sequence of open polygonal domains $\Omega_j\Subset\Omega$ exhausting the domain, with the support of $\mu_j$ inside $\Omega_j$. The finite Dirichlet construction chooses the heights of vertices above the $x_{ij}$ so that their polygons of supporting slopes have the prescribed areas $a_{ij}$. Taking the convex envelope with boundary height zero produces a convex $v_j$ with $M_{v_j}=\mu_j$ in $\Omega_j$.

The estimate in B5 gives, for $x\in\Omega_j$,
\[
 |v_j(x)|^2\le\operatorname{diam}(\Omega_j)
 \operatorname{dist}(x,\partial\Omega_j)\,\mu_j(\Omega_j)
 \le D\operatorname{dist}(x,\partial\Omega)\,\mu(\Omega).
\]
These functions have a common height bound. On each fixed interior compact set they also have a common slope bound, once $j$ is large. A diagonal Arzel\`a--Ascoli argument gives a locally uniform convex limit $v$. The displayed bound makes $v$ tend continuously to zero at every boundary point. Stability on interior subdomains gives $M_v=\mu$, and comparison gives uniqueness.
\fi

The boundary estimate also applies along flat sides, so the construction gives the asserted boundary values on every bounded convex domain.
\end{proof}

The tropical limits also construct solutions for their limiting measures. Theorem~\ref{thm:dirichlet} defines $F_{\mu,\Omega}$ independently, and Theorem~\ref{thm:comparison} identifies the two constructions.

\pdfbookmark[1]{Return to compactness and the limiting roof}{b-return}
\subsection*{B8. Compactness and convergence}\label{app:compactness-return}

\ifjournaledition
\else
On a rational polygon, Corollary~\ref{prop:holder} supplies equicontinuity. On a general convex domain, Proposition~\ref{prop:barriers} gives
\[
 0\le u_N\le C_{\Omega,K}D_\Omega,\qquad \|u_N\|_\infty\le H,
\]
with $D_\Omega$ continuous and zero on the boundary. The barrier makes all roofs small in a common boundary collar; \eqref{eq:interior-slope} bounds their slopes inside. Together these give equicontinuity and Arzel\`a--Ascoli compactness. 

A uniform cluster limit $F$ has curvature $\mu$ by stability and the discrepancy estimate, and boundary value zero by the barrier. Comparison identifies it with $F_{\mu,\Omega}$. Every cluster limit agrees, so the full sequence converges. See Theorem~\ref{thm:polygon-limit}, Theorem~\ref{thm:fixed-exhaustion}, and Theorem~\ref{thm:iid-boundary}.

\fi

\subsection*{Curvature escaping to the boundary}\label{app:boundary-escape}
Uniform convergence of roofs alone does not control total curvature. Here is an explicit example on $\Omega=(-1,1)^2$. For $0<r<1$, put $h_r=\sqrt{(1-r^2)/2}$ and
\[
 F_r(x,y)=h_r\min\left\{\frac{1-x}{1-r},\frac{1+x}{1+r},1-y,1+y\right\}.
\]
This concave roof is zero on the boundary, and $\|F_r\|_\infty=h_r\to0$. Its only interior vertex is $(r,0)$. The four active gradients form a diamond of area
\[
 h_r^2\left(\frac1{1-r}+\frac1{1+r}\right)=1.
\]
Hence $\MA(F_r)=\delta_{(r,0)}$. These measures converge vaguely to zero inside $\Omega$, while on $\overline\Omega$ they converge to the unit atom at $(1,0)$. These slopes are real; the example uses the general concave class of Aleksandrov theory.

\ifjournaledition
\else
For Theorem~\ref{thm:polygon-limit}, the total-mass identity prevents this loss of curvature to the boundary. Write $\nu_N=\MA(u_N)$ and choose $0\le\chi\le1$ in $C_c(\Omega)$ with $\chi=1$ near $K\supset\operatorname{supp}\mu$. Then
\[
 \int(1-\chi)\,d\nu_N
 =\nu_N(\Omega)-\int\chi\,d\nu_N\longrightarrow1-1=0.
\]
Test $\chi\phi$ by vague convergence and bound the remaining integral by $\|\phi\|_\infty\int(1-\chi)\,d\nu_N$. This proves weak convergence on the closed polygon, as in Theorem~\ref{thm:polygon-limit}.
\fi
For sampling throughout $\Omega$, Theorem~\ref{thm:iid-boundary} gives vague convergence of the curvature measures inside the domain.

\begingroup
\fontsize{10.5}{12.6}\selectfont
\ifjournaledition
\else
\setlength{\parskip}{3pt plus 1pt}
\fi
\setlength{\abovedisplayskip}{7pt plus 2pt minus 2pt}
\setlength{\belowdisplayskip}{7pt plus 2pt minus 2pt}
\needspace{12\baselineskip}\section{\texorpdfstring{\headingcolor{Geometry}{Finite parameters and generic tropical curves}}{Finite parameters and generic tropical curves}}\label{sec:C}

How does the least roof change when its sources move? For a fixed number of sources confined to an interior compact set of a rational-slope polygon, the height bound and the bound on total side quasidegree place all essential gradients in a common finite set. The roof can therefore be described by finitely many coefficients. We show that these coefficients depend piecewise affinely on the source positions and that the incidences of the curve are locally constant outside exceptional sets of positive codimension. These facts give the genericity theorem and dimension count used in Section~\ref{sec:3}.

Fix $N\ge1$, a compact convex polygon $\Delta$ with rational side slopes, and an open polygon $O\Subset\Delta^\circ$ containing the sources. Write $F_P=F_{\Delta^\circ,P}$ and take configurations to be ordered, with distinct points. The common finite gradient set $A_N\subset\Z^2$ may depend on $N$, $\Delta$, and $O$. For each allowed gradient $m$, lower its plane until it first touches the roof, possibly at the boundary. The resulting \emph{canonical coefficient} $\kappa_m(F_P)$ is determined by the function even when the plane is not essential, that is, does not attain the minimum on an open cell.

\ifjournaledition
\else
With the slopes fixed, nonnegativity, zero boundary values, and the condition that a mark lie on the corner locus are finite combinations of affine equalities and inequalities. The polygonal choice of $O$ gives the same form to the source conditions; its vertices may be chosen rational. Sets described in this way are called semilinear. Fourier--Motzkin elimination preserves this class under projection; see Schrijver~\cite[\S12.2]{Sch86}. It follows that the graph of the map from source positions to canonical coefficients is semilinear. The coefficients depend affinely on the source coordinates on each piece of a finite partition.

The sets where specified planes attain the minimum determine the embedded curve: its vertices and edges, gradients and weights, cyclic orders, edges containing the marks, and order of boundary contacts. On the open strata these data persist under independent motions of the marks. A source at a vertex or two sources on one edge impose extra incidence conditions of positive codimension, giving the open full-measure generic set of Section~\ref{sec:3}. The active-face description also records all four incident germs at a transverse node; strong genericity does not exclude such nodes. For a general convex domain, we impose these conditions on a countable family of rational-coordinate polygons. The resulting universal genericity holds on a set of full measure. C5 gives the countable construction used for exhaustion and sampling.

\fi

\needspace{10\baselineskip}\subsection*{C1. Canonical coefficients}\pdfbookmark[1]{C1. Canonical coefficients}{appc1}
To compare roofs as their sources move, we need one finite set of gradients for all the configurations under consideration. Once this set is fixed, the lowest supporting plane of each allowed gradient gives a unique coefficient for the roof.

\needspace{5\baselineskip}\begin{lemma}[A finite common set of slopes]\label{lem:slope-support}

For fixed $N$, a rational-slope polygon $\Delta$, and source configurations in a fixed relatively compact open set $O\Subset\Delta^\circ$, all least roofs can be represented with slopes in one finite set $A_N\subset\Z^2$.
\end{lemma}
\begin{proof}
Let $n_r$ be the primitive inward normal to side $S_r$, and put $b_r=m_{F_P}(S_r)n_r$. The complexity estimate, Theorem~\ref{thm:complexity}, gives $\sum_r m_{F_P}(S_r)\le C_{\Delta,O}\sqrt N$.

Every essential slope lies in $\operatorname{conv}\{b_r\}$. Indeed, let $m$ be the gradient on an open affine cell and choose a point $x$ in that cell. If $m$ lay outside this convex hull, strict separation would give a direction $v$ with $m\cdot v<\min_r b_r\cdot v$. A small perturbation of $v$ preserves this inequality and makes the ray $x+tv$ leave $\Delta$ through the relative interior of a side $S_r$. The restriction of $F_P$ to this segment is concave. Its initial derivative is $m\cdot v$, and its derivative near the endpoint is $b_r\cdot v$, contradicting the fact that its derivative is nonincreasing. This proves the boundary-gradient fact of \href{https://arxiv.org/pdf/1706.03062v1\#page=12}{Kalinin--Shkolnikov, Remark 9.7\readerlocator{ (PDF p.~12)}}~\cite{KS18}.

The bound on total side quasidegree therefore places every essential slope in a fixed bounded set. Its lattice points form $A_N$.
\end{proof}
The finite slope set depends on $N$, $\Delta$, and $O$, so it may grow along an exhaustion.

Write $f_c(x)=\min_{m\in {A_N}}(c_m+m\cdot x)$. A redundant plane may be raised without changing a roof. We remove this freedom by lowering every allowed plane until it first touches the roof, possibly on the boundary.
\needspace{5\baselineskip}\begin{lemma}[Canonical coefficients]\label{lem:completion}

For each allowed slope $m$, set $\kappa_m(f)=\max_{x\in\Delta}(f(x)-m\cdot x)$. These canonical coefficients are uniquely determined by $f$, and their affine terms have minimum $f$ on the common slope set.
\end{lemma}
\begin{proof}
If $f=\min_m(a_m+m\cdot x)$, then $\kappa_m(f)\le a_m$. At a point where the $m$th term attains the minimum, the reverse inequality follows by evaluating the maximum there. Thus replacing $a_m$ by $\kappa_m(f)$ leaves at least one active term unchanged at every point, and every new term remains above $f$. Their minimum is $f$. The defining maximum fixes each coefficient, including those of redundant terms. The zero boundary values fix the additive constant.
\end{proof}

\needspace{9\baselineskip}\subsubsection*{Repeated one-point relaxations}\label{app:wave-convergence}
\pdfbookmark[2]{Repeated one-point relaxations}{wave-convergence}
Canonical coefficients also explain the construction in Section~\ref{sec:2}. For this argument we allow any bounded open convex domain $\Omega$ and keep a finite cloud fixed. Set $a_m=\kappa_m(f)=\max_{\overline\Omega}(f-m\cdot x)$ for every $m\in\Z^2$.

\emph{One update.} Suppose that $m_*$ is the unique active slope at $p$, let $L(x)=a_{m_*}+m_*\cdot x$, and let $H$ be the minimum of the other planes. The proposed update is
\[
 h=\min\{L+\delta,H\},\qquad \delta=H(p)-f(p)>0.
\]
This minimum is nonnegative and locally tropical. We check its boundary values. If $m_*\ne0$, then $2L\ge f$ and the canonical $2m_*$-plane lies below $2L$. Since this plane occurs in $H$, we obtain
\[
 f\le h\le\min\{2L,H\}\le2f.
\]
Thus $h$ tends to zero at the boundary. If $m_*=0$ and $f\ne0$, the active constant has height $\max f>0$. In a boundary collar $f<\max f$, so raising the constant leaves $f$ unchanged there. Finally, if $f=0$, direct substitution gives
\[
 h=\min\{D_\Omega,D_\Omega(p)\},
\]
which has zero boundary values by the properties of $D_\Omega$. At $p$, the raised plane meets a plane of a different slope, so $h$ has the required corner.

To prove that $h$ is the least admissible roof above $f$ with a corner at $p$, let $g\ge f$ have such a corner. Canonical coefficients respect order, hence $\kappa_m(g)\ge a_m$ for every $m$. If $\kappa_{m_*}(g)<H(p)-m_*\cdot p$, that plane would be uniquely active at $p$. Thus its coefficient reaches at least this threshold, and every coefficient of $g$ is at least the corresponding coefficient in the presentation of $h$. It follows that $g\ge h$, proving $h=G_pf$. This is the coefficient rule of \href{https://arxiv.org/pdf/1706.03062v1\#page=7}{Kalinin--Shkolnikov, Lemma 5.9\readerlocator{ (PDF pp.~7--8)}}~\cite{KS18}; the zero-roof formula is also Corollary 5.11 there.

\emph{Repeated visits.} Let $P$ be nonempty, let $f_0$ be an initial tropical roof, and choose $q_k\in P$ so that every mark appears infinitely often. Put $f_{k+1}=G_{q_k}f_k$ and $U=G_Pf_0$. Since $U$ is admissible at every stage,
\[
 f_0\le f_k\le f_{k+1}\le U.
\]
Choose $r>0$ so that the closed ball of radius $r$ about each mark lies in $\Omega$, and put $M=\|U\|_\infty$. A plane of nonzero slope $m$ active at a mark $p$ supports $f_k$ throughout the domain. Evaluating it in the direction $-m$ gives
\[
 0\le f_k(p-rm/\|m\|)\le f_k(p)-r\|m\|,
 \qquad \|m\|\le M/r.
\]
Consequently every update changes a coefficient in the fixed finite set $S=\{m\in\Z^2:\|m\|\le M/r\}$. The presentation stays canonical: every unchanged plane still touches, and the raised plane touches at the mark. Each coefficient increases and is bounded above by $\kappa_m(U)$. Let $a_{\infty,m}$ be its limit. Outside $S$ the coefficients remain fixed. The infimum of the limiting planes defines a function $\widetilde f$, and
\[
 \|\widetilde f-f_k\|_\infty
 \le\max_{m\in S}\bigl(a_{\infty,m}-\kappa_m(f_k)\bigr)
 \longrightarrow0.
\]
The same slope estimate on an interior neighborhood gives a common finite local presentation, so $\widetilde f$ remains tropical. Just after each visit to a mark, two distinct slopes are active there. Only finitely many pairs are possible; one pair recurs along a subsequence and remains active in the limit. Thus every mark is a corner of $\widetilde f$. The bound $0\le\widetilde f\le U$ preserves continuity and the zero boundary values. Hence $\widetilde f$ is admissible for $G_Pf_0$, and minimality gives $\widetilde f=U$. This proves \href{https://arxiv.org/pdf/1706.03062v1\#page=8}{Kalinin--Shkolnikov, Proposition 6.1\readerlocator{ (PDF p.~8)}}~\cite{KS18} in the present notation.

For a lattice polygon, lattice marks, and the zero initial roof, the canonical coefficients start as integers and every update is integral. Only finitely many coefficients can change, and they are bounded above. Every effective update raises one of them by at least one, so the process terminates; see \href{https://arxiv.org/pdf/1706.03062v1\#page=8}{Kalinin--Shkolnikov, Remark 6.2\readerlocator{ (PDF p.~8)}}~\cite{KS18}. For real configurations the preceding argument gives uniform convergence.

\needspace{12\baselineskip}\subsection*{\ifjournaledition\else C2. Why the parameter dependence is piecewise affine\fi}\pdfbookmark[1]{\ifjournaledition\else C2. Why the parameter dependence is piecewise affine\fi}{appc2}
Return now to the fixed polygon $\Delta$ and let the marks move in $O$. The common slope set of Lemma~\ref{lem:slope-support} allows us to use the same finite list of coefficients for all these roofs. With the slopes fixed, the corner conditions are affine equations and inequalities in the source coordinates and coefficients. Describing a touching plane also introduces a contact point. We need to eliminate such auxiliary variables while staying within finite unions of polyhedra. 

\needspace{5\baselineskip}\begin{proposition}[Semilinear sets and dimension]\label{prop:semilinear}

A semilinear set is a finite union of relatively open polyhedra; its dimension is the largest affine dimension of those pieces, with $\dim\varnothing=-\infty$. Such sets are closed under finite Boolean operations and linear projection. For a linear map $\pi$, $\dim\pi(X)\le\dim X$. If $X\subset\R^a\times\R^b$ and all nonempty fibers of the coordinate projection have dimension at most $d$, then $\dim X\le\dim\pi(X)+d$.

A single-valued map with semilinear graph is affine on each cell of a finite semilinear partition of its domain. For a finite family of polyhedra with fixed coefficient matrices and affine-parameter right-hand sides, the parameter loci specifying nonemptiness, dimension, closure intersection, or face containment are semilinear.
\end{proposition}
\begin{proof}
Boolean combinations of finite affine sign conditions remain finite affine sign conditions. To project, eliminate one variable $t$. After separating equalities, each remaining inequality is a lower bound, an upper bound, or independent of $t$. Every lower bound must be at most every upper bound; the comparison is strict if either bound is strict. These finitely many comparisons are also sufficient. If bounds occur on only one side, $t$ can be chosen far enough in the other direction. Equalities with a nonzero $t$ coefficient permit substitution. Iteration proves projection closure. This is Fourier--Motzkin elimination; see Schrijver~\cite[\S12.2]{Sch86}.

For a relatively open cell $Q$, let $V_Q$ be its direction space. Rank--nullity gives
\[
 \dim Q=\dim\pi(Q)+\dim(V_Q\cap\ker\pi).
\]
Every nonempty fiber of $Q$ is relatively open in the corresponding affine fiber, so its dimension is the last term. Taking maxima proves both dimension bounds. If $Q$ belongs to a single-valued graph, a nonzero vertical direction would give two nearby values over one point. The projection therefore has an affine inverse on $Q$. Refine the finitely many projected cells to obtain a piecewise-affine map.

For the last assertion, nonemptiness is projection of an affine system. Split weak inequalities according to which are equalities; the ranks of the resulting fixed matrices determine dimensions. On each nonempty strict piece, closure is obtained by weakening its strict inequalities. Thus closure intersection is again an existential affine condition. Failure of containment is the existence of a point in one set and outside the other. Projection and complementation finish the proof, including the empty-fiber cases.
\end{proof}

Canonical coefficients respect order: if $f\le g$, then $\kappa_m(f)\le\kappa_m(g)$ for every allowed slope. The least roof therefore gives the least admissible canonical coefficient vector. We can describe this vector by affine conditions and projection, obtaining its piecewise-affine dependence on the sources.

\needspace{5\baselineskip}\begin{proposition}[The semilinear relaxation map]\label{prop:coefficient-model}

With the common slope set fixed, admissible pairs of source coordinates and canonical coefficients form a semilinear set. The least roof selects the coordinatewise least admissible canonical vector. Its graph is semilinear, and the coefficient map is affine on each member of a finite polyhedral partition.
\end{proposition}
\begin{proof}
First allow arbitrary coefficient vectors. The admissible set $\mathcal A$ is described by the following finite conditions:
\begin{align*}
 &c_m+m\cdot v\ge0 &&(m\in {A_N},\ v\text{ a vertex of }\Delta);\\
 &\text{for each side }[v,w],\ \exists m\in {A_N}:
       c_m+m\cdot v=c_m+m\cdot w=0;\\
 &\text{for each }i,\ \exists m\ne n:\quad
       c_m+m\cdot p_i=c_n+n\cdot p_i\le c_r+r\cdot p_i
       &&(r\in {A_N}).
\end{align*}
The first line makes every plane nonnegative on $\Delta$. Given this, a finite minimum vanishes on a whole side precisely when one of its affine terms vanishes there identically: finitely many isolated zeros cannot cover an interval. The last line puts a corner at each mark. All slopes are fixed, so these are affine conditions in $(P,c)$.

To require canonical coefficients, add, for every $m$, the condition
\[
 \exists x\in\Delta:\quad c_m+m\cdot x\le c_r+r\cdot x\quad(r\in {A_N}).
\]
It says that each plane touches $f_c$, and is equivalent to $c_m=\kappa_m(f_c)$. Projection therefore makes the admissible set with canonical coefficients semilinear as well. From now on, let $\mathcal A$ denote this set of admissible pairs with canonical coefficients.

Write $c(P)=\kappa_{A_N}(F_P)$. Minimality gives, for every $(P,d)\in\mathcal A$,
\[
 c_m(P)=\max_\Delta(F_P-m\cdot x)
 \le\max_\Delta(f_d-m\cdot x)\le d_m.
\]
Also $(P,c(P))\in\mathcal A$. Consequently its graph is exactly
\[
 \mathcal A\setminus
 \{(P,c):\exists d,\ (P,d)\in\mathcal A,
                         \ d_m<c_m\text{ for some }m\}.
\]
This set is semilinear. It is a graph because the least vector is unique. Proposition~\ref{prop:semilinear} gives the asserted affine partition.
\end{proof}

\needspace{10\baselineskip}\subsection*{C3. Recovering the embedded type}\pdfbookmark[1]{C3. Recovering the embedded type}{appc3}
The coefficients must also determine the vertices, edges, and boundary contacts of the curve. As they vary affinely, edges may contract and boundary contacts may coincide. We therefore record where each set of planes is exactly the set attaining the minimum. These regions and their closure incidences recover the embedded curve.

Fix a parameter cell $Q$ on which $c(P)$ is affine. Let $\sigma$ range over the polygon, its closed sides, and its vertices, with fixed affine equations and strict inequalities for $\operatorname{relint}\sigma$. For $\varnothing\ne B\subset {A_N}$, choose $b\in B$ and put
\begin{equation}\label{eq:exact-active}
 Z_{B,\sigma}(P)=\left\{x\in\operatorname{relint}\sigma:
 \begin{array}{ll}
 (m-b)\cdot x=c_b(P)-c_m(P)&(m\in B),\\
 (b-r)\cdot x<c_r(P)-c_b(P)&(r\notin B)
 \end{array}\right\}.
\end{equation}
The strict inequalities assert that $B$ is the \emph{entire} active set.

\needspace{5\baselineskip}\begin{proposition}[Recovering the full embedded type]\label{prop:active-faces}

The equations and inequalities defining each exact active face have fixed matrices and right-hand sides affine in the parameters. These faces and their incidences determine the gradient-labelled embedded tropical type, including transverse nodes, the cyclic orders at vertices, and the order of boundary contacts. At a fixed parameter, these active sets reconstruct the corresponding curve.
\end{proposition}
\begin{proof}
Including the equations of $\sigma$, write each fiber as
\[
 Z_\alpha(P)=\{x:E_\alpha x=q_\alpha(P),\ G_\alpha x<h_\alpha(P)\},
 \qquad \alpha=(B,\sigma),
\]
with fixed matrices and affine right-hand sides. If $L_\alpha$ spans the left kernel of $E_\alpha$, consistency is $L_\alpha q_\alpha(P)=0$. A nonempty fiber is convex and relatively open in its equality space, of dimension $2-\operatorname{rank}E_\alpha$. Its closure $\widehat Z_\alpha$ is obtained by replacing $<$ with $\le$: join a weak solution to one strict solution and approach along the segment. Nonemptiness is essential here; $x>0$, $x<0$ has empty closure although the weakened system contains $0$.

When $E_\alpha$ has rank two, two independent rows give a unique affine candidate $v_\alpha(P)$. The remaining equations check consistency and the strict inequalities check membership. Interior vertices use active equalities; side endpoints also use the side equation; polygon vertices have fixed coordinates.

Refine $Q$ according to the following conditions:
\begin{enumerate}[leftmargin=18pt,itemsep=2pt,topsep=3pt]
\item Consistency, nonemptiness, and dimension of every exact fiber.
\item For nonempty fibers, each incidence $Z_\beta(P)\subset\widehat Z_\alpha(P)$. Its failure is $\exists x\in Z_\beta(P)\setminus\widehat Z_\alpha(P)$.
\item Candidate membership and coincidence; for candidates on side $S$, the sign of $t_S(v_\alpha(P))-t_S(v_\beta(P))$, where $t_S$ is a fixed affine coordinate along $S$.
\item The exact face containing each $p_i$, obtained by substituting $x=p_i$ in \eqref{eq:exact-active} with $\sigma=\Delta$.
\end{enumerate}
Each condition defines a semilinear set by Proposition~\ref{prop:semilinear}. The nonempty fibers partition $\Delta$, and their closure incidences give the labelled face poset. Two-dimensional cells carry their single active gradient. An interior edge has direction perpendicular to the difference of its adjacent gradients and weight equal to the lattice length of that difference. Its endpoint labels distinguish it from parallel edges. At a vertex, the incident faces and their fixed gradients determine the outgoing germs and their cyclic order. At a transverse node, all four incident germs occur in this description. Affine vertex coordinates and ordered boundary contacts recover the edge segments. 
\end{proof}

\needspace{5\baselineskip}\begin{proposition}[Stable-type strata]\label{prop:type-strata}

The parameter space admits a finite semilinear stratification on which the marked embedded type is constant. The top-dimensional strata provide local stability under independent source perturbation.
\end{proposition}
\begin{proof}
Apply the preceding refinement on each affine coefficient cell, then refine all affine sign conditions into one finite polyhedral stratification. On a stratum, the embedded incidence data, edge directions, weights, boundary orders, and marked-face labels are fixed; corresponding vertices vary affinely. A stratum of dimension $2N$ is open in configuration space. Every sufficiently small independent perturbation stays in it.
\end{proof}

\needspace{10\baselineskip}\subsection*{C4. The exceptional incidences}\pdfbookmark[1]{C4. The exceptional incidences}{appc4}
The stratification gives stable embedded types on open cells. Strong genericity also requires the marks to lie in distinct edge interiors. A mark at a vertex removes one parameter from the incidence space; two marks on the same edge impose a linear relation on their coordinates.

\needspace{5\baselineskip}\begin{proposition}[The exceptional incidences have positive codimension]\label{prop:discriminants}

Configurations with a source at a vertex, or with two sources on the same carrier edge, lie in semilinear exceptional sets of dimension at most $2N-1$.
\end{proposition}
\begin{proof}
There are finitely many types on $\operatorname{Conf}_N(O)$. For a type actually realized by relaxation, the binary spanning lemma gives $g\le N$. If some mark is a vertex, at most $N-1$ marks can move along edges. Proposition~\ref{prop:dimension} therefore bounds its marked incidence space by $g+N-1\le2N-1$. Projection to source coordinates preserves this upper bound. The finite union over realized types gives the first exceptional set.

For two marks on one edge, let $w\ne0$ be its primitive direction. Then
\begin{equation}\label{eq:same-edge}
 \det(p_i-p_j,w)=0.
\end{equation}
For fixed $i\ne j$ and $w$ this is a proper linear hyperplane in $\R^{2N}$. Only finitely many directions occur. The union of these hyperplanes contains the second exceptional set.
\end{proof}

Adapt the strata to these exceptional sets. The union of the $2N$-dimensional open cells outside them is open, dense, semilinear, and full measure in $\operatorname{Conf}_N(O)$: the discarded cells lie in finitely many affine subspaces of dimension at most $2N-1$. The strongly generic locus of Definition~\ref{def:strong-genericity} also includes points on partition walls whenever the marked embedded type remains locally stable there.

The active-face description also records four-valent nodes. For example, the function $q(x,y)=\min(0,x,y,x+y)=\min(0,x)+\min(0,y)$ has a four-valent node at the origin, dual to a unit square. Its four active gradients occur together in the exact active set of \eqref{eq:exact-active}. This is a local example of a tropical node; it does not establish that a least roof can have a node stable under independent motions of its sources.

\needspace{10\baselineskip}\subsection*{C5. Universal genericity}\pdfbookmark[1]{C5. Universal genericity}{appc5}
A cloud will be used on every polygon of an exhaustion. We therefore need genericity on all these polygons at once. Polygons with rational vertex coordinates form a countable family large enough to exhaust the domain, so the union of their exceptional null sets is still null.

\needspace{5\baselineskip}\begin{lemma}[A countable polygon family]\label{lem:cofinality}

Compact convex polygons with vertices in $\Q^2$, nonempty interior, and closure contained in $\Omega$ form a countable cofinal family for interior compact subsets. They admit nested exhaustions $\Delta_j\subset\Delta_{j+1}^\circ$ with union of interiors equal to $\Omega$.
\end{lemma}
\begin{proof}
There are countably many finite lists of rational points. For cofinality, replace a compact set $C\Subset\Omega$ by its compact convex hull $H\Subset\Omega$. Choose $r>0$ with $H+\overline B_{2r}\subset\Omega$. On a rational square grid of mesh $\eta$ with $\sqrt2\eta<r$, take the convex hull of the vertices of all squares meeting $H+\overline B_r$. This finite polygon contains $C$ in its interior and lies in $H+\overline B_{2r}$. Apply the construction successively to the union of a compact exhaustion and the preceding polygon. At each step the previous polygon lies strictly inside the new one.
\end{proof}

\needspace{5\baselineskip}\begin{definition}[Universal and exhaustion genericity]\label{def:universal}

Let $\mathcal P_{\Q}(\Omega)$ be the countable family in \ref{lem:cofinality}. View $G_N(\Delta)$ as a subset of $\operatorname{Conf}_N(\Omega)$ and set
\[
 \widehat G_N(\Delta)=G_N(\Delta)\ \cup\
 \{(p_1,\ldots,p_N)\in\operatorname{Conf}_N(\Omega):
                 p_i\notin\Delta\text{ for some }i\},
\]
\[
 G_N^{\rm univ}(\Omega)=\bigcap_{\Delta\in\mathcal P_\Q(\Omega)}\widehat G_N(\Delta).
\]
This set is permutation invariant, residual, and of full measure. In particular the definition excludes a configuration lying entirely in a closed test polygon with a mark on its boundary.

For $K\Subset O\Subset\Delta_1^\circ$ and a fixed nested exhaustion $\mathcal E=(\Delta_j)$ of $\Omega$ by compact convex polygons with rational side slopes, put
\[
 G_N^{\mathcal E}=\operatorname{Conf}_N(O)\cap\bigcap_{j\ge1}G_N(\Delta_j).
\]
This is residual and of full measure in $\operatorname{Conf}_N(O)$. Universal genericity implies this fixed-exhaustion condition for an exhaustion drawn from $\mathcal P_\Q(\Omega)$.
\end{definition}
\begin{proof}[Proof of the genericity assertions]
For fixed $\Delta$, both alternatives defining $\widehat G_N(\Delta)$ are open in $\operatorname{Conf}_N(\Omega)$; the second uses the complement of the \emph{closed} polygon. Its complement is the union of
\[
 \operatorname{Conf}_N(\Delta^\circ)\setminus G_N(\Delta)
 \quad\text{and}\quad
 \{P\in\operatorname{Conf}_N(\Omega):p_i\in\Delta\ \forall i,
                              \ p_j\in\partial\Delta\text{ for some }j\}.
\]
The first is null by Theorem~\ref{thm:generic-locus}; the second lies in finitely many boundary-incidence hyperplanes and is null. The complement is also closed, since $\widehat G_N(\Delta)$ is open. A closed null set has empty interior and is nowhere dense. Thus $\widehat G_N(\Delta)$ is open dense and full measure. Baire's theorem and countable subadditivity give the universal assertion. Relabelling marks changes none of the conditions.

For the fixed exhaustion, each $G_N(\Delta_j)\cap\operatorname{Conf}_N(O)$ is relatively open dense and full measure; the same countable argument applies. If the polygons have rational vertices and $P\subset O$, universal genericity forces the first alternative at every $\Delta_j$, hence membership in $G_N^{\mathcal E}$.
\end{proof}

The fixed-polygon locus is open. These countable intersections are dense $G_\delta$ sets of full measure and need not be open. The countable family is formed from polygons with rational vertex coordinates. A polygon cut out by a positive roof level need not have rational vertices. Definition~\ref{def:universal} therefore does not by itself give genericity on that polygon.

\subsubsection*{Generic approximations of a measure}
Every compactly supported probability measure, including a singular one, admits universally generic empirical approximations. We first approximate the measure by equal point masses and then perturb the points into the generic locus.

Let $\mu$ be supported in $K\Subset\Omega$, and choose a compact neighborhood $K'$ with $K\subset\operatorname{int}K'\subset K'\Subset\Omega$. By refining finite partitions of $K$ slowly enough that their number of cells is $o(N)$, and approximating their masses by multiples of $1/N$, choose empirical measures $N^{-1}\sum_{i=1}^N\delta_{x_{i,N}}$ converging to $\mu$, with $x_{i,N}\in K$; these provisional points may repeat.

The universal locus in Definition~\ref{def:universal} is dense in the space of distinct configurations, and distinct configurations are dense in $\Omega^N$. We may therefore choose a universally generic cloud $P_N=\{p_{1,N},\ldots,p_{N,N}\}\subset K'$ with $\max_i|p_{i,N}-x_{i,N}|\to0$. Uniform continuity of test functions on $K'$ preserves empirical convergence, and Theorem~\ref{thm:many-point} applies with $K'$ in place of $K$. The perturbations may leave the original support of $\mu$.

For an absolutely continuous sampling law, each finite ordered sample avoids the corresponding null exceptional set almost surely. Intersect over $N$, and separately over the countably many accepted sequences used in Appendix~\ref{sec:D}.
\endgroup

\needspace{12\baselineskip}\section{\texorpdfstring{\headingcolor{Probability}{Sampling and passage from sandpiles}}{Sampling and passage from sandpiles}}\label{sec:D}

The applications in Section~\ref{sec:6} use conditional sampling and a lattice approximation of each finite roof. Retaining samples in an interior compact set preserves independence and gives the conditional law; this permits the passage from compactly supported sampling to sampling throughout the domain. For a fixed finite cloud, rescaled sandpile odometers converge to its least roof as the mesh tends to zero. Grain conservation then identifies the limiting deficit measure. D4 compares the tropical and sandpile approximations with explicit roofs.

\ifjournaledition\else
For conditional sampling, we sum over the waiting times between successive visits to the retained set. Applying the result to a countable exhaustion gives one probability-one event on which all the retained sequences have the required genericity and convergence.

The fixed-cloud sandpile theorem requires proper rounding and freezes a row of boundary sites. Our convention allows every interior site to topple and places the sinks outside the domain. We compare the two odometers using a discrete superharmonic majorant; their difference after rescaling tends to zero. The fixed-cloud approximation then permits the choice of meshes used in the many-point limit.

To identify the deficit measure, we start from the lattice conservation equation relating the missing grains, the added sources, and the discrete Laplacian of the odometer. Summation by parts puts the Laplacian on a smooth test function. On the meshes chosen in Theorem~\ref{thm:sandpile}, the rescaled odometers converge to a roof $F$, and the deficit measures converge to $-\Delta F$. D3 gives the calculation.

\fi

\needspace{8\baselineskip}\subsection*{D1. The accepted sequence}\label{app:accepted-stream}
From a sequence of independent samples with a common law, retain those that fall in a measurable set $K$ of positive probability, keeping their original order. The retained sequence remains independent and has the law conditioned on $K$. The following lemma proves this assertion and counts the retained samples among the first $N$ draws.

\begin{lemma}[Successive accepted samples]\label{lem:accepted-stream}
Let $X_n$ be i.i.d. with probability law $\mu$, and let $K$ be measurable with $a=\mu(K)>0$. Set $\tau_0=0$ and $\tau_r=\min\{n>\tau_{r-1}:X_n\in K\}$. Almost surely every $\tau_r$ is finite. The variables $Y_r=X_{\tau_r}$ are i.i.d. with law $\mu(\,\cdot\cap K)/a$. Moreover,
\[
 M_N:=\sum_{n\le N}\mathbf1_K(X_n),\qquad M_N/N\longrightarrow a,
\]
and the samples accepted among $X_1,\ldots,X_N$ are precisely $Y_1,\ldots,Y_{M_N}$.
\end{lemma}
\begin{proof}
The strong law gives $M_N/N\to a>0$, hence infinitely many acceptances. If $a=1$, every sample is accepted almost surely and the result is immediate. Assume $0<a<1$. To identify the joint law, let $B_1,\ldots,B_m\subset K$ be measurable. Prescribe the gaps $\ell_r=\tau_r-\tau_{r-1}\ge1$. Independence gives
\[
 \Pr\{Y_r\in B_r,\ \tau_r-\tau_{r-1}=\ell_r\text{ for }1\le r\le m\}
 =\prod_{r=1}^m(1-a)^{\ell_r-1}\mu(B_r).
\]
Summing the nonnegative terms over all gaps yields
\[
 \Pr\{Y_1\in B_1,\ldots,Y_m\in B_m\}
 =\prod_{r=1}^m\left(\mu(B_r)\sum_{\ell\ge1}(1-a)^{\ell-1}\right)
 =\prod_{r=1}^m\frac{\mu(B_r)}a.
\]
These joint probabilities prove the assertion for all finite initial segments. The last assertion follows directly from the definition of the acceptance times.
\end{proof}

\textbf{One event for the exhaustion.} In Theorem~\ref{thm:iid-boundary}, apply this lemma separately to each $K_j$. Since $\mu_j$ is absolutely continuous and supported in the fixed compact $K_j$, the compact-support theorem gives an event $E_j$ of probability one on which every finite accepted cloud has the required genericity and its normalized roofs converge. This event is taken in the original sample space: the map taking the original sample to its accepted sequence has law $\mu_j^{\otimes\mathbb N}$ by the lemma. Intersect the $E_j$, the countably many strong-law events, and the event of pairwise distinct samples. The intersection still has probability one.

For samples with an absolutely continuous law $\mu$ on $\Omega$, take a compact $K\Subset\Omega$ with $a=\mu(K)>0$. The lemma gives $M_N/N\to a$ almost surely. Normalizing the retained roof by $N^{-1/2}$ instead of $M_N^{-1/2}$ therefore introduces a factor tending to $\sqrt a$. This is the normalization used in the proof of Theorem~\ref{thm:iid-boundary}.

\subsection*{D2. Fixed-cloud sandpile convergence}\label{app:diagonal}
Fix a nonempty finite cloud $P$ in the interior $\Omega$ of a compact convex polygon with rational side slopes. We prove that the rescaled odometer $hH_{h,P}$ converges to $F_{\Omega,P}$ as $h\to0$, with the source rounding specified below. This fixed-cloud result allows a sufficiently fine mesh to be chosen for each cloud in Theorem~\ref{thm:sandpile}.

Consider the sandpile on $\Omega_h=\Omega\cap h\Z^2$. Every site of $\Omega_h$ may topple. It sends one grain to each of its four lattice neighbors and loses any grain sent outside $\Omega_h$; those exterior sites are sinks. We extend odometers by zero there and use the unscaled lattice Laplacian
\begin{equation}\label{eq:lattice-laplacian}
 \Delta_h V(v)=\sum_{i=1}^2\bigl(V(v+he_i)+V(v-he_i)-2V(v)\bigr).
\end{equation}
For each source $p$, let $p^h\in\Omega_h$ be its chosen lattice representative. The initial state is $3+\xi_{h,P}$, where $\xi_{h,P}=\sum_{p\in P}\mathbf1_{\{p^h\}}$. The odometer $H_{h,P}$ counts topplings. The final height lies in $\{0,1,2,3\}$ at every site.

\textit{Proper rounding.}\label{ext:proper-rounding}
Put $F=F_{\Omega,P}$. A proper rounding chooses the representatives $p^h$ so that
\begin{equation}\label{eq:proper-rounding}
 \|p^h-p\|_\infty\le h,
 \qquad
 \Delta_h\lfloor h^{-1}F\rfloor(p^h)<0.
\end{equation}
Here $\|z\|_\infty=\max\{|z_1|,|z_2|\}$; the Euclidean displacement is at most $\sqrt2\,h$. The negative-Laplacian condition is that of \ks{Definition 4.1}{20}. It makes the rounded roof a superharmonic majorant for the added grain. The following lemma gives a suitable corner of the mesh square containing $p$. An arbitrary nearest lattice point need not satisfy the condition.

\needspace{5\baselineskip}\begin{lemma}[Proper rounding in a mesh square]\label{lem:proper-rounding}
Let $F$ be a finite minimum of affine functions with integral gradients, and let $p\in\Gamma(F)$. Suppose a closed mesh square containing $p$, together with the neighbor stencils of its four corners, lies in the domain of $F$. At least one of those corners $v$ satisfies
\[
 \Delta_h\lfloor F/h\rfloor(v)<0.
\]
In particular $v$ is a proper rounding of $p$.
\end{lemma}
\begin{proof}
Combine terms with the same gradient, keeping the lowest intercept. Write $F(x)=\min_m(a_m+m\cdot x)$ and rescale by $x=hz$. On the integer lattice,
\[
 Q(z)=\lfloor F(hz)/h\rfloor=\min_m L_m(z),
 \qquad L_m(z)=\lfloor a_m/h\rfloor+m\cdot z.
\]
Each $L_m$ is harmonic, so $Q$ is superharmonic. Suppose $Q$ were harmonic at all four corners of the rescaled square $C$. Choose $L_m$ active at one corner. Since $Q\le L_m$ at its neighbors and their Laplacians agree, equality holds at every neighbor. Repeating this argument at the other corners gives $Q=L_m$ at those corners and all their neighbors.

For a different gradient $n$, the integral-affine function $D=L_n-L_m$ is nonnegative at all these sites. Choose a corner $z_*$ minimizing $D$ on $C$, and choose $i$ with $|n_i-m_i|=\|n-m\|_\infty$. At the neighboring lattice point $z_*-\operatorname{sgn}(n_i-m_i)e_i$, nonnegativity gives
\[
 \min_C D=D(z_*)\ge\|n-m\|_\infty\ge1.
\]
Write $a_m/h=\lfloor a_m/h\rfloor+\theta_m$, where $0\le\theta_m<1$. The difference of the original rescaled planes at $p/h$ is
\[
 D(p/h)+\theta_n-\theta_m>0.
\]
Thus $m$ is the unique active gradient at $p$, contrary to $p\in\Gamma(F)$.
\end{proof}
By Proposition~\ref{prop:relaxation}, $F_{\Omega,P}$ has a finite presentation. For sufficiently small $h$, the mesh squares and stencils lie inside $\Omega$, and the rounded sources are distinct.

\textit{The fixed-source lower bound.}\label{ext:fixed-source}
Let
\[
 B_h=\{v\in\Omega_h:v+he_i\notin\Omega_h
       \text{ or }v-he_i\notin\Omega_h\text{ for some }i\},
\]
and let $H^{\rm row}_{h,P}$ be the odometer when the sites of $B_h$
are frozen, as in \ks{Section 2.2}{4}. We derive the lower estimate
$hH^{\rm row}_{h,P}>F-\varepsilon$ for the roundings
in~\eqref{eq:proper-rounding} from the following ingredients of the
proof of \ks{Proposition 5.4}{22}. A smooth vertex is dual to a
unimodular triangle; a node is dual to a unimodular parallelogram.

For a one-point update at $q\in P$, call $q$ its \emph{base}. If
$L_m$ is the unique active plane there, its \emph{affine gap} from
a competing plane $L_n$ is $L_n(q)-L_m(q)>0$. A full update raises
$L_m$ until the smallest gap vanishes. A \emph{shortened update}
stops before this first contact, so the base remains in the same
open affine face. The construction below uses such shortenings to
keep the rounded bases away from the corner locus.

\needspace{5\baselineskip}\begin{proposition}[Kalinin--Shkolnikov auxiliary construction]\label{prop:ks-wave-input}
Let $\Omega$ be the interior of a compact convex rational-slope polygon,
let $P\subset\Omega$ be finite and nonempty, and put $F=F_{\Omega,P}$.
Fix $\varepsilon>0$.
\begin{enumerate}[label=\textup{(\roman*)},leftmargin=2em]
\item There are a compact rational-slope polygon $\Pi'\subset\overline\Omega$
with $P\subset(\Pi')^\circ$, a $\Pi'$-tropical polynomial $g$ with
$0\le g<\varepsilon/4$ and smooth corner locus in $(\Pi')^\circ$, and finitely many
one-point updates based in $P$, whose final roof $u$ satisfies
$\|u-F\|_{L^\infty(\Pi')}<\varepsilon/4$.
The updates leave $g$ unchanged on a common neighborhood of
$\partial\Pi'$, and $F<\varepsilon/2$ on $\Omega\setminus\Pi'$.
\item Fix shortenings of these updates for which each base stays in its affine face
and all intermediate curves have only smooth or nodal vertices.
For sufficiently small $h$, rounding the coefficients to multiples of $h$
gives roofs $\widehat f$ with only smooth or nodal vertices and
integer-valued canonical smoothings $S_h$. On the active sites of the
frozen-row lattice in $\Pi'$ they satisfy
\[
 -3\le\Delta_hS_h\le0,
 \qquad \|hS_h-\widehat f\|_\infty\le Ch.
\]
Here $C$ is independent of $h$. There is also an $h$-independent
$R$ such that $S_h=\widehat f/h$ outside the $Rh$-neighborhood of
the rounded corner locus. If a lattice base and its neighbor stencil
lie in the same affine face beyond this neighborhood, increasing
that face's coefficient by $h$ is realized by one sandpile wave.
The wave changes the background $3+\Delta_hS_h$ to
$3+\Delta_hS'_h$ and has toppling function $S'_h-S_h$.
Along the chosen updates these differences vanish near $\partial\Pi'$.
\end{enumerate}
The constants and the mesh threshold may depend on the fixed
auxiliary construction and its shortenings.
\end{proposition}

The auxiliary roof in part~\textup{(i)} is constructed in the proof of
\ks{Lemma 9.2}{29}, using \ks{Propositions 7.10--7.11}{26}.
For part~\textup{(ii)}, \ks{Theorem 2}{27} gives stabilization of the
local smoothings, and \ks{Remark 8.5}{28} gives the bound
$\Delta_hS_h\ge-3$. Their global assembly and the bound on their width
are given in \ks{Definition 8.8 and Remarks 8.9--8.10}{28}.
The wave identity is \ks{Lemma 8.6}{28}, in the global form used in
the \ks{proof of Proposition 5.4}{30}.

Common smoothing constants can be chosen along the finite shortened
sequence. There are finitely many local slope configurations. Near a
nodal transition, the two adjacent unimodular triangles are treated
together by the four-plane model in \ks{equation (8.4)}{27}.
Thus the smoothing bounds remain uniform as the two vertices approach.

\textit{The rounding adaptation.}\label{ext:rounding-lower-bound}
Fix $P$ and $\varepsilon>0$, and choose the auxiliary roof and updates
from Proposition~\ref{prop:ks-wave-input}\textup{(i)}.
Omit ineffective updates and denote the remaining bases by
$q_1,\ldots,q_r\in P$.
If $r=0$, the desired lower bound follows from $H^{\rm row}_{h,P}\ge0$
and the preceding approximation. Otherwise let $A$ be the finite
slope set. For coefficients $a=(a_m)_{m\in A}$ the full update at $q$
has components
\begin{equation}\label{eq:one-point-coefficient-update}
 (T_q a)_m=
 \max\left\{a_m,\ \min_{n\in A\setminus\{m\}}
                  \bigl(a_n+(n-m)\cdot q\bigr)\right\}.
\end{equation}
When the active gradient is unique, only its coefficient changes.
The map $T_q$ is $1$-Lipschitz for the coefficient supremum norm.
All effective steps of the chosen sequence start with a positive
affine gap. Let $\delta_*>0$ be the least of these gaps. Choose
$\eta>0$ with $2r\eta<\delta_*/2$ and
$r\eta<\varepsilon/4$. Starting from $g$, at step $k$ compute the full
update $G_{q_k}$ of the current roof and reduce its coefficient increase
by $\eta$.
The active gradient at the beginning of every step is unchanged:
by induction the coefficient error after $k$ steps is at most $k\eta$.
Each shortened step has positive length, after decreasing $\eta$
if necessary. Write $f_0=g,f_1,\ldots,f_r$ for the resulting roofs.
Then
\[
 \|f_r-F\|_{L^\infty(\Pi')}<\varepsilon/2.
\]
Throughout each shortened step, its base remains in the same affine
face with every competing affine gap at least $\eta$.
By \ks{Proposition 7.12}{26}, the curves stay smooth or nodal:
every face of such a curve is unimodular, and the update stops
before the face containing its base disappears.

The roofs and $\eta$ are now fixed. Round the initial coefficients
down to multiples of $h$, and replace each of the $r$ positive
coefficient increases $\alpha_k$ by
$h\lfloor\alpha_k/h\rfloor$. The coefficient error is at most
$E_h=(r+1)h$, also during a step. Choose a common smoothing width
$R$ in the $\ell^\infty$ norm, enlarged to include the neighbor
stencil, and put $B=\max_{m,n\in A}\|m-n\|_1$.
If $m$ is active at the base $q$ and $\widehat L_n$ are the
rounded affine planes at any time in that step, then
$\|q^h-q\|_\infty\le h$ and $\|z-q^h\|_\infty\le Rh$ give
\begin{equation}\label{eq:rounding-buffer}
 \widehat L_n(z)-\widehat L_m(z)
 \ge \eta-2E_h-B(R+1)h\qquad(n\ne m).
\end{equation}
Small coefficient perturbations preserve these local types: a
smooth dual triangle is unimodular, and a nodal unimodular
parallelogram can split into two unimodular triangles. Take $h$
small enough that the preceding gap is positive, the rounded sources
are distinct, and the canonical smoothings in the construction are defined. Thus every permitted $q^h$ lies, with
its stencil, beyond the smoothing neighborhood of the corner locus.
The smoothed background has height three there. Thus \ks{Lemma 8.6}{28} applies at every rounded base satisfying the stated conditions.

Put $\Pi'_h=(\Pi')^\circ\cap h\Z^2$, and let $S_{h,k}$ be the
smoothed potentials supplied by Proposition~\ref{prop:ks-wave-input}\textup{(ii)}
after the $k$th group of waves, with rounded roofs $\widehat f_k$.
On active sites, the initial auxiliary state
\[
 3+\Delta_hS_{h,0}+\xi_{h,P}
 \quad\text{lies between }0\text{ and }3+\xi_{h,P}.
\]
Outside $\Pi'$ extend this state by the maximal background with
its added grains; keep its inner boundary row frozen during the
partial stabilization. At each chosen base its added grain makes the first toppling legal.
The wave returns the smoothed background to height three at that
base, so the grain permits the next wave. Other source grains can
only permit further topplings. Hence the prescribed groups of
waves form a legal partial stabilization; their total odometer
is $S_{h,r}-S_{h,0}$. No toppling occurs near $\partial\Pi'$, so the
same sequence is legal in the frozen-row model on $\Omega_h$.
Monotonicity with respect to the initial state now yields
\[
 \begin{split}
 hH^{\rm row}_{h,P}
 &\ge h(S_{h,r}-S_{h,0})\\
 &\ge f_r-g-2E_h-2Ch\\
 &>F-3\varepsilon/4-2E_h-2Ch
 \qquad\text{on }\Pi'_h.
 \end{split}
\]
Choose $h$ smaller once more so that $2E_h+2Ch<\varepsilon/4$.
Outside $\Pi'$, the inequality follows from $F<\varepsilon/2$ and
the nonnegativity of the odometer. This proves the required lower
bound with~\eqref{eq:proper-rounding}.

The row upper bound follows directly from least action.
The function $Q_h=\lfloor F/h\rfloor$ is nonnegative and
superharmonic, and has Laplacian at most $-1$ at each rounded
source. Set its values on $B_h$ to zero; this can only decrease
its Laplacian at the sites which may topple. For small $h$ the
source stencils avoid $B_h$, so it remains a majorant for the
initial state. Consequently
\begin{equation}\label{eq:ks-row-bounds}
 F(v)-\varepsilon<hH^{\rm row}_{h,P}(v)\le F(v)
 \qquad(v\in\Omega_h).
\end{equation}
With $P,\varepsilon$, the auxiliary construction, and $\eta$ fixed, the estimates hold for all sufficiently small $h$. The mesh threshold may depend on $P$.

\textit{Moving the sinks outside.}\label{app:sink-comparison}
To move the sinks outside, extend the finite presentation of $F$ from Proposition~\ref{prop:relaxation} to the whole plane. Call this minimum $\widetilde F$, choose a Lipschitz constant $L_P$, and put $C_P=\lceil L_P\rceil$. On the full lattice set
\[
 Q_h(v)=\lfloor h^{-1}\widetilde F(v)\rfloor.
\]
If the affine pieces of $\widetilde F$ are $a_m+m\cdot x$, then
\begin{equation}\label{eq:rounded-affine-minimum}
 Q_h(v)=\min_m\bigl(\lfloor a_m/h\rfloor+m\cdot(v/h)\bigr).
\end{equation}
The integral slopes make each term discrete harmonic. A term active at $v$ majorizes the minimum at its four neighbors, so $\Delta_hQ_h(v)\le0$.

For an exterior neighbor $w$ of $\Omega_h$, the segment to its interior neighbor crosses $\partial\Omega$ within distance $h$. There $\widetilde F=0$, so $\widetilde F(w)\ge-L_Ph$ and $Q_h(w)+C_P\ge0$. Define
\[
 V_h(v)=Q_h(v)+C_P\quad(v\in\Omega_h),
 \qquad V_h(v)=0\quad(v\notin\Omega_h).
\]
Replacing the exterior values of $Q_h+C_P$ by zero only decreases its Laplacian at interior sites. Hence $\Delta_hV_h\le0$ there. At a rounded source the entire stencil lies inside, so \eqref{eq:proper-rounding} gives $\Delta_hV_h\le-1$. Since the sources are distinct,
$3+\xi_{h,P}+\Delta_hV_h\le3$ at every site. Also $V_h\ge0$ and is integer valued.

Least action gives $H_{h,P}\le V_h$. Indeed, just before a legal toppling count first exceeds $V_h$ at $v$, it satisfies $u\le V_h$ and $u(v)=V_h(v)$. Then $(3+\xi_{h,P}+\Delta_hu)(v)\le(3+\xi_{h,P}+\Delta_hV_h)(v)\le3$, so the proposed toppling is illegal.

Conversely, the stabilization which leaves $B_h$ frozen is a legal partial stabilization for our model. It follows that $H^{\rm row}_{h,P}\le H_{h,P}$. Together with \eqref{eq:ks-row-bounds}, these observations give
\begin{equation}\label{eq:exterior-sink-bounds}
 F-\varepsilon<hH^{\rm row}_{h,P}
 \le hH_{h,P}\le hV_h\le F+hC_P
 \quad\text{on }\Omega_h.
\end{equation}
Thus $\sup_{\Omega_h}|hH_{h,P}-F|\to0$ for the exterior-sink convention as well. For each fixed cloud $P$, the error $hC_P$ tends to zero.

\textit{Choosing the mesh for each cloud.} The proof of Theorem~\ref{thm:sandpile} chooses the meshes from this fixed-cloud estimate.

\subsection*{D3. The deficit calculation}\label{app:deficit}
Grain conservation determines the deficit measure from the odometer. Use the clouds $P_N$ and meshes $h_N$ chosen in Theorem~\ref{thm:sandpile}, and let $\mu$ be the limiting empirical measure of the clouds. Write $H_N=H_{h_N,P_N}$ for the odometer, $\xi_N$ for the number of added grains at each site, and $\varphi_N^\circ$ for the final height. The deficit relative to height three is $d_N=3-\varphi_N^\circ$. We prove that its normalized measure
\[
 D_N=\frac{h_N}{\sqrt N}\sum_v d_N(v)\delta_v
\]
converges vaguely in $\Omega$ to $-\Delta F_{\mu,\Omega}$. The normalized odometers $U_N=h_NH_N/\sqrt N$ already converge uniformly on the lattice sites to $F_{\mu,\Omega}$.

Each toppling subtracts four grains from its site and adds one at each neighbor. With exterior sinks and zero-extended odometer, grain conservation at each interior site gives
\begin{equation}\label{eq:odometer-deficit-identity}
 \varphi_N^\circ=3+\xi_N+\Delta_{h_N}H_N,
 \qquad d_N=3-\varphi_N^\circ=-\xi_N-\Delta_{h_N}H_N.
\end{equation}
Here $0\le d_N\le3$ and $\sum_v\xi_N(v)=N$.

For $\psi\in C_c^\infty(\Omega)$, extend $\psi$ by zero to the plane. For small enough $h_N$, its support and the neighboring stencils lie inside $\Omega$. Reindexing the finite lattice sums gives
\begin{equation}\label{eq:discrete-sbp}
 \sum_v\psi(v)\Delta_hU_N(v)
 =\sum_vU_N(v)\Delta_h\psi(v).
\end{equation}
There is no boundary term because these are full-lattice sums of zero-extended functions. Multiplying \eqref{eq:odometer-deficit-identity} by $h_N\psi/\sqrt N$ and summing proves \eqref{eq:deficit-test}. Since $Nh_N^2\to0$, the source contribution in that identity has absolute value at most
\[
 h_N\sqrt N\|\psi\|_\infty\longrightarrow0.
\]

The Taylor expansion of the unscaled operator is
\begin{equation}\label{eq:discrete-taylor}
 \Delta_h\psi(v)=h^2\Delta\psi(v)+R_h(v),
 \qquad |R_h(v)|\le C_\psi h^4.
\end{equation}
Odd powers cancel between opposite neighbors. Only $O_\psi(h^{-2})$ sites in a fixed compact neighborhood can contribute. Theorem~\ref{thm:sandpile} gives $\sup|U_N|\le M$ for all sufficiently large $N$. Summing the remainders gives
\[
 \left|\sum_vU_N(v)R_{h_N}(v)\right|
 \le M\,O_\psi(h_N^{-2})\,O_\psi(h_N^4)
 =O_\psi(h_N^2)\longrightarrow0.
\]
For the main term, uniform convergence of $U_N$ to $F=F_{\mu,\Omega}$ on the lattice sites gives
\[
 \left|h_N^2\sum_v(U_N(v)-F(v))\Delta\psi(v)\right|
 \le C_\psi\sup_{\Omega_{h_N}}|U_N-F|\longrightarrow0.
\]
The remaining Riemann sum tends to $\int F\Delta\psi\,dx$. This proves convergence of $D_N$ to $\nu=-\Delta F$ on smooth tests.

\textit{Vague convergence.} Interior mollifications of a concave function are concave, so their negative Laplacians are nonnegative. Passing to distributions shows that $-\Delta F$ is a positive distribution and hence a locally finite Radon measure. Given $L\Subset\Omega$, take $\chi\in C_c^\infty(\Omega)$, nonnegative and equal to one near $L$. Since $D_N\ge0$,
\[
 D_N(L)\le\int\chi\,dD_N\longrightarrow\int\chi\,d\nu<\infty.
\]
For a continuous compactly supported test, approximate uniformly by smooth tests supported in one larger compact set. These mass bounds make the approximation errors uniform in $N$, giving vague convergence. For a nonempty relatively compact Lipschitz window $W$ with $\nu(\partial W)=0$, vague convergence gives $D_N(W)\to\nu(W)$. The average-deficit limit follows from \eqref{eq:window-deficit-normalization}. 

\subsection*{D4. Two numerical illustrations}\label{app:diagnostics}

We compare normalized tropical roofs with the explicit quadratic solution for uniform sampling in the disk. A second example compares sandpile odometers with the explicit two-point roof on the square. The first comparison concerns the many-point approximation; the second concerns lattice refinement for a fixed cloud.

\textit{The disk profile.} For uniform probability on the unit disk $D$, Corollary~\ref{cor:disk-ellipse} gives
\[
 F(x)=\frac{1-|x|^2}{2\sqrt\pi}.
\]
We compare this solution with tropical roofs on the circumscribed rational polygons
\[
 \Delta_R=\frac1R\bigl(\operatorname{conv}(\Z^2\cap R\overline D)+[-1,1]^2\bigr).
\]
They satisfy $\overline D\subset\Delta_R\subset(1+\sqrt2/R)\overline D$. The main comparison uses $R=256$, for which $\Delta_R$ has $152$ sides. Three independent sequences were sampled uniformly in $D$, with coordinates rounded to $2^{-24}\Z^2$. Within each sequence, $P_N$ consists of the first $N$ points.

The roofs were computed by the one-point coefficient updates of Section~\ref{sec:2}, starting from zero and using exact rational arithmetic. We cycled through the sources until a full pass left every coefficient unchanged. All integral slopes in $[-32,32]^2$ were retained. To justify this restriction, let $G$ be the final computed roof. It is nonnegative, vanishes on $\partial\Delta_R$, and has a corner at every source. It is therefore admissible for the least roof $F_{\Delta_R^\circ,P_N}$. Minimality gives $F_{\Delta_R^\circ,P_N}\le G$, so the side quasidegree of the least roof on each side is at most that of $G$. The convex-hull argument of Lemma~\ref{lem:slope-support} then places every essential gradient of the least roof in $[-32,32]^2$: the boundary gradients of $G$ lie in this square, as does the origin. The least roof therefore belongs to the retained finite-slope class. The coefficient updates from zero stay below it, proving $G=F_{\Delta_R^\circ,P_N}$.

Put $u_{N,R}=N^{-1/2}F_{\Delta_R^\circ,P_N}$. On $K=\{|x|\le0.9\}$, measure the relative error by
\[
 E_{N,R}=\frac{\|u_{N,R}-F\|_{L^2(K)}}{\|F\|_{L^2(K)}}.
\]
The table gives grid approximations to the mean and range of $E_{N,256}$ over the three clouds, using a uniform $641\times641$ grid on $[-0.9,0.9]^2$, restricted to $K$. Figure~\ref{fig:numerical-profiles} shows central sections for the first sample sequence.

\begin{center}\small
\begin{tabular}{@{}rrr@{}}\toprule
$N$ & Mean relative $L^2$ error & Range over three clouds\\\midrule
25  & 0.115  & 0.104--0.133\\
100 & 0.0554 & 0.0487--0.0593\\
400 & 0.0338 & 0.0248--0.0435\\\bottomrule
\end{tabular}\normalsize\end{center}

For the same three $400$-point clouds, replacing $\Delta_{256}$ by $\Delta_{512}$ gives a mean relative error of $0.0309$. The mean of $\|u_{400,512}-u_{400,256}\|_{L^2(K)}/\|F\|_{L^2(K)}$ is $0.00727$. This comparison keeps the source points fixed and measures the change caused by refining the polygonal approximation to the disk.

\input{figures/numerical-profiles.tex}

\textit{An exact lattice comparison.} On $\Omega=(0,1)^2$, take the two points from Section~\ref{sec:2},
\[
 p=(1/4,12/25),\qquad q=(13/25,27/50).
\]
Their least roof is
\[
 F_{\Omega,\{p,q\}}(x,y)
 =\min\{23/50,\ x+21/100,\ 2x,\ 1-x,\ y,\ 1-y\}.
\]
Start with height three at every site of $\Omega\cap h\Z^2$, add one grain at each of $p,q$, and stabilize with exterior sinks. On each of the three meshes
\[
 h=1/100,\quad 1/200,\quad 1/400,
\]
the computed odometer satisfies
\[
 hH_{h,\{p,q\}}(z)=F_{\Omega,\{p,q\}}(z)
 \qquad(z\in\Omega\cap h\Z^2)
\]
exactly at every lattice site. Stabilization used legal integer topplings until every height lay in $\{0,1,2,3\}$. Grain conservation and the displayed equality were checked in integer arithmetic.
\clearpage
\section*{List of symbols}\label{symbols:list}
\pdfbookmark[1]{List of symbols}{symbols-list}
The entries follow the order of the main constructions.
\subsection*{\texorpdfstring{\textcolor{Algebra}{Domains, sources, and roofs}}{Domains, sources, and roofs}}\label{symbols:group1}

\begingroup\small
\setlength{\tabcolsep}{4pt}\renewcommand{\arraystretch}{1.1}
\begin{longtable}{@{}P{.205\textwidth}P{.565\textwidth}P{.17\textwidth}@{}}
\toprule Symbol & Meaning & Reference\\\midrule
\endfirsthead
\toprule Symbol & Meaning & Reference\\\midrule
\endhead
\bottomrule\endfoot
$\Omega,\ \overline\Omega,\ \partial\Omega$ & A bounded open convex planar domain, its closure, and its boundary; the domain has nonempty interior. & Conventions, p.~\pageref{sec:conventions}\\[5pt]
$\Delta,\ \Delta^\circ$ & A compact convex polygon with rational side slopes, and its interior. Rational vertex coordinates are an additional condition. & Conventions, p.~\pageref{sec:conventions}\\[5pt]
$K\Subset\Omega$ & A compact source set contained in the interior of $\Omega$, fixed in the quantitative estimates. & Thm.~\ref{thm:many-point}, p.~\pageref{thm:many-point}\\[5pt]
$P,\ P_N,\ N$ & A finite set of distinct interior sources; $|P_N|=N$. Ordered configurations are used when discussing parameter spaces. & Conventions, p.~\pageref{sec:conventions}\\[5pt]
$F_{\Omega,P},\ F_P$ & The least nonnegative concave tropical roof with integral slopes, zero boundary values, and a corner at every point of $P$. The domain is suppressed in $F_P$. & Def.~\ref{def:tropical}, p.~\pageref{def:tropical}\\[5pt]
$G_Pf$ & The least roof above $f$ with a corner at every point of $P$; $G_P0_\Omega=F_{\Omega,P}$. & Def.~\ref{def:tropical}, p.~\pageref{def:tropical}\\[5pt]
$G_pf$ & The least roof above $f$ with a corner at $p$; $G_pf=f$ when $p\in\Gamma(f)$. & One-point update, p.~\pageref{app:one-point-relaxation}\\[5pt]
$f_k,\ q_k,\ f_{N,k}$ & Iterated roofs and visited marks; every mark is visited infinitely often. In $f_{N,k}$, $N$ is the cloud size and $k$ counts updates. & Repeated visits, p.~\pageref{app:wave-convergence}\\[5pt]
$\Gamma(F),\ \Gamma_P$ & The corner locus of $F$; $\Gamma_P=\Gamma(F_{\Omega,P})$. & Def.~\ref{def:tropical}, p.~\pageref{def:tropical}\\[5pt]
$\nu_P,\ \mu_P$ & The counting measure $\nu_P=\sum_{p\in P}\delta_p$ and empirical probability $\mu_P=|P|^{-1}\nu_P$ for $P\ne\varnothing$. & Conventions, p.~\pageref{sec:conventions}\\[5pt]
$u_N$ & The normalized roof $u_N=N^{-1/2}F_{\Omega,P_N}$, where $|P_N|=N\ge1$. & Thm.~\ref{thm:many-point}, p.~\pageref{thm:many-point}\\[5pt]
$F_{\mu,\Omega}$ & The continuous concave Aleksandrov solution with curvature measure $\mu$ and zero boundary values. & Thm.~\ref{thm:many-point}, p.~\pageref{thm:many-point}\\[5pt]
$D_\Omega$ & The lattice support distance: $\inf_{m\in\Z^2\setminus\{0\}}(m\cdot x-\inf_{y\in\Omega}m\cdot y)$. Positive inside and zero on the boundary. & p.~\pageref{ext:tropical-distance}\\[5pt]
$\rho_\Omega(P),\ \delta_K$ & Tropical depths: $\rho_\Omega(P)=\min_PD_\Omega$ for a nonempty cloud, and $\delta_K=\min_KD_\Omega>0$. & Prop.~\ref{prop:barriers}, p.~\pageref{prop:barriers}\\[5pt]
$\lambda_S,\ n_S$ & The affine equation of a side $S$, positive inside the polygon, with primitive integral inward normal $n_S=\nabla\lambda_S$. & Lemma~\ref{lem:side}, p.~\pageref{lem:side}\\[5pt]
$B_\Delta,\ \ell_{{\rm cap},\Delta}$ & The side barrier $B_\Delta=\min_S\lambda_S$, with $D_{\Delta^\circ}\le B_\Delta$, and its positive integer multiplier in $E_\Delta$. & \eqref{eq:polygon-competitor}, p.~\pageref{eq:polygon-competitor}\\[5pt]
\end{longtable}\endgroup

\clearpage
\subsection*{\texorpdfstring{\textcolor{Geometry}{The curve, the marked tree, and curvature counts}}{The curve, the marked tree, and curvature counts}}\label{symbols:group2}
Weights count lattice multiplicity; crossings count edge germs. The counts $c_F$ and $\beta_F$ refer to the uncut graph; $c_{\rm cut}$ and $\beta_{\rm cut}$ refer to the graph cut at its marks.

\begingroup\small
\setlength{\tabcolsep}{4pt}\renewcommand{\arraystretch}{1.1}
\begin{longtable}{@{}P{.205\textwidth}P{.565\textwidth}P{.17\textwidth}@{}}
\toprule Symbol & Meaning & Reference\\\midrule
\endfirsthead
\toprule Symbol & Meaning & Reference\\\midrule
\endhead
\bottomrule\endfoot
$L_R$ & The essential plane over an open cell $R$; its contact set with the roof is exactly $\overline R$. & Prop.~\ref{prop:essential}, p.~\pageref{prop:essential}\\[5pt]
$g,\ R_1,\ldots,R_g$ & The number of bounded cells and the cells themselves. Strong genericity gives $g=N$. & Lemma~\ref{lem:binary}, p.~\pageref{lem:binary}\\[5pt]
$a_i\in\mathbf F_2^g$ & The vector whose $j$th entry is one when the plane over $R_j$ is active at $p_i$, and zero otherwise. & Lemma~\ref{lem:binary}, p.~\pageref{lem:binary}\\[5pt]
$G_0,\ f_\infty$ & The finite inward truncation of the tropical curve, and the exterior face of its plane embedding. & p.~\pageref{app:primal-truncation}\\[5pt]
$e_i,\ e_i^\vee$ & The edge containing $p_i$, and its crossing edge in the planar dual. & p.~\pageref{app:primal-truncation}\\[5pt]
$S^\vee,\ T_0,\ T_P$ & The marked dual spanning tree, the primal tree after deleting marked carriers, and the completed tree after cutting at the marks. & Thm.~\ref{thm:marked-tree}, p.~\pageref{thm:marked-tree}\\[5pt]
$H_P$ & The subtree of $T_P$ joining its $2N$ marked ends; identifying marked pairs gives the compact core. & Cor.~\ref{cor:compact-core}, p.~\pageref{cor:compact-core}\\[5pt]
$Q_v,\ d(v)$ & The Newton polygon $Q_v=\partial^+F(v)$ and the edge-germ valence of a tropical vertex. & Def.~\ref{def:window}, p.~\pageref{def:window}\\[5pt]
$w(e),\ v_e$ & The positive integral edge weight and a primitive integral tangent. The adjacent gradient jump is $w(e)$ times a primitive normal. & Thm.~\ref{thm:complexity}, p.~\pageref{thm:complexity}\\[5pt]
Regular window & A relatively compact open test set whose boundary avoids vertices and marks and crosses edges transversely. Each boundary incidence has its own endpoint in the truncated graph. & Def.~\ref{def:window}, p.~\pageref{def:window}\\[5pt]
$U,\ B_F(U),\ B_P(U)$ & A regular window and its number of crossing edge germs; $B_P=B_{F_P}$. & Def.~\ref{def:window}, p.~\pageref{def:window}\\[5pt]
$c_F(U),\ \beta_F(U)$ & The number of connected components and first Betti number of the uncut graph truncated to $U$. & Def.~\ref{def:window}, p.~\pageref{def:window}\\[5pt]
$c_{\rm cut}(U)$ & The number of components after cutting at all marks in $U$; $\beta_{\rm cut}(U)$ is the first Betti number after cutting. & Thm.~\ref{thm:window-balance}, p.~\pageref{thm:window-balance}\\[5pt]
$\sigma_F(v),\ \Xi_F,\ \Xi_P$ & The surplus $\sigma_F(v)=\operatorname{Area}(Q_v)-(d(v)-2)/2$ and its measure $\Xi_F=\sum_v\sigma_F(v)\delta_v$; $\Xi_P=\Xi_{F_P}$. & Def.~\ref{def:window}, p.~\pageref{def:window}\\[5pt]
$I(Q_v)$ & The number of interior lattice points of the Newton polygon. & Prop.~\ref{prop:pick}, p.~\pageref{prop:pick}\\[5pt]
$B_F^w(R)$ & The weighted crossing count along a regular rectangle: $\sum_{x\in\Gamma(F)\cap\partial R}w(e_x)$. & Lemma~\ref{lem:crofton}, p.~\pageref{lem:crofton}\\[5pt]
$B_{\rm term},\ D_{\rm term}$ & The number of terminal edges and the sum of their weights, relative to the stated ambient polygon. & Prop.~\ref{prop:terminal}, p.~\pageref{prop:terminal}\\[5pt]
\end{longtable}\endgroup

\clearpage
\subsection*{\texorpdfstring{\textcolor{Analysis}{Lengths, measures, and limits}}{Lengths, measures, and limits}}\label{symbols:group3}
Curvature is measured by Euclidean area in slope space. Curve length is measured in the domain.

\begingroup\small
\setlength{\tabcolsep}{4pt}\renewcommand{\arraystretch}{1.1}
\begin{longtable}{@{}P{.205\textwidth}P{.565\textwidth}P{.17\textwidth}@{}}
\toprule Symbol & Meaning & Reference\\\midrule
\endfirsthead
\toprule Symbol & Meaning & Reference\\\midrule
\endhead
\bottomrule\endfoot
$m_F(S),\ D_\partial(F)$ & The side quasidegree in $F=m_F(S)\lambda_S$ near a side interior, and the total side quasidegree $D_\partial(F)$. & Lemma~\ref{lem:side}, p.~\pageref{lem:side}\\[5pt]
$A_{\rm trop}(S)$ & The side factor $\ell(S)\|v_S\|$, where $v_S$ is a primitive integral tangent. & Thm.~\ref{thm:complexity}, p.~\pageref{thm:complexity}\\[5pt]
$A_{\rm symp}(\Gamma)$ & The weighted length functional $\sum_e w(e)\ell(e)\|v_e\|$, equal to $\sum_Sm_F(S)A_{\rm trop}(S)$ on a polygon. & Thm.~\ref{thm:complexity}, p.~\pageref{thm:complexity}\\[5pt]
$\ell(e),\ \mathcal H^1$ & Euclidean edge length and one-dimensional Hausdorff measure. Ordinary curve length is $\mathcal H^1(\Gamma)$. & Thm.~\ref{thm:complexity}, p.~\pageref{thm:complexity}\\[5pt]
$\partial^+F(x)$ & The concave superdifferential: slopes $p$ satisfying $F(y)\le F(x)+p\cdot(y-x)$ throughout the domain. & Conventions, p.~\pageref{sec:conventions}\\[5pt]
$v=-F,\ \partial v(x)$ & The convex function associated with a roof and its convex subdifferential; $\partial v(x)=-\partial^+F(x)$. & p.~\pageref{app:sign}\\[5pt]
$M_v(E),\ \MA(F)(E)$ & The area $|\partial v(E)|=|\partial^+F(E)|$, using ordinary Euclidean area. A unimodular triangle has mass $1/2$. & \eqref{eq:measure-definition}, p.~\pageref{eq:measure-definition}\\[5pt]
$\nabla_\tau\varphi$ & The derivative of $\varphi$ in a unit tangent direction along a tropical edge. & Thm.~\ref{thm:coarea}, p.~\pageref{thm:coarea}\\[5pt]
$\Delta_t,\ \Delta_j$ & The positive-level polygon $\Delta_t=\{F\ge t\}$ of one roof, and exhaustion polygons $\Delta_j$ chosen independently of the roof. Universal genericity applies to the fixed rational-coordinate polygons. & Thm.~\ref{thm:cores}, p.~\pageref{thm:cores}\newline Thm.~\ref{thm:fixed-exhaustion}, p.~\pageref{thm:fixed-exhaustion}\\[5pt]
$L,\ W,\ R,\ \Lambda$ & Compact test or observation regions, a regular interior rectangle $R$, and the numerical gradient bound $\Lambda$ in the local length and crossing estimate. & Lemma~\ref{lem:crofton}, p.~\pageref{lem:crofton}\\[5pt]
$\rightharpoonup$ & Vague convergence in $\Omega$ is tested against $C_c(\Omega)$. Weak convergence on $\overline\Omega$ is tested against $C(\overline\Omega)$, with interior measures extended by zero. & Conventions, p.~\pageref{sec:conventions}\\[5pt]
$C_{\Omega,K,L}$ & A constant uniform in $N$, the admissible cloud in $K$, and all sufficiently large polygons of the fixed exhaustion. & Thm.~\ref{thm:general-discrepancy}, p.~\pageref{thm:general-discrepancy}\\[5pt]
\end{longtable}\endgroup

\clearpage
\subsection*{\texorpdfstring{\textcolor{Algebra}{Hahn series and finite parameters}}{Hahn series and finite parameters}}\label{symbols:group4}

\begingroup\small
\setlength{\tabcolsep}{4pt}\renewcommand{\arraystretch}{1.1}
\begin{longtable}{@{}P{.205\textwidth}P{.565\textwidth}P{.17\textwidth}@{}}
\toprule Symbol & Meaning & Reference\\\midrule
\endfirsthead
\toprule Symbol & Meaning & Reference\\\midrule
\endhead
\bottomrule\endfoot
$\mathbb K=\C((t^\R))$ & The Hahn field of complex formal series with well-ordered real support. The symbol $t$ is formal. & p.~\pageref{app:hahn-field}\\[5pt]
$\operatorname{supp}(f)$ & The set of exponents with nonzero coefficient in a Hahn series. & p.~\pageref{app:hahn-support}\\[5pt]
$\operatorname{val}(f)$ & The least exponent of a nonzero Hahn series, with $\operatorname{val}(0)=+\infty$. & p.~\pageref{app:hahn-valuation}\\[5pt]
$[t^\gamma]f$ & The coefficient of the monomial $t^\gamma$ in the formal series $f$. & \eqref{eq:hahn-product}, p.~\pageref{eq:hahn-product}\\[5pt]
$m_j,\ a_{ij},\ \widetilde A$ & Distinct lattice slopes, the real entries $a_{ij}=\langle m_j,p_i\rangle$, and the Hahn matrix $\widetilde A=(t^{a_{ij}})$. & Lemma~\ref{lem:tropical-dependence}, p.~\pageref{lem:tropical-dependence}\newline Thm.~\ref{thm:interpolation}, p.~\pageref{thm:interpolation}\\[5pt]
$c_j,\ q$ & The interpolating roof $q(x)=\min_{j\in J}(c_j+\langle m_j,x\rangle)$. Here $(z_j)$ is a kernel vector of $\widetilde A$, $J=\{j:z_j\ne0\}$, and $c_j=\operatorname{val}(z_j)\in\R$ for $j\in J$. & Thm.~\ref{thm:interpolation}, p.~\pageref{thm:interpolation}\\[5pt]
$\widehat q,\ E,\ \ell_{\rm cap}$ & The nonnegative translate $\widehat q=q-\min_{\overline\Omega}q$ and the competitor $E=\min(\widehat q,\ell_{\rm cap}D_\Omega)$ with zero boundary values, for a sufficiently large positive integer $\ell_{\rm cap}$. & Prop.~\ref{prop:barriers}, p.~\pageref{prop:barriers}\\[5pt]
$r_N$ & The radius $\lceil(\sqrt{N+1}-1)/2\rceil$ of the square containing the interpolation slopes. & Thm.~\ref{thm:interpolation}, p.~\pageref{thm:interpolation}\\[5pt]
$O,\ \operatorname{Conf}_N(O)$ & A relatively compact open polygon used for semilinear parameter spaces, and its ordered configurations of $N$ distinct points. & App.~\ref{sec:C}, p.~\pageref{sec:C}\\[5pt]
$A_N,\ f_c$ & The finite common set of slopes for fixed $N,\Delta,O$, and $f_c(x)=\min_{m\in A_N}(c_m+m\cdot x)$. & Lemma~\ref{lem:slope-support}, p.~\pageref{lem:slope-support}\\[5pt]
$\kappa_m(f),\ c(P)$ & The canonical coefficient $\kappa_m(f)=\max_x(f(x)-m\cdot x)$, taken over $\overline\Omega$ for relaxation and over $\Delta$ for the finite model; $c(P)=\kappa_{A_N}(F_P)$. & Lemma~\ref{lem:completion}, p.~\pageref{lem:completion}\newline One-point update, p.~\pageref{app:one-point-relaxation}\\[5pt]
Strongly generic & Marks lie in distinct edge interiors, and the marked embedded type persists under sufficiently small independent motions of all sources. & Def.~\ref{def:strong-genericity}, p.~\pageref{def:strong-genericity}\\[5pt]
$G_N(\Delta)$ & The open, full-measure locus of strongly generic clouds on a fixed rational-slope polygon. & Thm.~\ref{thm:generic-locus}, p.~\pageref{thm:generic-locus}\\[5pt]
$\mathcal P_\Q(\Omega)$ & The countable family of compact convex polygons with rational vertex coordinates contained in $\Omega$ and with nonempty interior. & Lemma~\ref{lem:cofinality}, p.~\pageref{lem:cofinality}\\[5pt]
$\widehat G_N(\Delta)$ & $G_N(\Delta)$ together with configurations having at least one point outside the closed polygon. A cloud contained in the closed polygon with a mark on its boundary is excluded. & Def.~\ref{def:universal}, p.~\pageref{def:universal}\\[5pt]
$G_N^{\rm univ}(\Omega),\ G_N^{\mathcal E}$ & Universal genericity on the countable rational-coordinate family, and genericity along one fixed exhaustion. These countable intersections need not be open. & Def.~\ref{def:universal}, p.~\pageref{def:universal}\\[5pt]
\end{longtable}\endgroup

\clearpage
\subsection*{\texorpdfstring{\textcolor{Probability}{Sampling, affine maps, and the sandpile}}{Sampling, affine maps, and the sandpile}}\label{symbols:group5}
$M_N$ counts retained samples, $N$ counts all sources, and $h_N$ is the chosen lattice mesh.

\begingroup\small
\setlength{\tabcolsep}{4pt}\renewcommand{\arraystretch}{1.1}
\begin{longtable}{@{}P{.205\textwidth}P{.565\textwidth}P{.17\textwidth}@{}}
\toprule Symbol & Meaning & Reference\\\midrule
\endfirsthead
\toprule Symbol & Meaning & Reference\\\midrule
\endhead
\bottomrule\endfoot
$X_n,\ Y_r,\ \tau_r$ & Original i.i.d. samples, successive accepted samples $Y_r=X_{\tau_r}$ in a measurable set $K$, and their acceptance times. & Lemma~\ref{lem:accepted-stream}, p.~\pageref{lem:accepted-stream}\\[5pt]
$a,\ \mu|_K/a$ & The retained mass $a=\mu(K)>0$ and the conditional probability law on $K$. & Lemma~\ref{lem:accepted-stream}, p.~\pageref{lem:accepted-stream}\\[5pt]
$M_N,\ M_{N,j}$ & Numbers retained among the first $N$ samples in $K$, or in the exhaustion set $K_j$. For fixed $j$, $M_{N,j}/N\to a_j=\mu(K_j)$. & Thm.~\ref{thm:iid-boundary}, p.~\pageref{thm:iid-boundary}\newline Lemma~\ref{lem:accepted-stream}, p.~\pageref{lem:accepted-stream}\\[5pt]
$P_N^{(j)},\ \mu_j,\ G_j$ & The retained cloud $P_N\cap K_j$, the probability $\mu_j=\mu|_{K_j}/a_j$, and $G_j=\sqrt{a_j}\,F_{\mu_j,\Omega}$ with curvature $\mu|_{K_j}$. & Thm.~\ref{thm:iid-boundary}, p.~\pageref{thm:iid-boundary}\\[5pt]
$T(x)=Ax+b,\ T_\#\mu$ & An invertible affine map and the pushforward measure. Supporting slopes transform by the inverse transpose $A^{-T}$. & Thm.~\ref{thm:covariance}, p.~\pageref{thm:covariance}\\[5pt]
$J(\mu,\Omega)$ & The roof integral $\int_\Omega F_{\mu,\Omega}\,dx$, with affine factor $|\det A|^{3/2}$. & Thm.~\ref{thm:covariance}, p.~\pageref{thm:covariance}\\[5pt]
$\Omega_h,\ h,\ h_N$ & The sites $\Omega\cap h\Z^2$, a fixed-cloud mesh, and the mesh selected after the full cloud $P_N$. The selection ensures the approximation bound and $h_N<1/N$. & p.~\pageref{app:diagonal}\newline Thm.~\ref{thm:sandpile}, p.~\pageref{thm:sandpile}\\[5pt]
$p^h,\ \xi_{h,P}$ & A source rounding with $\|p^h-p\|_\infty\le h$ and the added-grain field $\sum_{p\in P}\mathbf1_{\{p^h\}}$. Proper rounding satisfies~\eqref{eq:proper-rounding}. & \eqref{eq:proper-rounding}, p.~\pageref{eq:proper-rounding}\\[5pt]
$H_{h,P},\ H_N,\ U_N$ & $H_{h,P}$ counts topplings. We write $H_N=H_{h_N,P_N}$ and $U_N=h_NH_N/\sqrt N$. & Thm.~\ref{thm:sandpile}, p.~\pageref{thm:sandpile}\newline Thm.~\ref{thm:deficit}, p.~\pageref{thm:deficit}\\[5pt]
$\Delta_h$ & The unscaled four-neighbor Laplacian: $\sum_{i=1}^2[V(v+he_i)+V(v-he_i)-2V(v)]$. Thus $\Delta_h\psi=h^2\Delta\psi+O(h^4)$. & \eqref{eq:lattice-laplacian}, p.~\pageref{eq:lattice-laplacian}\\[5pt]
$\varphi_N^\circ(v),\ d_N(v)$ & The number of grains at $v$ after stabilization, and its deficit $d_N(v)=3-\varphi_N^\circ(v)\in\{0,1,2,3\}$. & Thm.~\ref{thm:deficit}, p.~\pageref{thm:deficit}\\[5pt]
$D_N$ & The deficit measure $\frac{h_N}{\sqrt N}\sum_v d_N(v)\delta_v$, converging vaguely to $-\Delta F_{\mu,\Omega}$. & Thm.~\ref{thm:deficit}, p.~\pageref{thm:deficit}\\[5pt]
$\overline d_N(W)$ & The average deficit over lattice sites in a continuity window $W$. Its normalization is $\overline d_N(W)/(h_N\sqrt N)$. & \eqref{eq:window-deficit-normalization}, p.~\pageref{eq:window-deficit-normalization}\\[5pt]
$\rho,\ -\Delta F$ & The prescribed curvature density and the positive measure represented by $-\Delta F$. For smooth $F$, $\rho=\det(-D^2F)$ and $-\Delta F=\operatorname{tr}(-D^2F)\,dx$. & Thm.~\ref{thm:many-point}, p.~\pageref{thm:many-point}\newline Thm.~\ref{thm:deficit}, p.~\pageref{thm:deficit}\\[5pt]
\end{longtable}\endgroup

\clearpage

\section*{Acknowledgments}\addcontentsline{toc}{section}{Acknowledgments}
The work of E.L., H.S., and M.S. was supported by the Simons Foundation under grant SFI-MPS-T-Institutes-00007697 and by the Ministry of Education and Science of the Republic of Bulgaria under grant DO1-239/10.12.2024. E.L. acknowledges Cinvestav for sabbatical leave during which most of this work was performed, and the Institute for the Mathematical Sciences of the Americas (IMSA) at the University of Miami for hospitality and support during the preparation of this work.

We thank Conan Leung and Grigory Mikhalkin for helpful discussions on tropical and affine geometry.

\textbf{Use of generative AI tools.} The mathematical ideas and proof strategies are the authors' own. The authors used ChatGPT (OpenAI), Claude (Anthropic), and Gemini (Google) to assist with checking arguments and references, revising the exposition, and carrying out numerical calculations. The authors take responsibility for the manuscript.

\addcontentsline{toc}{section}{References}
\begingroup\small\raggedright
\nocite{*}
\bibliographystyle{amsalpha}
\bibliography{references}
\endgroup
\medskip\needspace{12\baselineskip}
\section*{Authors' addresses}
\small
\begin{minipage}[t]{.48\textwidth}\raggedright \textbf{Nikita Kalinin.} Guangdong Technion--Israel Institute of Technology, Shantou, China.\\
\href{mailto:nikita.kalinin@gtiit.edu.cn}{nikita.kalinin@gtiit.edu.cn}\end{minipage}\hfill
\begin{minipage}[t]{.48\textwidth}\raggedright \textbf{Ernesto Lupercio.} Departamento de Matem\'aticas, Cinvestav, Mexico City, Mexico; Institute of Mathematics and Informatics, Bulgarian Academy of Sciences, Sofia, Bulgaria.\\
\href{mailto:lupercio@math.cinvestav.mx}{lupercio@math.cinvestav.mx}; \href{mailto:ernesto.lupercio@cinvestav.mx}{ernesto.lupercio@cinvestav.mx}\end{minipage}

\medskip
\begin{minipage}[t]{.48\textwidth}\raggedright \textbf{Higinio Serrano.} Institute of Mathematics and Informatics, Bulgarian Academy of Sciences, Sofia, Bulgaria.\\
\href{mailto:hserrano@math.bas.bg}{hserrano@math.bas.bg}\end{minipage}\hfill
\begin{minipage}[t]{.48\textwidth}\raggedright \textbf{Mikhail Shkolnikov.} Institute of Mathematics and Informatics, Bulgarian Academy of Sciences, Sofia, Bulgaria.\\
\href{mailto:m.shkolnikov@math.bas.bg}{m.shkolnikov@math.bas.bg}\end{minipage}
\end{document}